\documentclass[graybox,12pt]{svmult}
\usepackage{
amsmath, amsfonts, amssymb, 
amscd
}
\pdfoutput=1 
\usepackage{mathptmx}       
\usepackage{helvet}         
\usepackage{courier}        
\usepackage{type1cm}        
 
\usepackage{enumerate} 
\usepackage{makeidx}         
\usepackage[pdftex]{graphicx}
\usepackage{multicol}        
\usepackage[bottom]{footmisc}

\makeindex             

\newtheorem{assumption}{Assumption}[section]
\newtheorem{thm}{Theorem}[section]
\newtheorem{prop}{Proposition}[section]
\newtheorem{lem}{Lemma}[section]
\newtheorem{ex}{Exercise}[section]

\newcommand{\sumtwo}[2]{\sum_{\substack{#1 \\ #2}}} 
\newcommand{\abs}[1]{\left| #1\right|}
\newcommand{\norm}[2]{\left| #1\right|_{#2}}
\renewcommand{\emptyset}{\varnothing}

\newcommand{\setof}[2]{\left\{#1 \,:\, #2 \right\}}

\def\1{\ifmmode {1\hskip -3pt \rm{I}}
\else {\hbox {$1\hskip -3pt \rm{I}$}}\fi} 

\newcommand{\bigo}[1]{{\mathrm O}\lb #1\rb }
\newcommand{\smo}[1]{{\mathrm o}\lb #1\rb }

\newcommand{\df}{\stackrel{\Delta}{=}}

\newcommand{\eqvs}{\stackrel{\sim}{=}}
\newcommand{\leqs}{\stackrel{<}{\sim}}            
\newcommand{\geqs}{\stackrel{>}{\sim}}             

\newcommand{\lb}{\left(}
\newcommand{\rb}{\right)}
\newcommand{\lbr}{\left\{}
\newcommand{\rbr}{\right\}}

\newcommand{\dd}{{\rm d}}

\newcommand{\be}{\begin{equation}}
\newcommand{\ee}{\end{equation}}

\newcommand{\step}[1]{S{\small TEP}\,#1}
\providecommand{\opt}[1]{C{\small ASE}\,#1.}

\newcommand{\calA}{\mathcal{A}}

\newcommand{\calE}{\mathcal{E}}
\newcommand{\calF}{\mathcal{F}}

\newcommand{\calH}{\mathcal{H}}

\newcommand{\calM}{\mathcal{M}}
\newcommand{\calN}{\mathcal{N}}

\newcommand{\calP}{\mathcal{P}}
\newcommand{\calQ}{\mathcal{Q}}

\newcommand{\calT}{\mathcal{T}}

\newcommand{\calY}{\mathcal{Y}}

\newcommand{\frh}{\mathfrak{h}}

\newcommand{\frl}{\mathfrak{l}}
\newcommand{\frm}{\mathfrak{m}}
\newcommand{\frn}{\mathfrak{n}}

\newcommand{\frs}{\mathfrak{s}}
\newcommand{\frt}{\mathfrak{t}}

\newcommand{\frv}{\mathfrak{v}}

\newcommand{\bbB}{\mathbb{B}}
\newcommand{\bbC}{\mathbb{C}}
\newcommand{\bbD}{\mathbb{D}}
\newcommand{\bbE}{\mathbb{E}}

\newcommand{\bbL}{\mathbb{L}}

\newcommand{\bbN}{\mathbb{N}}

\newcommand{\bbP}{\mathbb{P}}
\newcommand{\bbQ}{\mathbb{Q}}
\newcommand{\bbR}{\mathbb{R}}
\newcommand{\bbS}{\mathbb{S}}
\newcommand{\bbT}{\mathbb{T}}

\newcommand{\bbV}{\mathbb{V}}

\newcommand{\bbZ}{\mathbb{Z}}

\newcommand{\sfb}{\mathsf b}
\newcommand{\sfc}{{\sf c}}
\newcommand{\sfe}{{\sf e}}
\newcommand{\sff}{{\sf f}}

\newcommand{\sfn}{{\sf n}}

\newcommand{\sfq}{{\sf q}}
\newcommand{\sfr}{{\sf r}}
\newcommand{\sfs}{{\sf s}}
\newcommand{\sft}{{\sf t}}
\newcommand{\sfu}{{\sf u}}
\newcommand{\sfv}{{\sf v}}
\newcommand{\sfw}{{\sf w}}
\newcommand{\sfx}{{\sf x}}
\newcommand{\sfy}{{\sf y}}
\newcommand{\sfz}{{\sf z}}

\newcommand{\sfA}{{\sf A}}

\newcommand{\sfC}{{\sf C}}

\newcommand{\sfE}{{\sf E}}

\newcommand{\sfP}{{\sf P}}

\newcommand{\sfS}{{\sf S}}
\newcommand{\sfT}{{\sf T}}
\newcommand{\sfU}{{\sf U}}
\newcommand{\sfV}{{\sf V}}
\newcommand{\sfW}{{\sf W}}

\newcommand{\sfX}{{\sf X}}
\newcommand{\sfY}{{\sf Y}}
\newcommand{\sfZ}{{\sf Z}}

\newcommand{\Zd}{\bbZ^d}
\newcommand{\Rd}{\bbR^d}

\newcommand{\cHm}[1]{\calH^-_{#1}}
\newcommand{\cHp}[1]{\calH^+_{#1}}

\newcommand{\sffo}{{\sf f}^\omega}
\newcommand{\sfso}{{\sf s}^\omega}
\newcommand{\sfto}{{\sf t}^\omega}

\newcommand{\sffof}[1]{{\sf f}^{\theta_{#1}\omega}}
\newcommand{\sfsof}[1]{{\sf s}^{\theta_{#1}\omega}}
\newcommand{\sftof}[1]{{\sf t}^{\theta_{#1}\omega}}

\newcommand{\tq}[1]{\sft^\omega ({#1})}
\newcommand{\fq}[1]{\sff^\omega ({#1})}
\newcommand{\ta}[1]{{\sft}({#1})}
\newcommand{\fa}[1]{{\sff}({#1})}

\newcommand{\rra}[1]{\sfr_{#1}}
\newcommand{\rrq}[1]{\sfr^\omega_{#1}}
\newcommand{\rrqof}[2]{{\sfr }_{#1}^{\theta_{#2}\omega}}

\newcommand{\Var}{\bbV{\rm ar}}
\newcommand{\Cov}{\bbC{\rm ov}}

\newcommand{\Zn}{Z_n}

\newcommand{\Plf}[1]{\bbP_\lambda^{#1}}
\newcommand{\Pnf}[1]{\bbP_n^{#1}}
\newcommand{\Enf}[1]{\bbE_n^{#1}}

\newcommand{\Pd}{\sfP_d}
\newcommand{\Ed}{\sfE_d}
\newcommand{\Wd}{\sfW_d}
\newcommand{\Wdn}[1]{\sfW_d^{#1}}

\newcommand{\tl}{\tau_\lambda}
\newcommand{\tlo}{\tau_\lambda^\sfq}

\newcommand{\Kl}{{\mathbf K}_\lambda}
\newcommand{\Knot}{{\mathbf K}_0}

\newcommand{\pKnot}{\partial{\mathbf K}_0}
\newcommand{\pKl}{\partial {\mathbf K}_\lambda}

\newcommand{\bfK}{{\mathbf K}}
\newcommand{\pbfK}{\partial {\mathbf K}}

\newcommand{\UlK}[1]{{\mathbf U}_\lambda^{#1}}

\newcommand{\Wdo}{\sfW_d^\omega}
\newcommand{\Wdon}[1]{\sfW_d^{#1 , \omega}}
\newcommand{\Zno}{Z_n^\omega}

\newcommand{\Pnfo}[1]{\bbP_n^{#1 , \omega}}
\newcommand{\Enfo}[1]{\bbE_n^{#1 , \omega}}

\newcommand{\Klo}{{\mathbf K}_\lambda^\sfq}
\newcommand{\Knoto}{{\mathbf K}_0^\sfq}

\newcommand{\sfVo}{{\sf V}^\omega}

\begin{document}

\title*{Multidimensional Random Polymers : A Renewal Approach
}
\author{Dmitry Ioffe}
\institute{Dmitry Ioffe  \at Faculty of IE\& M, Technion,  Haifa 32000, 
Israel.  \email{ieioffe@ie.technion.ac.il}\\ 
Supported by the ISF grant 817/09 and by  the Meitner Humboldt Research Award.
}
%
%
\maketitle

\abstract{In these lecture notes,  
which are based on the mini-course given at 
2013 Prague School on Mathematical Statistical Physics, 
we discuss ballistic phase of quenched and annealed 
stretched polymers in random environment on $\bbZ^d$  with an emphasis on the 
natural renormalized renewal structures which appear in such models. In the ballistic
 regime an irreducible decomposition of typical polymers leads to an effective 
 random walk  reinterpretation of the latter. In the annealed case 
 the Ornstein-Zernike theory based on this approach paves
 the way to 
 an essentially complete control on the level of local limit results 
 and invariance principles.  
 In the quenched case, the renewal structure maps the 
 model  of stretched polymers 
 into an effective model of directed polymers. As a result one is able to 
 use techniques and ideas developed in the context of directed polymers in order 
 to address issues like strong disorder in low dimensions and weak disorder in 
 higher dimensions. Among the topics addressed: Thermodynamics of quenched and
 annealed models, multi-dimensional renewal theory (under Cramer's condition), 
 renormalization and effective random walk structure of annealed polymers, 
 very weak disorder in dimensions $d\geq 4$ and strong disorder in dimensions
 $d=1,2$. }

\section{Introduction}
Mathematical and probabilistic developments presented here draw inspiration 
from statistical mechanics of stretched polymers, see for instance 
\cite{deGennes,RubinsteinColby}. 

Polymers chains to be discussed  in these 
lecture notes are modeled by paths of finite range random walks on $\bbZ^d$. 
We shall always assume that the underlying random walk distribution has 
zero mean. The word {\em stretched} alludes to the situation  when  the 
end-point of a polymer is pulled by an external force, or in the random walk 
terminology, by a drift. In the case of random walks this leads to a ballistic 
behaviour with limiting spatial extension described in terms of the usual law
of large numbers (LLN) for independent sums. Central limit theorem (CLT) and 
large deviations (LD) also hold. 

Polymer measures below are non-Markovian objects (see Remark~\ref{IN-NonM}), 
which gives rise to a  rich  morphology. We shall distinguish between 
ballistic and sub-ballistic phases and between quenched and
annealed polymers. Quenched polymers correspond to pulled random walks in random 
potentials. Their annealed counterparts correspond to pulled random walks in a 
deterministic attractive self-interaction potentials. 

Two main themes are the impact of the drift, that is when the model in question, annealed or 
quenched, becomes ballistic, and the impact of disorder, that is whether or not quenched and 
annealed models behave similarly. 

It is instructive to compare models of stretched polymers with those of directed polymers 
\cite{CSY2004}. In the latter case sub-ballistic to ballistic transition is not an issue. 
Furthermore, in the stretched case polymers can bend and return to the same vertices, which makes
even the annealed model to be highly non-trivial (in the directed case the annealed model is a
usual random walk). On the other hand it is unlikely that a study of stretched polymers will shed
light on  questions which are open  in the directed context. 
\subsection{Class of models.}
{\bf Underlying random walk.}
Consider random walk on $\bbZ^d$  with an  irreducible finite range step distribution. 
We  use the notation $\sfP_d$ both  for the random walk path measure 
and for the distribution of individual steps. 
For 
convenience we shall assume that nearest neighbour steps $\pm \sfe_k$ are permitted, 
\be 
\label{eq:L2-sfe}
\sfP_d \lb \pm \sfe_k \rb > 0
\ee
 The size of the range is denoted $R$: $\sfP_d \lb \sfX = \sfx \rb >0 
 \Rightarrow \abs{\sfx}\leq R$. Without loss of generality we shall assume that
 $\sfE_d \sfX = 0$.
\smallskip 

\noindent
{\bf Random environment.} The random environment is modeled by a collection 
$\lbr \sfVo_\sfx\rbr_{\sfx\in\bbZ^d}$ i.i.d {\em non-negative} random variables. The notation
$\calQ$ and $\calE$ are reserved for the corresponding product probability measure and
the corresponding expectation. We shall assume: 
\smallskip 

\noindent
{\bf (A1 )} $\sfVo$ is non-trivial and $0\in {\rm supp}\lb \sfVo\rb$ .
\smallskip 

\noindent
{\bf (A2 )} $\calQ\lb \sfVo < \infty \rb > p_c (\bbZ^d )$, where $p_c$ is the critical 
Bernoulli site percolation probability. 
\smallskip 

\noindent
{\bf Polymers and polymer weights.} Polymers 
$\gamma = (\gamma_0, \dots ,\gamma_n )$ are paths of the underlying random walk. 
For  each polymer $\gamma$ we define $\abs{\gamma} = n$ as the number of steps, and 
$\sfX (\gamma ) = \gamma_n -\gamma_0$ as the displacement along the polymer. 

The are two type of weights we associate with polymers: quenched random weights
\be 
\label{eq:IN-qweight}
\Wdo (\gamma ) = {\rm exp}\lbr -\beta \sum_{i=1}^{\abs{\gamma}}  \sfVo_{\gamma_i}\rbr \, 
\Pd (\gamma ) , 
\ee
and annealed weights 
\be 
\label{eq:IN-aweight}
\Wd (\gamma ) = \calE \lb \Wdo (\gamma ) \rb  
= {\rm e}^{- \Phi_\beta (\gamma )}\, 
\Pd (\gamma ) , 
\ee
where the self-interacting potential 
\be 
\label{eq:Phi}
 \Phi_\beta  (\gamma ) = \sum_\sfx \phi_\beta  (\ell_{\gamma }(\sfx )). 
\ee
Above  $\ell_\gamma (\sfx )$ is the local time of $\gamma$ at $x$;
\be 
\label{eq:loc-time}
 \ell_\gamma (x) = \sum_{i=1}^n \1_{\lbr \gamma_i = \sfx\rbr}, 
\ee
and $\phi_\beta$ is given by: 
\be 
\label{eq:IN-phi-beta}
\phi_\beta (\ell ) = - \log\calE\lb {\rm e}^{-\beta\ell\,  \sfVo }\rb .
\ee
The inverse temperature $\beta >0$ modulates the strength of disorder. 
\smallskip 

\noindent 
{\bf Pulling force, partition functions and probability distributions.} 
For $h\in\bbR^d$ we shall consider quenched and annealed partition functions 
\be 
\label{IN-pf}
\Zno (h ) = \sum_{\abs{\gamma} =n}  {\rm e}^{h\cdot \sfX (\gamma )} \Wdo (\gamma )\quad
{\rm and}\quad \Zn (h ) = \calE (\Zno (h ) ) = 
\sum_{\abs{\gamma} =n}  {\rm e}^{h\cdot \sfX (\gamma )} \Wd (\gamma ) , 
\ee
and the corresponding probability distributions, 
\be 
\label{IN-pd} 
\Pnfo{h}(\gamma ) = \frac{1}{\Zno (h )} {\rm e}^{h\cdot \sfX (\gamma )} \Wdo (\gamma )
\quad{\rm and}\quad 
\Pnf{h}(\gamma ) = \frac{1}{\Zn (h )} {\rm e}^{h\cdot \sfX (\gamma )} \Wd (\gamma ). 
\ee
\begin{remark}
\label{IN-NonM} 
Annealed measures $\Pnf{h}$ are non-Markovian. Quenched measures $\Pnfo{h}$ are also 
non-Markovian in the sense that in general 
$\Pnfo{h}$ is not a marginal of $\bbP_{m}^{h, \omega}$ for $m>n$.  
\end{remark}
\subsection{Morphology.} 
We shall distinguish between ballistic and sub-ballistic behaviour of quenched and
annealed polymers \eqref{IN-pd} and between strong an  weak impact of disorder on
the properties of quenched polymers (as compared to the annealed ones). 
\smallskip 

\noindent 
{\bf Ballistic phase.} 
For the purpose of these lecture notes, 
 let us say that a self-interacting random walk (or polymer) is ballistic if
there exists 
$\delta >0$ and 
a vector $\sfv\neq 0$
such that
\be 
\label{eq:L2-ballistic}
\lim_{n\to\infty} \Pnf{h} \lb \abs{\sfX}\leq \delta n\rb = 0\quad{\rm and}\quad 
\lim_{n\to\infty}\frac{1}{n}\bbE_n^h\,  \sfX (\gamma )  = \sfv .
\ee
The model is said to be sub-ballistic, if 
 \be 
\label{eq:L2-subballistic}
\lim_{n\to\infty}\frac{1}{n}\bbE_n^h \abs{\sfX (\gamma )}  = 0 . 
\ee
At this stage it is  unclear whether 
 there are  models which comply neither with \eqref{eq:L2-ballistic} nor 
with \eqref{eq:L2-subballistic}. It is the content of 
Theorem~\ref{thm:attractive-morp} below that for the annealed models the above
dichotomy  always holds. 

Similarly, the quenched model is said to be in the ballistic, respectively 
sub-ballistic, phase if $\calQ$-a.s 
\be 
\label{eq:L2-ballistic-q}
\lim_{n\to\infty} \Pnfo{h} \lb \abs{\sfX}\leq \delta n\rb = 0\quad{\rm and}\quad 
\lim_{n\to\infty}\frac{1}{n}\Enfo{h}\,  \sfX (\gamma )  = \sfv , 
\ee
and, respectively, 
\be 
\label{eq:L2-subballistic-q}
\lim_{n\to\infty}\frac{1}{n}\Enfo{h} \abs{\sfX (\gamma )}  = 0 . 
\ee
For quenched models it is in general open question whether the limiting spatial
extension (second limit in \eqref{eq:L2-ballistic-q}) always exists. 
See Theorem~\ref{thm:quenched-morp}  below for a precise statement.

Ballistic and sub-ballistic phases correspond to very different patterns of behaviour. 
We focus here on the ballistic phase. In a sense  sub-ballistic behaviour is more intricate than
the ballistic one.  
In the continuous context (Brownian motion) results about sub-ballistic phase are summarized 
in \cite{Sznitman}. The theory (so called {\em enlargement of obstacles}) was adjusted to
random walks on $\bbZ^d$ in \cite{Antal95}. 
\smallskip 

\noindent
{\bf Strength of disorder.}  For each value of the pulling 
force $h$ and the interaction $\beta$ 
 strength of disorder may be quantified on several 
levels: 
\smallskip 

{\bf L1.}\quad  $\calQ$-a.s $\limsup_{n\to\infty} \frac{1}{n}\log\frac{\Zno (h )}{\Zn (h)} < 0$.

\noindent
Since by Assumption {\bf (A1)} the annealed potential 
 $\phi_\beta$ in 
\eqref{eq:IN-phi-beta} satisfies $\lim_{\beta\to\infty} \frac{\phi_{\beta}}{{\beta}} = 0$, 
it is not difficult to see that, at least in the case when 
$\lbr \sfx~:~ \sfV^\omega_\sfx = 0\rbr$ does not percolate, 
{\bf L1.} holds in any dimension whenever the strength of 
interaction $\beta$ 
and the pulling force $h$ are large enough. 

Furthermore, as we shall see 
in Section~\ref{sec:Strong}  (and as it was originally proved in 
\cite {Zygouras-StrongDisorder}) the disorder is strong in the 
 sense of {\bf L1.} in lower dimensions 
 $d=2,3$  for any $\beta > 0$ 
 provided  that the annealed polymer is in (the interiour of) the ballistic phase.
\smallskip 

{\bf L2.}\quad  $\calQ$-a.s $\lim_{n\to\infty} \frac{\Zno (h )}{\Zn (h)} = 0$. 

This is presumably always the case when the quenched model is in 
sub-ballistic phase. The case
$h=0$ is worked out in great detail \cite{Sznitman, Antal95}. 
\begin{remark}
\label{IN-Rem-ballistic}
A characterization of annealed and quenched sets of sub-critical drifts; $\Knot$ and 
$\Knoto$ is given in \eqref{eq:L2-critical-a} and \eqref{eq:L2-critical-a-o} below. 
It always holds that $\Knot\subseteq \Knoto$. 
The inclusion 
is strict in any dimension for $\beta$ large enough, and it is presumably strict 
in dimensions $d=2,3$ for any $\beta >0$. 
On the other hand, it seems to be  an open question whether in higher dimensions
 $d>3$ the two sets of sub-critical drifts coincide at sufficiently small $\beta$. 
\end{remark}

{\bf L3.} Typical polymers under $\Pnf{h}$ and $\Pnfo{h}$ have very different 
properties. 
\smallskip 

\noindent 
 Ballistic phase of annealed polymers in dimensions 
$d\geq 2$ is by now completely 
understood, and we expose the core of the corresponding (Ornstein-Zernike) theory
 developed in \cite{IoffeVelenik-Annealed,IV-Critical} in Section~\ref{Asec:renewal}. 
In dimensions $d\geq 4$, the disorder happens to be weak in the sense of 
any of {\bf L1-L3} in the following regime (which we shall call 
{\em very weak disorder}): Fix $h\neq 0$ and then take 
 $\calQ\lb \sfVo = \infty \rb$ and 
$\beta$ in \eqref{eq:IN-qweight}
to be sufficiently small. These results~\cite{Flury,Zygouras,IV-Crossing,IV-CLT} 
are explained in Section~\ref{sec:Weak}. 
\smallskip 

\noindent
{\bf Zero drift case.} 
At $h=0$  properties of both annealed and quenched measures were described in depth
in \cite{Sznitman} and references therein, 
following an earlier analysis of Wiener sausage in 
\cite{DV1,DV2}. 
This is {\em not} the case we consider here. However, it is instructive 
to keep in mind what happens if there is no pulling force, and, accordingly,  we give a brief 
heuristic sketch. To fix ideas consider the case of pure traps 
$p =\calQ \lb \sfVo = 0\rb = 1 -\calQ \lb \sfVo = \infty\rb$. If $1-p$ is small, 
then $\lbr \sfx : \sfVo_\sfx = 0\rbr$ percolates, 
and the model is non-trivial. 
Let us start with a quenched case. Let $B_r$ be a lattice box
 $B_r =\lbr \sfx :\abs{\sfx}_1\leq r\rbr$ and $B_r (\sfx ) = \sfx + B_r$. We say that there is an
 $(R, r )$-clearing if 
 \[
  \exists\, \sfx \in B_R\ \text{such that $\sfVo_\sfy = 0$ for all $\sfy\in B_r (\sfx )$} .
 \]
The probability 
\[
 \calQ\lb \text{there is a $(R, r)$ clearing}\rb \approx  1 - \lb 1- p^{c_1 r^d}\rb^{c_2 R^d/r^d } 
 \approx 1 - {\rm e}^{-c_3p^{c_1 r^d}\frac{R^d}{r^d}}.  
\]
Up to leading terms this is non-negligible if $p^{c_1r^d}R^d\approx {\rm const}$, or if 
$r\approx \lb \log R\rb^{1/d}$. On the other hand, a probability that a random walk will 
go ballistically to a box (clearing) $B_r (\sfx )$ at distance of order $R$ from the origin 
is of order ${\rm e}^{-c_4 R}$, and the probability that afterwards it will
spend around $n$ units of time in $B_r (\sfx )$ is $\approx {\rm e}^{-c_5 n/r^2}$. We, therefore 
need to find an optimal balance between $R$ and $n/r^2\approx n/(\log R)^{2/d}$ terms, which gives, 
again up to leading terms, $R\approx n/(\log n)^{2/d}$. This suggests both a survival pattern for typical
quenched polymer (see Figure~\ref{fig:hzero}), and an asymptotic relation for the quenched partition function
\be 
\label{eq:IN-qpf}
\log\Zno \approx -\frac{n}{(\log n )^{2/d}} .
\ee
\begin{figure}[t]
\scalebox{.35}{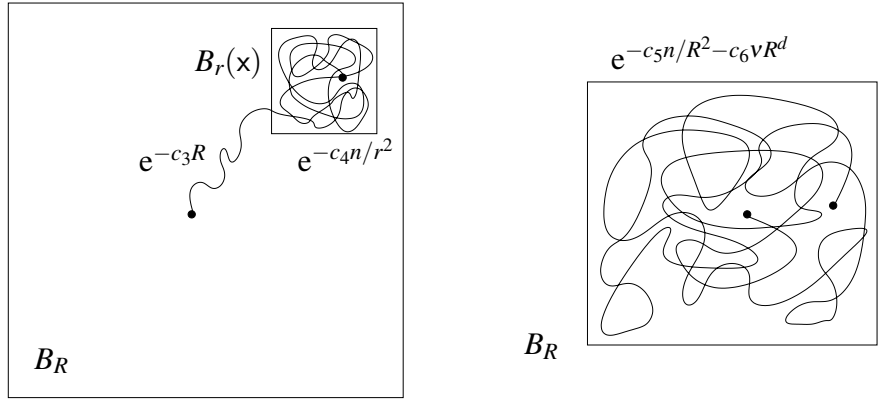}
%
%
\caption{On the left: A survival pattern for an  $n$-step quenched polymer with 
$R\approx n/(\log n)^{2/d}$.\newline 
On the right: $n$-step annealed polymer in $B_R$ with $R\approx n^{1/(d+2)}$.}
\label{fig:hzero}       
\end{figure}
As far as the annealed model is considered for $\ell\geq 1$ define 
as before $\phi_\beta (\ell ) = -\log \calE \lb{\rm e}^{-\beta\ell \sfV_0 }\rb = -\log p\df\nu$. 
Consider random walk which stays all $n$ units of time inside $B_R$. The probabilistic price for the 
latter is $\approx {\rm e}^{-c_6 n/ R^2}$. On the other hand, the self-interaction price is 
$\approx {\rm e}^{-c_7 \nu R^d}$. Choosing optimal balance leads to $R\approx n^{1/(d+2)}$.  
This suggests a   behaviour 
pattern for typical annealed polymers (see Figure~\ref{fig:hzero}), which is very different from the 
survival  pattern for typical quenched polymer as discussed above. This also suggests the 
following asymptotics for the annealed partition function:
\be 
\label{eq:IN-apf}
\log\Zn \approx - n^{d/(d+2 )} .
\ee
The above discussion indicates that in the zero drift case the disorder is strong on levels {\bf L2, L3}, 
but not on {\bf L1.} 
\smallskip 

\noindent 
{\bf Outline of the notes.} We do not attempt to give a comprehensive survey of the existing results
on the subject. Neither the notes are self-contained, in many instances below we shall refer to 
the literature for more details on the corresponding proofs. The emphasis is on the exposition of
the renewal structure behind stretched polymers in the ballistic regime, and how this might help to
explore and understand various phenomena in question. 

Section~\ref{sec:TH} is devoted to the thermodynamics of annealed and quenched polymers, namely
to the facts which can be deduced from sub-additivity arguments and large deviation principles.

Multidimensional renewal theory (under assumption of exponential tails) 
is discussed in detail  in Section~\ref{Asec:renewal}. In Subsection~\ref{ssec:BalPhaseA} we explain 
renormalization procedures which lead to a reformulation of annealed models in this renewal context, 
which is the core of the Ornstein-Zernike theory of the latter. 
As a result we derive very sharp and essentially complete description of the ballistic phase 
of the annealed polymers on the level of  invariance principles and 
local limit asymptotics on all deviation scales. 

In Section~\ref{sec:Weak} we explain why the annealed renewal structure persists for quenched models
in the regime of very weak disorder in dimensions $d\geq 2$. More precisely, it happens that in the 
latter case the disorder is weak on all three levels ${\bf L1 - L3.}$. 

In Section~\ref{sec:Strong} we explain how to check  that the disorder is always strong 
already on  level {\bf L1} in 
dimensions $d=2, 3$. A more or less complete argument is given only in two dimensions. 

To facilitate references and the reading some of the back ground material on convex geometry and
large deviations is collected in the Appendix. 
\smallskip 

\noindent
{\bf Notation conventions.} Values of positive constants $c, \nu, c_1, \nu_1, c_2,\nu_2,  \dots $ 
may change between different Sections. 

In the sequel we shall use the following notation for asymptotic relations: 
Given a  set of indices
$\calA$ and two positive sequences $\lbr a_\alpha , b_\alpha \rbr_{\alpha\in\calA}$, 
we say that 
\begin{itemize}
 \item 
$a_\alpha \leqs
b_\alpha$ {\em uniformly in $\alpha\in \calA$} if there exists a constant $c>0$ such 
that $a_\alpha \leq  c b_\alpha$ for all $\alpha\in\calA$ . 
\item
We shall use
$a_\alpha \cong b_\alpha$ if both $a_\alpha\leqs b_\alpha$ and $a_\alpha\geqs b_\alpha$ hold.
\end{itemize}
For $1\leq p\leq \infty$, the $\ell_p$-norms on 
$\bbR^d$ are denoted
\[
\norm{x}{p} = \Bigl( \sum_{i=1}^d  \abs{x_i}^p \Bigr)^{\frac1{p}} .
\]
The default notation is for the Euclidean norm $\abs{\cdot} = 
\norm{\cdot}{2}$.

If not explicitly stated otherwise paths $\gamma = \lb \gamma_0,\dots , \gamma_n\rb$ are assumed to have 
their starting point at the origin; $\gamma_0 = 0$. A concatenation $\gamma\circ\eta$ of two paths 
$\gamma =\lb \gamma_0,\dots , \gamma_n\rb$ and $\eta = \lb \eta_0, \dots , \eta_m\rb$ is the path 
\[
  \gamma\circ\eta = \lb \gamma_0, \dots , \gamma_n, \gamma_n +\eta_1, \dots , \gamma_n +\eta_m\rb .                                                                           
 \]                                                                                                 
A union of two paths $\gamma\cup\eta$, with end-points at the origin or not, is a subset of $\bbZ^d$ with 
multiplicities counted. In particular, local times satisfy $\ell_{\gamma\cup\eta} (\sfx ) = 
\ell_\gamma (\sfx ) +\ell_\eta (\sfx )$.

\section{Thermodynamics of Annealed and Quenched Models.} 
\label{sec:TH} 
In the sequel we shall employ the following notation for families of polymers:
\be 
\label{eq:TH-FamiliesPolymers}
\calP_\sfx = \lbr\gamma ~:~\sfX (\gamma) = \sfx\rbr
\quad 
\calP_n = \lbr\gamma ~:~\abs{\gamma} = n\rbr
\quad 
{\rm and}\quad 
\calP_{\sfx , n} = \calP_\sfx\cap \calP_n
\ee
{\bf Conjugate ensembles.}  
Let $\lambda\geq 0$. 
Consider
\be 
\label{eq:Glx}
G_\lambda^\omega  (\sfx ) = 
\sum_{\sfX (\gamma ) = \sfx} {\rm e}^{-\lambda\abs{\gamma}} \Wdo (\gamma ) 
\quad{\rm and}\quad 
G_\lambda   (\sfx ) = \calE\lb G_\lambda^\omega  (\sfx )\rb = 
\sum_{\sfX (\gamma ) = \sfx} {\rm e}^{-\lambda\abs{\gamma}} \Wd (\gamma ).
\ee
{\bf Free energy and inverse correlation length.} 
One would like to define quenched and annealed free energies via:
\be 
\label{eq:lambda-h}
\lambda^\sfq (h ) = \lim_{n\to\infty}\frac1n\log \Zno (h )\quad{\rm and}\quad 
\lambda (h ) = \lim_{n\to\infty}\frac1n\log Z_n (h ). 
\ee
Similarly one would like to define 
and the inverse correlation lengths: For $\sfx\in \bbR^d$ set 
\be
\label{eq:L2-tl}
\tlo (\sfx) = - \lim_{r\to\infty} \frac{1}{r}\log G_\lambda^\omega 
\lb \lfloor r\sfx\rfloor\rb\quad {\rm and}\quad 
\tl (\sfx) = - \lim_{r\to\infty} \frac{1}{r}\log G_\lambda \lb \lfloor r\sfx\rfloor\rb .
\ee
Depending on the context other names for  $\lambda (h )$ are {\em connectivity constant} and 
{\em log-moment generating function}, and for  $\tl (\sfx )$  are {\em Lyapunov exponent} and, 
for some models in two dimensions, {\em surface tension}. 

By {\em Thermodynamics} we mean here 
statements about existence of limits in \eqref{eq:lambda-h} 
and \eqref{eq:L2-tl}, and 
their relation to {\em Large Deviation} asymptotics under quenched 
and annealed polymer measures \eqref{IN-pd}. 
Facts about Thermodynamics of annealed and quenched models are collected 
in Theorem~\ref{thm:attractive-morp} and Theorem~\ref{thm:quenched-morp} below, 
and, accordingly, 
 discussed in some  detail in Subsections~\ref{sub:TH-annealed} and 
\ref{sub:TH-quenched}.

\subsection{Annealed Models in dimensions $d\geq 2$.}
\label{sub:TH-annealed}
\begin{thm}
\label{thm:attractive-morp} 
{\bf A.}  The free energy $\lambda$ is well defined, non-negative  and convex on $\bbR^d$. 
Furthermore, 
\be 
\label{eq:L2-lambda0-a}
0  = \min_h \lambda (h) = \lambda (0). 
\ee
The set 
\be
\label{eq:L2-critical-a}
{\bf K}_{0}\df \lbr h~:~ \lambda (h )= 0\rbr 
\ee
is  a compact  convex set  with a non-empty interiour. \newline
{\bf B.} The inverse correlation length $\tl$ is well defined for any $\lambda\geq 0$, 
and it can be identified as the support function of the compact convex set 
\be
\label{eq:TH-Kl}
{\bf K}_\lambda \df \lbr h~:~ \lambda (h )\leq \lambda \rbr .
\ee
Define 
\be 
\label{eq:L3-LD-rate-f}
 I (\sfv ) = \sup_{h}\lbr h\cdot \sfv - \lambda (h )\rbr = 
 \sup_{\lambda}\lbr \tl (\sfv ) - \lambda\rbr  .
\ee
{\bf C.} 
For any $h\in\bbR^d$ the family of polymer measures $\Pnf{h}$
satisfies LD principle with the rate function
\be 
\label{eq:L3-LD-rate-fh}
 I_h (\sfv )\df \sup_f \lbr f\cdot \sfv - \lb  \lambda (f +h ) -\lambda (h )\rb 
 \rbr = I(\sfv)  - \lb 
h\cdot\sfv - 
\lambda (h )\rb. 
\ee
{\bf D.} For $h\in {\rm int}\lb  {\bf K}_{0} \rb $  the model is sub-ballistic,  
whereas for any    
$h\not\in {\bf K}_{0} $ the model is ballistic. 
\newline
{\bf E.} Furthermore, at critical drifts $h\in\partial {\bf K}_0$  
the model is still ballistic. 
In other words, the ballistic to sub-ballistic transition is always of the first order 
in dimensions $d\geq 2$. 
\end{thm}
 Proofs of Parts A-C  and of Part~D for sub-critical drifts ($h\in {\rm int}\lb  {\bf K}_{0} \rb $) 
 of 
Theorem~\ref{thm:attractive-morp} are based
on sub-additivity arguments.  Parts D (namely existence 
of limiting spatial extension $\sfv$ in \eqref{eq:L2-ballistic}  
 for super-critical drifts $h\not\in {\bf K}_{0} $) and E require a more refined 
multidimensional renewal analysis based on Ornstein-Zernike theory.
\smallskip 

\noindent 
{\bf Sub-additivity.} The following result is due to Hammersley~\cite{Ham62}:
\begin{prop}
\label{prop:L2-Hammersley-sabad}
Let $\lbr a_n\rbr$ and $\lbr b_n\rbr$ be two sequences such 
that:
\smallskip 

\noindent
(a) For all $m,  n$, $a_{n+m} \leq a_n + a_m + b_{n+m}$.

\noindent
(b)  The sequence $b_n$ is non-decreasing and 
\be 
\label{eq:L2-Hammersley-sabad}
\sum_{n}\frac{b_n}{n (n+1 )} <\infty .
\ee
Then, there exists the limit 
\be 
\label{eq:L2-subad}
\xi\df \lim_{n\to\infty}\frac{a_n}{n} \ \  \text{and}\ \ \frac{a_n}{n}\geq \xi 
+ \frac{b_n}{n} - 4 \sum_{k=2n}^\infty \frac{b_k}{k (k+1)}
\ee
\end{prop}
Note that $\xi\df \liminf_{n\to\infty}\frac{a_n}{n}$ is always defined. 
The usual sub-additivity statement is for $b_n\equiv 0$. 
In this case 
$\xi = \inf_n\frac{a_n}{n}$. 
The proof (in the case $b_n\equiv 0$)  is 
straightforward: 
Indeed, 
 iterating on the sub-additive property, we infer that for any $n\geq m$ and any $k$ 
(and $N = kn +m$),  
\[
\frac{a_N}{N}  \leq \frac{k a_n}{kn +m} + \frac{a_m}{kn +m}
\]
Therefore, for any $n$ fixed
\[
 \limsup_{N\to\infty} \frac{a_N}{N} \leq \frac{a_n}{n } .
\]
\begin{remark}
\label{rem:L2-subad}
Note that \eqref{eq:L2-subad} implies that 
$\xi <\infty$, however the case $\xi = -\infty$ is not excluded by the argument. 
Note also that even if $\xi > -\infty$, no upper bounds (apart from $\lim\frac{a_n}{n } =\xi$) on 
$\frac{a_n}{n }$ 
are claimed.  In other words the sub-additivity argument above does not give information 
on the speed of convergence.
\end{remark}
\smallskip 

\noindent
{\bf Attractivity of the interaction.} 
The interaction $\phi_\beta$ in \eqref{eq:IN-phi-beta} is  attractive 
in the sense that 
\be 
\label{eq:attractive}
\phi_\beta (\ell + m ) \leq \phi_\beta (\ell ) + \phi_\beta (m ) .
\ee
Indeed, \eqref{eq:attractive}  follows from 
positive association of probability measures on $\bbR$. Namely, if $\sfX\in\bbR$ is a 
random variable, 
and $f,g$ two bounded  functions on $\bbR$ which are either both non-increasing or 
both non-decreasing, 
then 
\be 
\label{eq:L2-FKG}
 \bbE f(\sfX )g (\sfX ) \geq \bbE f(\sfX ) \bbE g(\sfX ) .
\ee
The universal validity of the latter inequality  is related 
to total ordering of $\bbR$. If $\sfY$ is an i.i.d. copy of $\sfX$, then 
\[
 \lb f (\sfX ) - f (\sfY )\rb \lb g (\sfX ) - g (\sfY )\rb  \geq 0.
\]
 Taking expectation we deduce \eqref{eq:L2-FKG}. 
\begin{ex}
\label{ex:L2.1}
 Show that if $\phi$ is attractive, then 
that for any $h\in\bbR^d$ and for any $\lambda$ and  $\sfx , \sfy\in\Zd$
\be 
\label{eq:L2-attractive-sup}
Z_{n+m} (h) \geq  Z_n (h )Z_m (h). 
\ee
Note that due to a possible over-counting such line of reasoning does not imply 
that $G_\lambda (\sfx +\sfy ) \geq G_\lambda (\sfx )G_\lambda (\sfy )$. 
\end{ex}
{\bf Part~A of Theorem~\ref{thm:attractive-morp}.} 
  Since the underlying random walk has finite range $R$, 
 $\Zn (h ) \leq {\rm e}^{R n\abs{h}}$. 
 On the other hand, since we assumed that $\Pd (\sfe_1 ) >1$, 
 \[
  \Zn (h ) \geq \lb {\rm e}^{h\cdot \sfe_1 - \phi_\beta (1)}\Pd (\sfe_1  >1 )\rb^n .
 \]
 Consequently, 
 $\lambda_n (h )\df \frac{1}{n}\log \Zn (\cdot )$
 is a sequence of convex (by H\"{o}lder's inequality) locally uniformly 
 bounded 
 functions.
 By Jensen's inequality $\lambda_n (0 ) = \min\lambda_n$. 
 For each $h$ fixed, $\log \Zn ( h )$ is, 
 by \eqref{eq:L2-attractive-sup} super-additive in $n$. By 
 Proposition~\ref{prop:L2-Hammersley-sabad}, $\lambda (h )= \lim_{n\to\infty}\lambda _n (h)$
 exists, and, by the above, convex and finite on $\bbR^d$. 
 
 Let us check that $\lambda (0 )=\min_h \lambda (h )= 0$. 
Assumption  {\bf (A1)} implies that 
\be
\label{eq:L4-phi-1}
\text{$\phi_\beta$ is monotone non-decreasing,\   
$\phi_\beta (1) = \min_\ell\phi_\beta  (\ell ) >0$}, 
\ee
and that 
\be 
\label{eq:TH-sublin}
\lim_{\ell\to\infty}\frac{\phi_\beta  (\ell )}{\ell} = 0. 
\ee
Next we rely on the following  well-known estimate 
for the underlying finite  range random walk: 
\smallskip 

\noindent
{\bf Estimate~1}. There exists $c = c (\Pd ) <\infty$ and $L_0 <\infty$, 
such that for any $L >L_0$, 
\be 
\label{eq:L4-EstimateRW-1}
\Pd\lb \max_{\ell\leq n}\abs{\sfX (\ell )} \leq L\rb \geqs {\rm e}^{-c\frac{n}{L^2}} .
\ee
uniformly in $n\in \bbN$.

\noindent
Let $\Lambda_L = \lbr \sfx ~:~ \abs{\sfx}_1\leq L\rbr$. Then, 
\[
 Z_n (h ) \geq {\rm e}^{- \abs{h}L} \sum_{\gamma\subset B_L} 
 \Wd (\gamma )\1_{\lbr \abs{\gamma}= n\rbr} .
\]
If $\abs{\gamma} = n$ and $\gamma\subset B_L$, then 
\[ 
 \Phi_\beta (\gamma ) \leq \lb 2L +1\rb^d \max_{\ell\leq n}\phi_\beta  (\ell ) .
\]
Consequently, in view of \eqref{eq:L4-EstimateRW-1}, 
\be 
\label{eq:L4-Zn-lbR}
\liminf_{n\to\infty} \frac{1}{n}\log Z_n (h ) \geq -\frac{c}{L^2} - 
\lb 2L +1\rb^d \liminf_{n\to\infty}
\frac{\max_{\ell\leq n}\phi (\ell )}{n} ,
\ee
for any $L >L_0$. 
 By \eqref{eq:TH-sublin}, 
\[
 \lim_{n\to\infty} \frac{\max_{\ell\leq n}\phi_\beta  (\ell )}{n} = 0. 
\]
It follows that $\lambda (h )\geq 0$ for any $h\in\bbR$. 
On the other hand, since the interaction potential
$\phi_\beta$ is non-negative $Z_n =Z_n (0) \leq 1$, and, 
consequently, $\lambda (0)\leq 0$.  Hence $\lambda (0) = 0 =\min_h \lambda (h )$
as claimed. 
\smallskip 

Since $\lambda\geq 0$, the set $\Knot$ in \eqref{eq:L2-critical-a} is convex. In 
order to check that it contains an open neighbourhood of the origin it would be enough
to show that there exists $\delta >0$, such that
\be 
\label{eq:TH-serriesKnot}
\sum_n \Zn (h ) = \sum_\sfx {\rm e}^{h\cdot \sfx} G_0 (\sfx ) <\infty , 
\ee
whenever, $\abs{h}<\delta$.  The convergence in \eqref{eq:TH-serriesKnot} 
 will follow as soon as we shall show that
the critical two-point function $G_0 (\sfx )$ 
in \eqref{eq:Glx} is exponentially decaying in $\sfx$. We continue 
to employ notation $\calP_\sfx$ for paths $\gamma$ with $\sfX (\gamma ) =\sfx$ (and, of course, 
with $\Pd (\gamma ) >0$). Consider the disjoint decomposition
\[
 \calP_\sfx = \bigcup_{k\geq 2}\calP_\sfx^{(k )} , 
\]
where (see Figure~\ref{fig:rhok}), 
\begin{figure}[t]
\sidecaption
\scalebox{.3}{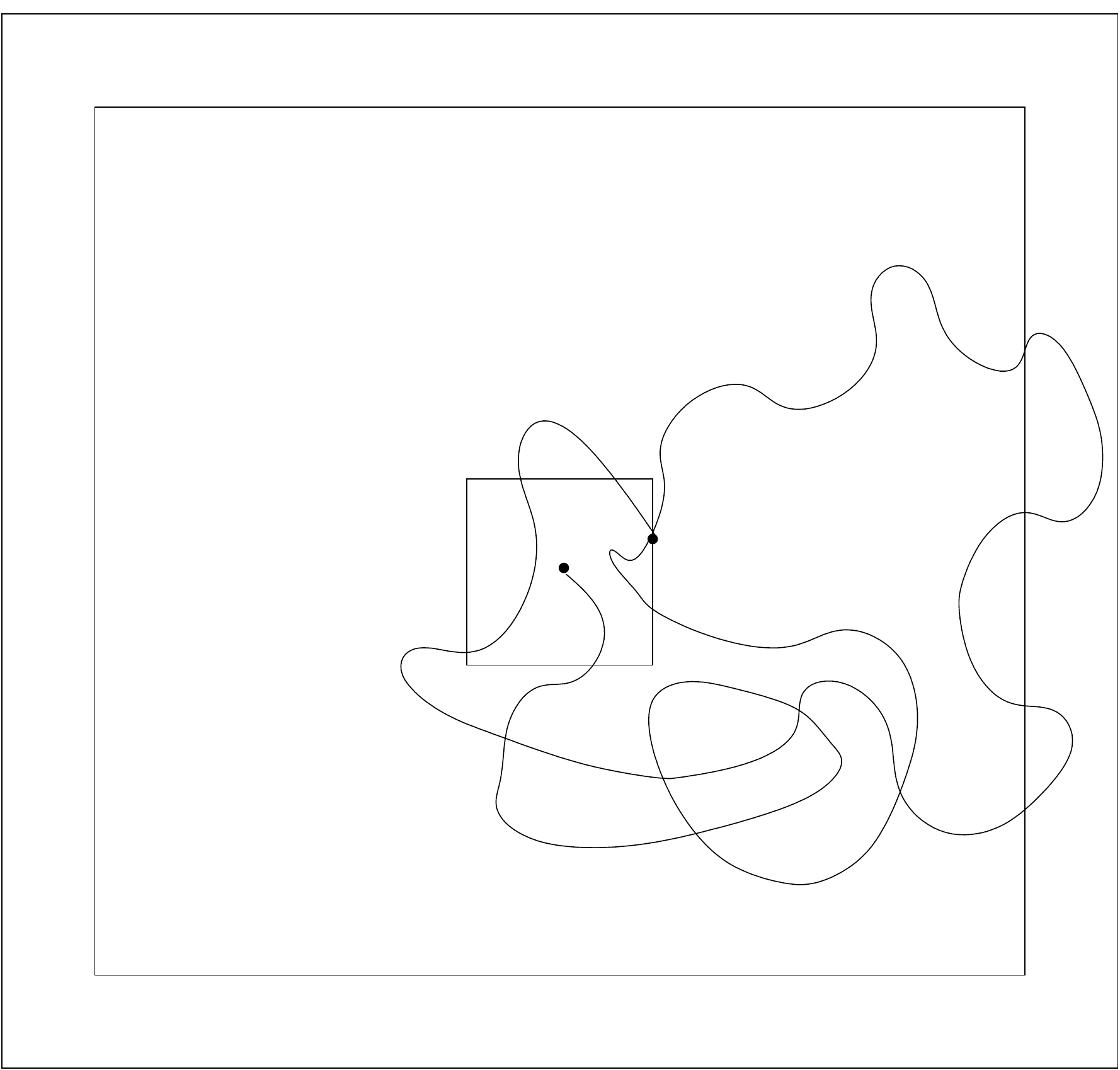}
%
%
\caption{Paths $\gamma\in \calP_\sfx^{(k )}$.}
\label{fig:rhok}       
\end{figure}
\[
 \calP_\sfx^{(k )} = \lbr \gamma\in\calP_\sfx ~:~\gamma\subset\Lambda_{k\abs{\sfx}_1}
 \rbr\setminus 
\lbr \gamma\in\calP_\sfx ~:~\gamma\subset\Lambda_{(k-1)\abs{\sfx}_1}\rbr .
\]
If $\gamma\in \calP_\sfx^{(k )}$, then since the range $R$ of 
the underlying random walk is finite, 
\[
 \Phi_\beta  (\gamma )\geq \frac{(k-1)\abs{\sfx}_1}{R} \inf_{\ell}\phi_\beta (\ell ) = 
\frac{(k-1)\abs{\sfx}_1}{R}
 \phi_\beta (1)  .
\]
As a result, 
\[
 G_0^{(k )} (\sfx ) \df \Wd \lb \calP_\sfx^{(k )} \rb 
 \leq {\rm e}^{ -(k-1)\abs{x}_1 \phi_\beta (1)/R} \sum_{\gamma\in \calP_\sfx^{(k )}}\Pd (\gamma ).
\]
At this stage we shall rely on another well known estimate for short range zero-mean 
random walks: 

\noindent
{\bf Estimate~2}. Let $\sigma_0$ be the first hitting time of $0$. Then, 
\be 
\label{eq:L4-EstimateRW-2} 
\Ed\lb \sum_{\ell = 0}^{\sigma_0} \1_{\lbr \sfX (\ell )=\sfx\rbr}~\big|\sfX (0 )=\sfx\rb 
\leqs A_d (\abs{\sfx} )\df  
\begin{cases}
 \abs{\sfx},\quad & d=1\\
\log\abs{\sfx},& d=2\\
1,& d\geq 3 
\end{cases}
\ee
uniformly in $\sfx\in \bbZ^d$. 
\smallskip 
By a crude application of  \eqref{eq:L4-EstimateRW-2}, 
\be 
\label{eq:Glk-bound}
\sum_{\gamma\in \calP_\sfx^{(k )}}\Pd (\gamma ) \leq A_d ( k\abs{\sfx}_1 ) \ 
\Rightarrow \ G_0^{(k )} (\sfx ) \leq 
A_d (k \abs{\sfx}_1 ) 
 {\rm e}^{-\frac{ \phi_\beta (1) }{R} (k-1 ) 
 \abs{\sfx}_1 } . 
\ee
Therefore, 
\[
 G_0 (\sfx ) \leq \sum_{k\geq 1} A_d (( k+1) \abs{\sfx}_1 ) 
 {\rm e}^{-\frac{ \phi_\beta (1) }{R} k 
 \abs{\sfx}_1 } , 
\]
and \eqref{eq:TH-serriesKnot} follows. 
\smallskip 

\noindent 
{\bf Part~B of Theorem~\ref{thm:attractive-morp}.} As we have already noted, 
due to a possible over-counting it is not obvious  that $G_\lambda (\sfx +\sfy )\geq 
G_\lambda (\sfx )G_\lambda (\sfy )$. However, 
in view of the attractivity \eqref{eq:attractive} of $\phi_\beta$, 
the latter super-multiplicativity property
holds for the following first-hitting time version $H_\lambda$ of $G_\lambda$:
\be 
\label{eq:TH-Hlambda}
H_\lambda (\sfx ) = \sum_{\gamma\in\calP_\sfx} {\rm e}^{-\lambda \abs{\gamma}}
\Wd (\gamma )\1_{\ell_\sfx (\gamma )=1} .
\ee
In particular, the limit 
\be
\label{eq:TH-tl-H}
\tl (\sfx ) = -\lim_{r\to\infty}\frac{1}{r}\log H_\lambda\lb {\lfloor r\sfx\rfloor}\rb
\ee
exists, and, by Proposition~\ref{prop:L2-Hammersley-sabad}, 
is a non-negative, convex, homogeneous of order one function on $\bbZ^d$. Furthermore, 
$H_\lambda (\sfx )\leqs {\rm e}^{-\tl (\sfx )}$. 
\begin{ex}
 \label{ex:prop-H-tl}
 Prove the above statements. 
\end{ex}
We claim that the second of \eqref{eq:L2-tl} holds with the very same $\tl$. Clearly, 
$H_\lambda (\sfx )\leq G_\lambda (\sfx )$. The proof, therefore, boils down to a 
derivation of a complementary upper bound, which would render negligible correction 
on the logarithmic scale. We shall consider two cases: Fix any $\lambda_0 >0$.
\smallskip 

\noindent
\opt{1} $\lambda > \lambda_0$. Then for any $\sfx$,  
\be 
\label{eq:TH-Case1-Gl}
G_\lambda (\sfx ) \leq H_\lambda (\sfx ) G_\lambda (0) 
\leq H_\lambda (\sfx ) \sum_\sfy G_\lambda (\sfy )
\leq H_\lambda (\sfx )\frac{1}{1- {\rm e}^{-\lambda_0}} .
\ee
\smallskip 

\noindent
\opt{2} $\lambda\leq \lambda_0$. Evidently $\tl$ is non-decreasing in $\lambda$. 
Define: 
\be
\label{eq:TH-knot}
k_0 = 3\frac{R}{\phi_\beta (1 )} \max_{\sfy}\frac{\tau_{\lambda_0} (\sfy )}{\abs{\sfy}_1}. 
\ee
By \eqref{eq:Glk-bound} , 
\[
 \sum_{k\geq k_0} G_\lambda^{(k)} (\sfx ) \leq \sum_{k\geq k_0} G_0^{(k)} 
 \leq \sum_{k\geq k_0} A_d (k \abs{\sfx}_1 ) 
 {\rm e}^{-\frac{ \phi_\beta (1) }{R} (k-1 )\abs{\sfx}_1 } \leq {\rm e}^{-2\tl (\sfx )}
\]
is exponentially negligible with respect to ${\rm e}^{-\tl (\sfx )}$. Consequently, 
\be 
\label{eq:TH-GlHl-bound} 
G_\lambda (\sfx ) \leqs 
\sum_{k < k_0} G_\lambda^{(k)} (\sfx ) \leq H_\lambda (\sfx )  A_d (k_0 \abs{\sfx}_1 ) 
\leq {\rm e}^{-\tl (\sfx )}A_d (k_0 \abs{\sfx}_1 ) , 
\ee
and the second of \eqref{eq:L2-tl} indeed follows. 
\smallskip 

\noindent
{\bf $\tl$ is the support function of $\Kl$.} In order to see this notice that
 $\Kl = \lbr h~:~\lambda (h )\leq \lambda\rbr$ is the closure of the 
 domain of convergence
 \be 
 \label{eq:TH-serries}
  h \mapsto \sum_n {\rm e}^{-\lambda n}\Zn (h ) = \sum_{\sfx } {\rm e}^{h\cdot\sfx} 
  G_\lambda (\sfx ) . 
 \ee
 Consider 
 \be 
 \label{eq:TH-alpha-l}
 \alpha_\lambda (h ) = \max\lbr h\cdot \sfx ~:~ \tl (\sfx )\leq 1\rbr .
 \ee
The series in \eqref{eq:TH-serries} diverges if $\alpha_\lambda (h )> 1$, whereas, 
$h\in {\rm int}(\Kl )$ if $\alpha_\lambda (h ) <1$. Hence, 
\be 
\label{eq:TH-support} 
\partial\Kl = \lbr h~:~ \alpha_\lambda (h ) = 1\rbr\quad{\rm or}\quad 
\tl (\sfx ) = \max_{h\in\partial \Kl} h\cdot\sfx .
\ee

\noindent 
{\bf Part  C of Theorem~\ref{thm:attractive-morp}.} 
To be precise large deviations are claimed for the distribution of end-points, 
which we, with a slight abuse of notation,  proceed to call $\Pnf{h} (\sfx )$:
\be 
\label{eq:TH-Pnfx}
\Pnf{h} (\sfx ) = \sum_{\sfx (\gamma ) = \sfx }\Pnf{h} (\gamma ) \df  \frac{ {\rm e}^{h\cdot\sfx}
\Zn (\sfx  )}{\Zn (h )} . 
\ee
Let $h\in\bbR^d$. The limiting log-moment generation function under 
the sequence of measures $\lbr \Pnf{h}\rbr$
is 
\[
 \lim_{n\to\infty}\frac{1}{n}\log \Enf{h}\lb {\rm e}^{f\cdot \sfX}\rb = 
 \lambda (h+f  )-\lambda (h ). 
\]
The function $I_h$ in \eqref{eq:L3-LD-rate-fh} is just the Legendre-Fenchel transform of
the above. Since the underlying random walk has bounded range, exponential tightness 
is automatically ensured. By Theorem~\ref{thm:LD-convex} and Excercise~\ref{ex:LD.4} 
of the Appendix, upper large deviation bounds hold with $I_h$. 

We still need a matching lower bound. Let us sketch the proof which  relies on sub-additivity and 
Lemma~\ref{lem:LD-IJ} of Appendix. 
Due to a possible 
over-counting the function $n\to\log \Zn (\lfloor n\sfx\rfloor  )$ is not necessarily 
super-additive. However, its first hitting time version (see \eqref{eq:TH-FamiliesPolymers} for the 
definition of $\calP_{\sfx , n}$);  
\be 
\label{eq:TH-Zhat}
\hat\Zn (\lfloor n\sfx\rfloor  ) = \sum_{\gamma\in\calP_{\sfx ,n}} \Wd (\gamma )
\1_{\lbr \ell_\gamma (\sfx )= 1\rbr}, 
\ee
is super-additive. 
By Proposition~\ref{prop:L2-Hammersley-sabad} the limit
\be 
\label{eq:TH-J-funcZ}
J (\sfx ) = - \lim_{n\to\infty} \frac{1}{n} \log \hat\Zn (\lfloor n\sfx\rfloor  )\ 
{\rm and}\ \hat\Zn (\lfloor n\sfx\rfloor  )\leq c {\rm e}^{-n J (\sfx )}
\ee
exists and is a convex non-negative function on $\bbR^d$. Some care is needed to make this statement
 rigorous. Indeed, $\hat\Zn (\sfx )= 0$ whenever $\sfx$ does not belong to the 
 (bounded) range of the underlying $n$-step walk, and  what happens at boundary 
 points should be explored separately. We shall 
 ignore this issue  here. 
 
A slight modification of arguments leading to \eqref{eq:L2-lambda0-a} imply that $J (0)= 0$. 
By convexity this means that for any $\alpha\in (0,1)$ and any $\sfx$,  
\be 
\label{eq:TH-J-bound}
J(\sfx )\leq \alpha J\lb\frac{\sfx}{\alpha}\rb .
\ee
Since, $\phi_\beta$ is non-negative, 
\be 
\label{eq:TH-Z-Zhat-bound}
\hat\Zn (\lfloor n\sfx\rfloor  ) \leq \Zn (\lfloor n\sfx\rfloor  ) \leq \sum_{m=1}^{n} 
\hat Z_m (\lfloor n\sfx\rfloor  ) \leq 
\sum_{m=1}^{n}  c {\rm e}^{-m J (\frac{n}{m}\sfx )} \leq cn {\rm e}^{-n J (\sfx )} , 
\ee
where we used \eqref{eq:TH-J-funcZ} and \eqref{eq:TH-J-bound} on the last two steps. 
We conclude: 
\be 
\label{eq:TH-J-func}
J (\sfx ) = - \lim_{n\to\infty} \frac{1}{n} \log \Zn (\lfloor n\sfx\rfloor  )\ 
{\rm and}\ \Zn (\lfloor n\sfx\rfloor  )\leq c {\rm e}^{-n J (\sfx ) +\frac{\log n}{n}}
\ee
 Now, \eqref{eq:TH-J-func} means that \eqref{eq:LD-LB} and 
 \eqref{eq:LD-UB-x} are satisfied, the latter uniformly in $\sfx$. Since the range of the underlying 
 random walk is bounded, \eqref{eq:J-growth}, and in particular exponential tightness, 
 is trivially satisfied as well. Hence, by Lemma~\ref{lem:LD-IJ},  $I = J$. Since 
 \[
  -\frac{1}{n}\log \Pnf{h} (\lfloor n\sfx\rfloor ) = - \frac{1}{n} \log \Zn (\lfloor n\sfx\rfloor  ) 
  - \lb h\cdot\frac{\lfloor n\sfx\rfloor}{n} - \frac{\log\Zn (h )}{n}\rb , 
 \]
claim {\bf C} of Theorem~\ref{thm:attractive-morp} follows as well. 
\smallskip 

\noindent 
{\bf Part  D of Theorem~\ref{thm:attractive-morp} and 
limiting spatial extension.} 
For drifts   $h\in  {\rm int}\lb \Knot \rb$, 
\[
 I_h (\sfv ) = \sup_g \lbr g\cdot\sfv - \lb \lambda (h+g ) - \lambda (h )\rb\rbr \geq \abs{\sfv} 
{\rm dist}\lb h , \pKnot\rb > 0 , 
\]
for any $\sfv\neq 0$. Hence, if 
$h\in  {\rm int}\lb \Knot\rb $, then 
 the model is sub-ballistic in the sense of \eqref{eq:L2-subballistic}.

On the other hand, if $h\not\in\Knot$ or, equivalently, if $\lambda (h )>0$, then
\[
 I_h (0) = \sup_{g}\lbr -\lambda (g+h ) +\lambda (h )\rbr = \lambda (h ) >0. 
\]
This is a rough expression of ballisticity. 
It implies that the polymer is pulled away from the origin on the 
linear scale, but  it does not imply that the limit in \eqref{eq:L2-ballistic} exists.

More precisely, if  $\lambda$ is differentiable at $h\not\in \Knot$, 
then \eqref{eq:L2-ballistic} holds with $\sfv = \nabla\lambda (h )$. Indeed in the 
latter case
$I_h$ is strictly convex at $\sfv$ and, consequently $I_h (\sfv )= 0$ is the unique 
minimum. Furthermore, in such a case, the following law of large numbers holds: For 
any $\epsilon >0$, 
\be 
\label{eq:TH-LLN}
\sum_n \Pnf{h }\lb \Big| \frac{\sfX}{n} - \sfv\Big| \geq \epsilon\rb < \infty , 
\ee
and the series converge exponentially fast. 

However, the above sub-additivity based thermodynamics of annealed polymers does not 
imply that the sub-differential $\partial \lambda (h ) \df\calM_h = \lbr \sfv : I_h (\sfv )= 0 \rbr$
is always  a singleton. The general form of \eqref{eq:TH-LLN} is  
\be 
\label{eq:TH-LLN-Mh}
\sum_n \Pnf{h }\lb \min_{\sfv\in\calM_h}\Big| \frac{\sfX}{n} - \sfv\Big| \geq \epsilon\rb < \infty .  
\ee
Therefore, in general,  large deviations 
(Part  {\bf C} of Theorem~\ref{thm:attractive-morp})  imply neither existence of the 
limit in \eqref{eq:L2-ballistic}, nor a LLN. The set $\calM_h$ could be characterized as follows
\cite{Flury-LD,IV-Erwin}: 
\begin{lem}
 \label{lem:TH-Mhset}
 For any $h\not\in\Knot$ the set $\calM_h$ satisfies: Set $\mu = \lambda (h ) >0$. Then, 
 \be 
 \label{eq:TH-Mhset}
 \sfv\in\calM_h \Longleftrightarrow 
 \begin{cases}
  &\tau_\mu  (\sfv ) = h\cdot\sfv\\
  &\frac{\dd^-}{\dd\lambda}\Big|_{\lambda = \mu }\tl (\sfv )\leq 1\leq 
  \frac{\dd^+}{\dd\lambda}\Big|_{\lambda = \mu }\tl (\sfv ) 
 \end{cases}
\ee
\end{lem}
\begin{proof}
 By \eqref{eq:L3-LD-rate-fh}, 
 \[
  \sfv\in \calM_h\Longleftrightarrow 
  \sup_{\lambda}\lb \tl (\sfv ) - \lambda )\rb  + \lb \mu - h\cdot\sfv \rb =0 .
 \]
The choice $\lambda =\mu$ implies that $\tl (\sfv )\leq h\cdot\sfv$. Since $h\in\partial{\bf K}_\mu$, 
the first of \eqref{eq:TH-Mhset} follows by \eqref{eq:TH-support}. As a result, 
\be 
\label{eq:TH-concave-tl}
\tl (\sfv ) - \tau_\mu (\sfv ) \leq \lambda -\mu ,  
\ee
for any $\lambda$. Since the function $\lambda\to\tl (\sfv )$ is concave, left and right derivatives are
well defined, and the second of \eqref{eq:TH-Mhset} follows from \eqref{eq:TH-concave-tl}. 
\end{proof}

The differentiability (and even analyticity) of $\lambda$ at 
super-critical drifts $h\not\in \Knot$ and, in particular, 
the existence of the limit in  \eqref{eq:L2-ballistic} 
and the LLN \eqref{eq:TH-LLN}, 
is established in 
Subsection~\ref{ssec:BalPhaseA} 
as a consequence of much sharper asymptotic results based on analysis of renewal
structure of ballistic polymers. 
\smallskip 

\noindent 
{\bf Part  E of Theorem~\ref{thm:attractive-morp}.} 
Finally, $I_h (0) = 0$ whenever $h\in\pKnot$, which sheds little light on 
ballistic  properties
of the model at critical drifts. The critical case was worked out in
\cite{IV-Critical} via refinement of the renormalization construction of 
the Ornstein-Zernike theory (see Subsection~\ref{ssec:BalPhaseA}), and 
it is beyond the scope of these notes to 
reproduce the corresponding arguments here. 
\smallskip

\subsection{Thermodynamics of quenched polymers.}
\label{sub:TH-quenched}
The underlying random walk imposes a directed graph structure on $\bbZ^d$. Let us
say that $\sfy$ is a neighbour of $\sfx$; $\sfx \leadsto\sfy$ if $\Pd (\sfy -\sfx )>0$. 
Because of \eqref{eq:L2-sfe} and Assumption {\bf (A.2)} there is a unique 
infinite component ${\rm Cl}_\infty$ of $\lbr \sfx :\sfVo_\sfx <\infty\rbr$. Clearly, 
non-trivial thermodynamic limits may exist only if $0\in {\rm Cl}_\infty$. Furthermore, 
if $\calE \lb \sfVo \rb =\infty$, then 
$\sum_r \calQ\lb V^\omega_{\lfloor r\sfx\rfloor  }> cr \rb =\infty$ for any $c>0$, 
and consequently, 
\[
 \liminf_{r\to\infty} \frac{1}{r}\log G_\lambda^\omega (\lfloor r\sfx \rfloor ) = -\infty , 
\]
$\calQ$-a.s. for any $\sfx\neq 0$. Hence, in order to define inverse 
correlation length $\tlo$ one needs either to impose more stringent requirements
on disorder and use \eqref{eq:L2-tl}, or to find a more robust definition of $\tlo$. 
A more robust  definition is in terms of the so called point to hyperplane exponents:
\newline
Given $h\neq 0$ define $\cHp{h, t} = \lbr \sfx : h\cdot\sfx\geq t\rbr$. Let 
$\calP_{h, t}$ be the set of paths $\gamma = \lb \gamma (0), \dots , \gamma (n)\rb$ with 
$\gamma (n)\in \cHp{h, t}$. For $\lambda\geq 0$ consider, 
\[
 D_{h, \lambda}^\omega (t ) = \sum_{\gamma\in \calP_{h, t}} {\rm e}^{-\lambda\abs{\gamma}}
 \Wdo (\gamma ) .
\]
Assume that the limit 
\be 
\label{TH-qD-lim}
-\lim_{t\to\infty}\frac{1}{t}\log D_{h, \lambda}^\omega (t ) \df 
\frac{1}{\alpha^\sfq_\lambda  (h )}
\ee
exists. Should the inverse correlation length $\tlo$ be also defined (and positive), 
the following relation should hold:
\be 
\label{eq:TH-tau-alpha}
\frac{1}{\alpha^\sfq_\lambda (h )} = \min_{\sfx\in \cHp{h, 1}} \tlo (\sfx ). 
\ee
If $\tlo$ is the support function of a convex set $\Klo$, then, by 
\eqref{eq:A-dual-support} of the Appendix, $\alpha^\sfq_\lambda$ should be the support
function of the polar set \eqref{eq:A-polar} or, equivalently, the Minkowski function
of $\Klo$. 

Conversely, if the limit $\alpha_\lambda^\sfq$ in \eqref{TH-qD-lim} exists, 
then we may define $\tlo$ via 
\be 
\label{eq:TH-tlo-indirect}
\tlo (\sfx )  = \max\lbr h\cdot x ~ :~  \alpha_\lambda^\sfq (h)\leq 1\rbr , 
\ee
even if a direct application of \eqref{eq:L2-tl} does not make sense. 
\smallskip 

There is an extensive literature on thermodynamics of quenched models, 
\cite{Sznitman, Zerner, Flury-LD, Mourrat} to mention a few. The paper \cite{Mourrat}
 contains state of the art information on the matter, and several conditions on 
 the random environment were worked out there in an essentially optimal form. The 
 treatment of $\calQ (\sfVo = \infty  )>0$ case and, more generally, of 
 $\calE (\sfVo ) = \infty$ case is based on renormaliztion techniques for high 
 density site percolation and, eventually, on sub-additive ergodic theorems and 
 large deviation arguments. It is beyond the scope of these lectures to 
 reproduce the corresponding results here. Below we formulate some of the statements
 from \cite{Mourrat} and refer to the latter paper for proofs and detailed discussions.
 
 We assume {\bf (A1)} and {\bf (A2)}. 
 \begin{thm}
\label{thm:quenched-morp} 
The following happens $\calQ$-a.s on the event $0\in {\rm Cl}_\infty$:
\newline
{\bf A.}  The free energy $\lambda^\sfq $ is well defined, deterministic,  non-negative  
and convex on $\bbR^d$. 
Furthermore, 
\be 
\label{eq:L2-lambda0-o}
0  = \min_h \lambda^\sfq  (h) = \lambda^\sfq  (0). 
\ee
The set 
\be
\label{eq:L2-critical-a-o}
{\bf K}_0^\sfq \df \lbr h~:~ \lambda^\sfq   (h )= 0\rbr 
\ee
is  a compact  convex set  with a non-empty interiour. \newline
{\bf B.} The 
point to hyperplane exponent $\alpha_\lambda^\sfq$ in 
\eqref{TH-qD-lim}
is well defined for 
any $\lambda\geq 0$. Consequently, the inverse 
correlation length $\tlo$ is well defined
 via \eqref{eq:TH-tlo-indirect} 
also  for any $\lambda\geq 0$, 
and, furthermore, 
it can be identified as the support function of the compact convex set 
\be
\label{eq:TH-Kl-o}
{\bf K}_\lambda^\sfq  \df \lbr h~:~ \lambda^\sfq  (h )\leq \lambda \rbr .
\ee
Define 
\be 
\label{eq:L3-LD-rate-f-o}
 I^\sfq  (\sfv ) = \sup_{h}\lbr h\cdot \sfv - \lambda^\sfq  (h )\rbr = 
 \sup_{\lambda}\lbr \tl^\sfq (\sfv ) - \lambda\rbr  .
\ee
{\bf C.} 
For any $h\in\bbR^d$ the family of polymer measures $\Pnfo{h}$
satisfies LD principle with the rate function
\be 
\label{eq:L3-LD-rate-fh-o}
 I_h^\sfq (\sfv )\df \sup_f \lbr f\cdot \sfv - \lb  \lambda^\sfq  (f +h ) 
 -\lambda^\omega (h )\rb 
 \rbr = I^\sfq (\sfv)  - \lb 
h\cdot\sfv - 
\lambda^\sfq  (h )\rb. 
\ee
{\bf D.} For $h\in {\rm int}\lb  {\bf K}_{0}^\sfq\rb  $  the model is sub-ballistic,  
whereas for any    
$h\not\in {\bf K}_{0}^\sfq  $ the model is ballistic in the sense that $I_h^\sfq (0) >0$. 
\end{thm}
The above theorem does not imply  strong limiting spatial extension form 
of the ballisticity condition \eqref{eq:L2-ballistic-q} for all $h\not\in {\bf K}_0^\sfq$, 
exactly for the same reasons 
as Theorem~\ref{thm:attractive-morp} does not imply the corresponding 
statement for annealed models. Existence of limiting spatial extension 
for quenched models in the very weak disorder regime is discussed, together
with other limit theorems, in Section~\ref{sec:Weak}.

In the case of critical drifts $h\in \partial{\bf K}_0^\sfq$, 
a form of ballistic behaviour was established in the continuous context in 
\cite{Sznitman95}.

\section{Multidimensional Renewal Theory and Annealed Polymers.}

\label{Asec:renewal}
\subsection{Multi-dimensional renewal theory}
\label{ssec:renewal}
{\bf One-dimensional renewals.} 
Let $\lbr \sff (\sfn )\rbr$ be a probability distribution on $\bbN$ (with strictly positive variance). 
We can think of $\sff$ as of a probability distribution for a step $\sfT  $ of the effective 
 one-dimensional 
random walk
\[
 \sfS_N = \sum_1^N \sfT_i .
\]
The distribution of $\lbr\sfS_n\rbr$ is governed by the product measure $\bbP$. 
The renewal array
$\lbr \sft (n )\rbr $ is given by
\be 
\label{eq:ARenewal}
\sft (0 ) = 1 \quad{\rm and}\quad 
\sft (n) = \sum_{m=1}^{n} \sff (m )\sft ( n-m ) .
\ee
In probabilistic terms \eqref{eq:ARenewal} reads as:
\be 
\label{eq:ARenewal-p}
 \sft (n) = \bbP\lb ~\exists ~N~:~ \sfS_N = n\rb = \sum_N \bbP\lb \sfS_N = n\rb .
\ee

Renewal theory implies that  
\be 
\label{eq:ARen-lim}
\lim_{n\to\infty}\sft (n ) = \frac{1}{\bbE \sfT }\df \frac{1}{\mu}.
\ee 
A proof of \eqref{eq:ARen-lim} is based on an analysis of  complex power series
\be 
\label{eq:A-hatt}
\hat\sft (z)  \df  \sum_n \sft (n) z^n \quad{\rm and}\quad 
\hat\sff (z)  \df  \sum_n \sff (n) z^n 
\ee
\begin{ex}
 \label{ex:R.1}
Show that $\hat\sft $ is absolutely convergent and hence analytic on the interiour of the unit 
disc $\bbD_1 =\lbr z : \abs{z} <1\rbr$. Check that $\hat\sft (1 ) = \infty$. 
\end{ex} 
It follows that on $\bbD_1$, 
\be 
\label{eq:ARenewal-hat}
\hat\sft (z) = \frac{1}{1-\hat\sff (z )} .
\ee
\begin{ex}
 \label{ex:R.2}
Check that $\abs{\sff (z)}\leq \sff (\abs{z}) <1 $ for any $z\in\bbD_1$. Prove \eqref{eq:ARenewal-hat}. 
\end{ex}
Consequently, by Cauchy formula, 
\be 
\label{eq:AR-Cauchy}
\sft (n) = \frac{1}{2\pi i}\oint_{\abs{z} = r}\frac{\hat\sft (z)}{z^{n+1}}\dd z = 
\frac{1}{2\pi i}\oint_{\abs{z} = r}\frac{\dd z}{z^{n+1}\lb 1-\hat\sff (z )\rb} , 
\ee
for any $r <1$.

\noindent 
{\bf Exponential tails.} Assume that there exists $\nu >0$, such that
\be 
\label{eq:AR-exp}
\sff (\sfn ) \leqs {\rm e}^{-\nu  n}, 
\ee
uniformly in $n \in\bbN$. 
\begin{lem}
\label{lem:AR-lim-exp} 
Under Assumption \eqref{eq:AR-exp} the convergence in 
\eqref{eq:ARen-lim} is exponentially fast in $n$. 
\end{lem}
We start proving Lemma~\ref{lem:AR-lim-exp} by noting that under 
\eqref{eq:AR-exp} the function 
 $\hat\sff$ is defined and analytic on $\bbD_{1+\nu}$. 
 \begin{ex}
 \label{ex:R.3}
Check that there exists $\epsilon \in (0,\nu )$ such that $z=1$ is the only zero of $1-\sff (z)$ on 
$\bar\bbD_{1+\epsilon}$. Furthermore, $\frac{1-z }{\hat\sff (z) - 1 }$ is analytic on $\bbD_{1+\epsilon}$.
\end{ex}

Recall that we defined $\mu = \sum_n n\sff (n) = \hat\sff^\prime (1)$. Consider the representation, 
\[
 \frac{1}{\mu} = \frac{1}{2\pi i}\int_{\abs{z}=r}\frac{\dd z}{\mu (1-z) z^{n+1}} ,
\]
which holds for any $r <1$. By \eqref{eq:AR-Cauchy}, 
\be 
\label{eq:AR-difference}
\sft (n) - \frac{1}{\mu} = \frac{1}{2\pi i}\int_{\abs{z}=r}\frac{\hat\sff (z) - 1 -\mu (z-1 )}{(1-\hat\sff (z ))
(1-z) z^{n+1}}\dd z \df \frac{1}{2\pi i}\int_{\abs{z}=r}\frac{\Delta (z)}{z^{n+1}}\dd z .
\ee
On $\bar\bbD_{1+\epsilon}$ the denominator in the definition of $\Delta$ vanishes only at $z=1$. 
However, since $\mu = \hat\sff^\prime (1)$,  expansion of the numerator in a neighbourhood of $z=1$ gives:
\[
 \hat\sff (z) - 1 -\mu (z-1 )  = (z-1)^2 U (z) ,
\]
with some analytic $U$. It follows that 
\[
 \Delta (z) = \frac{U (z)}{ (1- \hat\sff (z) )/(z-1 )} .
\]
In view of Exercise~\ref{ex:R.3} $\Delta$ is analytic on $\bar\bbD_{1+\epsilon}$. As a result, 
\[
  \frac{1}{2\pi i}\int_{\abs{z}=r}\frac{\Delta (z)}{z^{n+1}}\dd z = 
 \frac{1}{2\pi i}\int_{\abs{z}=1+\epsilon}\frac{\Delta (z)}{z^{n+1}}\dd z .
\]
By \eqref{eq:AR-difference}, 
\be 
\label{eq:AR-difbound}
\left| \sft (n) - \frac{1}{\mu}\right| \leqs (1+\epsilon )^{-n} , 
\ee
uniformly in $n$. This is precisely the claim of Lemma~\ref{lem:AR-lim-exp}. \qed

\noindent
{\bf Complex renewals.} 
Suppose that \eqref{eq:ARenewal} holds with complex $\lbr \sff (n)\rbr$ and, accordingly, 
with complex  $\lbr \sft (n )\rbr$. As before, define  $\hat\sft (z)$ and 
$\hat\sff (z )$ as in \eqref{eq:A-hatt}. 
In the sequel we shall work with complex renewals 
which satisfy one of the following two assumptions, Assumption~\ref{as:ARenewal-c} 
or Assumption~\ref{as:ARenewal-c-small}, below. 


\begin{assumption}
\label{as:ARenewal-c} 
There exists $\epsilon >0$, such that 
the function $\hat\sff$ satisfies the following three properties:

\noindent
{\bf (a)} $\hat\sff (0) = 0$ and the $\hat\sff (z)$ in \eqref{eq:A-hatt} is 
 absolutely convergent in a neighbourhood of $\bar\bbD_{1+\epsilon}$. 

\noindent
{\bf (b)} $z=1$ is the only zero of $\lb \hat\sff (z ) - 1\rb $ in $\bar\bbD_{1+\epsilon}$.

\noindent
{\bf (c)} $\hat\sff^\prime (1 ) 
\neq 0$. 
\end{assumption}
Under Assumption~\ref{as:ARenewal-c} the exponential convergence bound \eqref{eq:AR-difbound} still
holds. Indeed, the only thing we have to justify is that $\hat\sft (z) = \sum_n \sft (n)  z^n$ is defined 
and  
analytic on some neighbourhood of the origin, and that 
\be 
\label{eq:AR-tvrsf}
 \hat\sft (z) = \frac{1}{1-\hat\sff (z)} ,
\ee
for all $\abs{z}$ sufficiently small. Indeed, if this is the case, then \eqref{eq:AR-Cauchy} holds for some 
$r >0$, and we may just proceed as before. However, by Assumption~\ref{as:ARenewal-c}{\rm (a)}, 
$\sum_n \abs{\sff (n)}\abs{z}^n <1$
for all $\abs{z}$ small enough. Hence \eqref{eq:AR-tvrsf}. 
 \begin{assumption}
 \label{as:ARenewal-c-small} 
 There exists $\epsilon >0$, such that 
the function $\hat\sff$ satisfies the following two  properties:

\noindent
{\bf (a)} The series $\hat\sff (z)$ in \eqref{eq:A-hatt} is 
 absolutely convergent in a neighbourhood of $\bar\bbD_{1+\epsilon}$.
 
 \noindent
 {\bf (b)} There exists  $\kappa >0$, such that 
 $\min_{\abs{z}\leq 1+\epsilon}\abs{ 1- \hat\sff (z )} \geq \kappa$.
\end{assumption}
Under Assumption~\ref{as:ARenewal-c-small}, the function 
$\lb 1- \hat\sff (z)\rb^{-1}$ is analytic in a neighbourhood of $\bar\bbD_{1+\epsilon}$, and 
the Cauchy formula \eqref{eq:AR-Cauchy}, which again by absolute convergence of $\hat\sff$ still holds for 
$r$ sufficiently small,  implies: 
\be 
\label{eq:AR-tn-c-small}
\abs{\sft (n )} \leq \frac{1}{\kappa (1+\epsilon )^n} .
\ee
\smallskip 

 \noindent
{\bf Multi-dimensional renewals.} 
Let $\lbr \sff (\sfx , \sfn )\rbr$ be a probability distribution on $\bbZ^d\times\bbN$. 
As in the one-dimensional case 
we can think of $\sff$ as of a probability distribution for a step $\sfU = (\sfX , \sfT ) $ of the effective 
$(d+1)$-dimensional random walk
\[
 \sfS_N = \sum_1^N \sfU_i .
\]
The distribution of $\lbr\sfS_n\rbr$ is governed by the product measure $\bbP$. 
We assume:
\begin{assumption}
\label{as:AR-distribution} 
Random  vector  $\sfU = \lb \sfX , \sfT\rb$ has a non-degenerate $(d+1)$-dimensional 
distribution. The random walk $\sfS_N$ is aperiodic (that is its support is not concentrated on 
a regular sub-lattice). 
\end{assumption}

The renewal array
$\lbr \sft (\sfx , n )\rbr $ is given by
\be 
\label{eq:ARenewal-d}
\sft (\sfx , 0 ) = \1_{\lbr \sfx =0\rbr}\quad{\rm and}\quad 
\sft (\sfx , n) = \sum_{m=1}^{n}\sum_{\sfy} \sff (\sfy , m )\sft (\sfx -\sfy , n-m ) .
\ee
Again, as in the one dimensional case \eqref{eq:ARenewal-p}, 
in probabilistic terms \eqref{eq:ARenewal-d} reads as:
\be 
\label{eq:ARenewal-p-d}
 \sft (\sfx , n) = \bbP\lb ~\exists ~N~:~ \sfS_N = (\sfx , n)\rb .
\ee
The renewal relation is inherited by  one-dimensional marginals: Set 
\[
\sff (n ) = \sum_\sfx\sff (\sfx , n)
\quad {\rm and}\quad 
 \sft (n ) = \sum_\sfx\sft (\sfx , n).
\]
Then, \eqref{eq:ARenewal} holds. 

We are going to explore the implications of the renewal relation \eqref{eq:ARenewal-d}
for a local limit analysis of conditional measures
\be 
\label{eq:ACond-m}
\bbQ_n \lb \sfx \rb = \frac{\sft ( \sfx , n )}{\sft (n )} .
\ee
{\bf Exponential tails.} 
Assume that there exists $\nu >0$, such that
\be 
\label{eq:AR-exp-d}
\sff (\sfx , \sfn ) \leqs {\rm e}^{-\nu (  \abs{\sfx} +n)  }, 
\ee
uniformly in $(\sfx , n )\in\bbZ^d\times\bbN$. In particular, \eqref{eq:AR-difbound} holds, and as a 
result we already have a sharp control over denominators in \eqref{eq:ACond-m}. 

Consider the following equation
\be 
\label{eq:ARen-Key-eq}
F (\xi , \lambda ) \df  \log \sum_{\sfx , n }{\rm e}^{\sfx\cdot \xi - \lambda n} \sff (\sfx , n )= 0.
\ee
Above $F :\bbC^d\times\bbC\mapsto \bbC$.
\begin{ex}
\label{ex:R.4}
Check that under \eqref{eq:AR-exp-d} there exists $\delta > 0$ such that 
$F$ is well defined and analytic on the disc $\bbD_\delta^{d+1}\subset \bbC^{d+1}$. 
\end{ex}
{\bf Shape theorem.} 
We shall assume that $\delta$ is sufficiently small. Then by the analytic implicit function 
theorem \cite{Kaup}, 
whose application is secured by Assumption~\ref{as:AR-distribution}, there is an analytic function
$\lambda  :\bbD_{\delta}^d\mapsto \bbC$ with such that for $(\xi , \lambda )\in \bbD_\delta^{d+1}$,
\be 
\label{eq:AR-lfunc}
 F (\xi , \lambda ) = 0\ \Leftrightarrow \  \lambda =\lambda (\xi ) . 
\ee
\smallskip

\noindent
For $\xi\in\bbD_\delta^d$ define:
\be 
\label{eq:AR-ftxi}
\sff_\xi (\sfx , n ) = \sff (\sfx , n ) {\rm e}^{\xi\cdot \sfx - \lambda (\xi )n }\quad{\rm and} \quad 
\sft_\xi (\sfx , n ) = \sft (\sfx , n ) {\rm e}^{\xi\cdot \sfx -\lambda (\xi )n } .
\ee
Evidently, the arrays $\lbr \sff_\xi (\sfx , n )\rbr$ and $\lbr \sft_\xi(\sfx , n )\rbr$ satisfy
\eqref{eq:ARenewal-d}. Also, under \eqref{eq:AR-exp-d},  $\sff_\xi (n)\df \sum_{\sfx } \sff_\xi (\sfx , n )$
 is well
defined for all $\abs{\xi} <\nu$.   
\begin{lem}
\label{lem:AR-fxi}  
There exists $\delta >0$ and $\epsilon >0$  such that 
\[
 \hat\sff_\xi (z)  \df \sum_n \sff_\xi (n) z^n ,
\]
satisfies Assumption~\ref{as:ARenewal-c} for all $\abs{\xi} <\delta$.  
\end{lem}
\begin{proof}
Conditions {\bf (a)} and {\bf (c)} are straightforward. In order to check {\bf (b)} note that it is
trivially satisfied at $\xi =0$. Which, by continuity means that we can fix $\epsilon >0$ such that
for any $\nu >0$ fixed,the equation
\be 
\label{eq:ARf-point}
 \hat\sff_\xi (z) = 1
\ee
 has no solutions in $\bbD_{1+\epsilon}\setminus \bbD_\nu (1)$ for all $\abs{\xi}<\delta$. 
However, the family of analytic functions $\lbr \hat\sff_\xi\rbr_{\abs{\xi }<\delta}$ is uniformly 
bounded on $\bar\bbD_\nu (1)$. Furthermore, for $\delta >0$ small the collection  of derivatives.
\be 
\label{eq:AR-muxi}
\lbr \hat\sff_\xi^\prime (1) \df  \mu (\xi ) \df \sum_n n\sff_\xi (n) \rbr_{\abs{\xi} <\delta}
\ee
is uniformly bounded away from zero. 
Therefore, there exist $\nu>0$ 
 and $\delta =\delta (\nu )> 0$, such that $z=1$ is the only solution of $\hat\sff_\xi  (z) = 1$ on 
$\bar\bbD_\nu (1)$ for all $\abs{\xi} <\delta$. 
\end{proof}
\begin{remark}
\label{rem:ConvF}
 Note that the restriction of  $F$ to $\bbR^{d+1}\cap \bbD_\delta^{d+1}$ is convex, 
 and it is monotone non-increasing in $\lambda$. Hence, the 
 restriction of $\lambda$ to $\bbR^d\cap \bbD_\delta^{d}$ is convex as well. Indeed, let 
 $\lambda_i =\lambda (\xi_i )$; $i=1,2$, for two vectors $\xi_1, \xi_2\in  \bbR^d\cap \bbD_\delta^{d}$. 
 From  convexity of level set $\lbr\lb \xi , \lambda \rb ~: ~ F (\xi , \lambda )\leq 0\rbr$, we
 infer that for any convex  combination $\xi =\alpha \xi_1 + (1-\alpha )\xi_2$
 \[
  F (\xi , \alpha\lambda_1 + (1-\alpha )\lambda_2 )\leq 0\ \Rightarrow\ \lambda (\xi )\leq 
  \alpha\lambda_1 +  (1-\alpha )\lambda_2 .
 \]
The term {\em shape theorem} comes from the fact that in applications function $\lambda$  frequently describes
local parametrization of the boundary of the appropriate limiting shape.
\end{remark}

\noindent 
{\bf Limit theorems}
Consider the canonical measure $\bbQ_n$ defined in \eqref{eq:ACond-m}. The following 
Proposition describes ballistic behaviour under $\bbQ_n$. 
\begin{prop}
 \label{prop:AR-Ballistic}
Under assumption on exponential tails  \eqref{eq:AR-exp-d}, 
\be 
\label{eq:AR-Ballistic}
\lim_{n\to\infty} \frac{1}{n}\bbQ_n\lb \sfX\rb =\lim_{n\to\infty}\frac{1}{n} 
\frac{\sum_{\sfx }\sft (\sfx , n)\sfx}{\sft (n )}
= \frac{\bbE \sfX}{\bbE \sfT} \df\sfv .
\ee
\end{prop}
\begin{proof}
 Note that
\[
\bbQ_n\lb \sfX \rb = 
\frac{\sum_{\sfx }\sft (\sfx , n)\sfx}{\sft (n )} = \nabla_\xi \log\lb
\sum_{\sfx }{\rm e}^{\xi\cdot\sfx }\sft (\sfx ,n )\rb (0) .
\]
For $\abs{\xi}$ small we can rely on  Lemma~\ref{lem:AR-lim-exp} and Lemma~\ref{lem:AR-fxi}
to conclude that 
\be 
\label{eq:AR-conv-to-muxi}
 {\rm e}^{-\lambda (\xi )n} \sum_{\sfx} {\rm e}^{\xi\cdot\sfx }\sft (\sfx ,n ) = \frac{1}{\mu (\xi )}
\lb 1 +\smo{(1+\epsilon)^{-n}}\rb ,
\ee
where $\mu (\xi )$ was defined in \eqref{eq:AR-muxi}. The convergence in 
\eqref{eq:AR-conv-to-muxi} is in a sense of analytic functions on $\bbD_\delta^d$ for $\delta$ 
small enough. As a result, the convergence, 
\be 
\label{eq:AR-Qn-exp}
{\rm e}^{-\lambda (\xi )n} \bbQ_n \lb {\rm e}^{\xi \cdot \sfX }\rb = \frac{\mu (0)}{\mu (\xi )}
\lb 1 +\smo{(1+\epsilon)^{-n}}\rb, 
\ee
is also in a sense of analytic functions on $\bbD_\delta^d$. 
Since $\lambda (0) = 0$, 
\be 
\label{eq:AR-vn}
 \sfv_n \df \frac{1}{n} \bbQ_n\lb \sfX\rb = \nabla\lambda (0) - \frac{1}{n}\nabla\log \mu (0) +
\smo{(1+\epsilon)^{-n}}. 
\ee
\eqref{eq:AR-Ballistic} follows, since  by \eqref{eq:ARen-Key-eq}
\be
\label{eq:AR-QnOfX}
\nabla\lambda (0) = \frac{\bbE \sfX}{\bbE \sfT} =\sfv .
\ee
\end{proof}
{\bf Integral central limit theorem.} 
For $\xi\in\bbC^d$ consider
\be 
\label{eq:AR-Sn-Func}
 \phi_n (\xi ) =\bbQ_n\lb {\rm exp}\lbr 
 i (\sfX  - n \sfv_n )\cdot\xi\rbr\rb .
\ee
For any $R$ fixed,  $\phi_n \lb \frac{\cdot}{\sqrt{n} } \rb$ is well defined on $\bbD_R^d$.
 By \eqref{eq:AR-Qn-exp} and \eqref{eq:AR-vn}, 
 \be 
 \label{eq:AR-Sn-CLT}
 \phi_n \lb \frac{\xi }{\sqrt{n} } \rb = {\rm exp}\lbr n\lb \lambda (\frac{i\xi}{\sqrt{n}}) 
 - \nabla\lambda (0)\cdot 
 \frac{i\xi}{\sqrt{n}} \rb + \bigo{\frac{R}{\sqrt{n }} }\rbr = 
 {\rm e}^{-\frac{1}{2}\Xi\xi\cdot\xi + \bigo{\frac{R}{\sqrt{n }} }}, 
 \ee
in the sense of analytic functions on  $\bbD_R^d$.
Above $\Xi \df {\rm Hess}(\lambda )$. We claim:
\begin{lem}
\label{lem:AR-Xi} 
$\Xi$ is a positive definite $d\times d$ matrix. 
\end{lem}
\begin{proof}
 Fix $\xi\in\bbR^d\setminus 0$ and consider \eqref{eq:AR-lfunc}:
\[
 \sum_{\sfx , n} f (\sfx , n) {\rm e}^{\epsilon\xi\cdot\sfx - n\lambda (\epsilon \xi )}\equiv 1 , 
\]
which holds for all $\abs{\epsilon} <\delta/\abs{\xi}$. The second order expansion gives:
\[
 {\rm Hess}\lambda (0)\, \xi\cdot\xi = 
 \frac{1}{\bbE \sfT}\bbE\lb \lb \sfX - \frac{\bbE \sfX}{\bbE \sfT} \sfT\rb\cdot\xi
\rb^2 .
\]
The claim of the lemma follows from the non-degeneracy Assumption~\ref{as:AR-distribution}. 
\end{proof}  
In view of Lemma~\ref{lem:AR-Xi}, 
 asymptotic formula \eqref{eq:AR-Sn-CLT} 
already implies the integral form of the CLT: The family of random vectors 
$\frac{1}{\sqrt{n}} \lb \sfX - n\sfv_n \rb$ weakly converges (under 
$\lbr\bbQ_n\rbr$) 
  to $\calN\lb 0 , \Xi \rb$. 
\smallskip 

\noindent
{\bf Local CLT.} 
$\phi_n$ is related to the characteristic function of $\sfX$ in the 
following way: For any $\theta\in\bbR^d$ , 
\[
 \phi_n (\theta ) = {\rm e}^{ - i  n\sfv_n\cdot\theta } \bbQ_n\lb {\rm e}^{i\theta\cdot \sfX}\rb 
 = \frac{{\rm e}^{ - in\sfv_n\cdot\theta } }{\sft ( n )} 
 \sum_{\sfx} \sft (\sfx , n) {\rm e}^{i\sfx\cdot\theta} .
\]
The complex array $\lbr \sft (\sfx , n) {\rm e}^{i\sfx\cdot\theta}\rbr$ is generated via multi-dimensional
renewal relation \eqref{eq:ARenewal-d} by $\lbr \sff (\sfx , n) {\rm e}^{i\sfx\cdot\theta}\rbr$. 
Since  $\lbr \sff (\sfx , n)\rbr$ is a non-degenerate probability distribution on $\bbZ^d$ with 
exponentially decaying tails, for any $\delta >0$ one can find 
$\kappa  =\kappa (\delta ) >0$ and $\epsilon = \epsilon (\delta )$, such that the 
array $\lbr \sff (\sfx , n) {\rm e}^{i\sfx\cdot\theta}\rbr$ 
satisfies Assumption~\ref{as:ARenewal-c-small} uniformly in $\abs{\theta} \geq \delta$. 
We conclude: 
\begin{lem}
\label{lem:AR-chi-control} 
For any $\delta >0$ there exists $c_\delta >0$ such that 
\be 
\label{eq:AR-chi-control}
\abs{\phi_n (\theta )}\leq {\rm e}^{-c_\delta n} ,
\ee
whenever $\theta \in\bbR^d$ satisfies $\abs{\theta }\geq \delta$. 
\end{lem}
One applies Lemma~\ref{lem:AR-chi-control} as follows: By the Fourier inversion formula, 
\be 
\label{eq:AR-Fourier}
\bbQ_n (\sfx ) = \frac{1}{(2\pi  )^d}\int_{\bbT^d} {\rm e}^{-i\theta \cdot \lb \sfx -n \sfv_n\rb}
\phi_n (\theta )\dd \theta .
\ee
Choose $\delta, \epsilon  >0$ small. The above integral splits into the sum
of three terms:
\be 
\label{eq:AR-threeterms}
\begin{split}
& \int {\rm e}^{-i\theta \cdot \lb \sfx -n \sfv_n\rb}
\phi_n (\theta )\dd \theta 
= 
\int_{A_n } {\rm e}^{-i\theta \cdot\lb \sfx -n \sfv_n\rb}
\phi_n (\theta )\dd \theta \\
&\quad +
\int_{B_n  } {\rm e}^{-i\theta \cdot\lb \sfx -n \sfv_n\rb}
\phi_n (\theta )\dd \theta +
\int_{\abs{\theta} \geq \delta } {\rm e}^{-i\theta \cdot\lb \sfx -n \sfv_n\rb}
\phi_n (\theta )\dd \theta .
\end{split}
\ee 
Above $A_n = \lbr\theta ~:~\abs{\theta} < n^{-\frac{1}{2}+\epsilon}\rbr$ and 
$B_n = \lbr\theta ~:~n^{-\frac{1}{2}+\epsilon} \leq \abs{\theta} <\delta\rbr$. 
The third integral is negligible by Lemma~\ref{lem:AR-chi-control}. In order to control the second
integral (over $B_n$) note that for $\delta$ small enough 
\eqref{eq:AR-Qn-exp} applies, 
and hence, in view of positive definiteness of $\Xi$, 
\[
 \abs{\int_{ B_n} {\rm e}^{-i\theta \cdot\lb \sfx -n \sfv_n\rb}
\phi_n (\theta )\dd \theta } 
 \leq \int_{ B_n }
{\rm e}^{-\frac{n}{4}\Xi\, \theta\cdot\theta }\dd\theta  .
\]
The first integral in \eqref{eq:AR-threeterms} gives  local CLT asymptotics uniformly in 
$\abs{\sfx - n\sfv_n} = \smo{n^{\frac{1}{2} - (d+1 )\epsilon} }$.  Namely, for such $\sfx$-s
\[
 \int_{\abs{\theta} < n^{-\frac{1}{2}+\epsilon} } {\rm e}^{-i\theta \cdot \lb \sfx -n \sfv_n\rb }
 \phi_n (\theta )
\dd\theta  = 
\int_{\abs{\theta} < n^{-\frac{1}{2}+\epsilon} } 
\phi_n (\theta ) 
 \dd\theta + 
\smo{
\frac{1}{\sqrt{n^{d}}}
} .
\]
As in \eqref{eq:AR-Sn-CLT}, 
\[
\phi_n (\theta )  = {\rm e}^{-\frac{n}{2}\Xi\, \theta\cdot \theta + \bigo{n \abs{\theta}^3}} ,
\]
uniformly in $\abs{\theta} < n^{-\frac{1}{2}+\epsilon}$. 
We have proved: 
\begin{prop}
 \label{prop:loc-CLT}
For any fixed $\epsilon >0$ the asymptotic relation:
\be 
\label{eq:loc-CLT}
\bbQ_n (\sfx )  = \frac{1}{\sqrt{(2\pi n )^d {\rm det}\Xi}}\lb 1 +\smo{1}\rb, 
\ee
holds uniformly in $\abs{\sfx  -n\sfv_n } \leqs n^{\frac{1}{2} -  \epsilon}$. 
\end{prop}
In fact, as it will become clear from local large deviations estimates below, 
it would be enough to state \eqref{eq:loc-CLT} only for $\abs{\sfx - n\sfv_n}\leq 1$. 
\smallskip 

\noindent
{\bf Local large deviations estimates.}  Assumption~\eqref{eq:AR-exp-d} implies that the family of measures
$\lbr \bbQ_n\rbr$ is exponentially tight. Furthermore, by the very definition of $\sft$ in 
\eqref{eq:ARenewal-d}
\[
 \sft (\sfx +\sfy , n+ m ) \geq \sft (\sfx ,n )\sft (\sfy , m ) . 
\]
Hence, by the sub-additivity argument the function 
\be 
\label{eq:AR-LD-func}
J (\sfu ) = -\lim_{n\to\infty}\frac{1}{n}\log\sft (n , \lfloor n\sfu\rfloor )
\ee
is well defined and convex on $\bbR^d$. By the renewal theorem \eqref{eq:ARen-lim}, 
\[
 J (\sfu ) = -\lim_{n\to\infty}\frac{1}{n}\log\bbQ_n  (\lfloor n\sfu\rfloor ) .
\]
Consequently, $\lbr \bbQ_n\rbr$ satisfies the large deviation principle with $J$. 
\smallskip 

A large deviation result states what it states.  Obviously, $J$ in \eqref{eq:AR-LD-func} is 
non-negative, and $\min J = J (\sfv ) =0$, where $\sfv$ was defined in 
 \eqref{eq:AR-Ballistic}. 
We shall show that $J$ has a quadratic minimum on $\bbB_\kappa^d (\sfv )$, and prove a 
local LD asymptotic relation for any $\sfu \in \bbB_\kappa^d (\sfv )$.  

Recall that $\lambda$ is analytic (and convex )  on a (real) ball
$\bbB_\delta^d$. Since, as we already know by Lemma\ref{lem:AR-Xi}, ${\rm Hess}\lb \lambda\rb (0 )$
is non-degenerate, and since $\nabla \lambda (0) = \sfv$, there exists $\kappa >0$ such that
\be 
\label{eq:AR-kappa}
\bbB_\kappa^d (\sfv )\subset \nabla\lambda\big|_{\bbB_\delta^d} .
\ee

Let $\sfu\in \bbB_\kappa^d (\sfv )$. Set $\sfu_n = \lfloor n\sfu \rfloor /n$ 
and choose $\xi_n\in \bbB_\delta^d$ such that 
$\sfu_n = \nabla\lambda (\xi_n )$. By \eqref{eq:AR-kappa} such $\xi_n$ exists (at least for all $n$ sufficiently large), 
and, for it is unique
by the implicit function theorem. Recall how we defined tilted function $\sft_{\xi_n}$ in 
\eqref{eq:AR-ftxi}. Then, 
\be 
\label{eq:AR-xi-tilt}
\sft (n, \lfloor n\sfu \rfloor ) = {\rm e}^{n\lambda  (\xi_n  ) - \xi_n\cdot  \lfloor n\sfu \rfloor  }
\sft_{\xi_n}  (n, \lfloor n\sfu \rfloor )
\ee
The term $\sft_{\xi_n}  (n, \lfloor n\sfu \rfloor )$ obeys  uniform 
sharp CLT asymptotics  \eqref{eq:loc-CLT} with $\Xi (\xi_n ) = {\rm Hess} (\lambda )(\xi_n )$.
The term 
\[
J (\sfu_n ) = \xi_n\cdot \sfu_n  - \lambda (\xi_n )
\]
is quadratic. Indeed, for any $\eta\in\bbB_\kappa^d (\sfv )$ and  $\sfw =\nabla\lambda (\eta  )$, one, 
using  $\lambda (0 )= 0$,  can rewrite:
\be 
\label{eq:AR-expansion}
\begin{split}
 \eta   \cdot\sfw -\lambda (\eta ) &= \lambda (0 ) - \lambda (\eta  ) - (-\eta )\cdot\nabla\lambda (\eta ) 
 \\
 &= \lbr \int_0^1\int_0^s {\rm Hess}(\lambda ) ((1-\tau )\eta )\dd\tau\dd s\rbr \eta\cdot \eta, 
 \end{split}
\ee
and rely on non-degeneracy of ${\rm Hess} (\lambda )$ on $\bbB_\kappa^d$. Incidentally, we have checked
that on $\bbB_\kappa^d (\sfv )$ the function   $J = \lambda^*$ is real analytic 
with ${\rm Hess}( J )(\sfv )$ being positive definite. 

The local limit estimates we have derived reads as: Recall notation $\sfu_n = \lfloor n\sfu \rfloor/n$ and
$\xi_n$ being defined via $\sfu_n = \nabla\lambda ( \xi_n )$. Then, 
\be 
\label{eq:AR-LocLim}
\bbQ_n \lb \lfloor n\sfu \rfloor \rb = 
\frac{\mu ( 0 )}{\mu (\xi_n ) \sqrt{(2\pi )^d{\rm det}\, \Xi (\xi_n )}} {\rm e}^{-n J (\sfu_n )}
\lb 1 +\smo{1} \rb , 
\ee
uniformly in  $\sfu \in \bbB_\kappa^d (\sfv )$

\subsection{Ballistic phase of annealed polymers} 
\label{ssec:BalPhaseA}
Recall that the reference polymer weights $\Wd$ are given by \eqref{eq:IN-aweight}. 
$\Phi$ is the self-interaction  potential \eqref{eq:Phi}  which satisfies the 
attractivity  condition 
\eqref{eq:attractive}.  
Let $h\not\in \Knot$ and, accordingly, $h\in\pKl$ with  $\lambda = \lambda (h ) >0$. 
We shall 
consider the normalized  weights
\be 
\label{eq:L4-Weight-hl}
\Wdn{ h , \lambda } (\gamma ) = {\rm e}^{h\cdot\sfX (\gamma ) - 
\lambda \abs{\gamma}}\Wd (\gamma ) .
\ee
These weight are normalized for the following reason: As before define 
\be 
\label{eq:L3-Famhl-d}
\calP_{\sfx}  =\lbr \gamma~:~ \sfX (\gamma ) = \sfx\rbr\quad 
\calP_{n}  =\lbr \gamma~:~ \abs{\gamma } = n\rbr\quad {\rm and}\quad 
\calP_{\sfx , n} =  \calP_{\sfx}\cap \calP_{n} .
\ee
Then, 
\be
\label{eq:L3-Weight-Famhl-d}
\sum_{\gamma\in\calP_{\sfx}} \Wdn{ h , \lambda } (\gamma )
= {\rm e}^{h\cdot \sfx }G_\lambda (\sfx ) 
 \asymp {\rm e}^{h\cdot\sfx - \tl (\sfx )}
 \ 
{\rm and }\ 
\sum_{\gamma\in\calP_{n}} \Wdn{ h , \lambda } (\gamma ) = {\rm e}^{-n\lambda} Z_n (h ) \asymp 1 .
\ee
If $h\cdot \sfx = \tl (\sfx )$, then the first term in \eqref{eq:L3-Weight-Famhl-d} 
is also of order $1$. More generally, let us define the following crucial notion:

\noindent 
{\bf Surcharge function. }
For any $\sfx\in\bbR^d$ 
define the surcharge function
\be 
\label{eq:Ans-function}
\sfs_h (\sfx ) = \tl (\sfx ) -
h\cdot\sfx \geq 0 .
\ee
In view of  \eqref{eq:TH-GlHl-bound}
 the first  of the estimates in \eqref{eq:L3-Weight-Famhl-d} could be upgraded 
as follows (see \eqref{eq:L4-EstimateRW-2} for the definition of $\sfA_d$ )
\be 
\label{eq:43-Weight-surcharge}
\Wdn{ h , \lambda } (\sfx )\df 
\sum_{\gamma\in\calP_{\sfx}} \Wdn{ h , \lambda } (\gamma ) 
\asymp {\rm e}^{-\frs_h (\sfx )}\  {\rm and, moreover,}
\ 
\Wdn{ h , \lambda } (\sfx  ) \leqs   \sfA_d (k_0\abs{\sfx}_1 )
{\rm e}^{-\frs_h (\sfx )} .
\ee
{\bf Surcharge cone.} 
Let us say that $\calY_1$ is a $\delta_1$-surcharge cone with respect to $h$ if:
\smallskip 

\noindent
(a) $\calY_1$ is a positive cone (meaning that its opening is strictly less than $\pi$) 
and it   contains a lattice direction $\pm\sfe_k$ in its interiour. 

\noindent
(b) For any $\sfx\not\in \calY_1$ the surcharge function $\frs$ satisfies 
\be 
\label{eq:L3-cone-Y}
\frs (\sfx ) =  \tl (\sfx ) - h\cdot\sfx > \delta_1\tl (\sfx ) .
\ee 
For the rest of this section we shall fix $\delta_1 \in (0,1)$ and a $\delta_1$-surcharge cone
$\calY_1$  with respect to $h$.
\smallskip 

\noindent
{\bf Factorization bound.} 
Assume that
the path $\gamma$ can be represented as a concatenation,
\be 
\label{gleta}
 \gamma\, =\, \gamma_0\circ\eta_1\circ\gamma_1\circ\dots\circ \gamma_m\circ\eta_{m+1} , 
\ee
such that paths $\gamma_\ell = \lb\sfu_\ell , \dots , \sfv_{\ell+1}\rb$ satisfy the 
following two properties:
\smallskip 

\noindent
{\bf (P1)} $\gamma_i$ is disjoint from $\gamma_j$ for all $i> j$. 

\noindent
{\bf (P2)} For any $i$ the local time 
$\ell_{\gamma_i} (\sfv_{i+1} )=1$.
\smallskip 

\noindent 
By \eqref{eq:L4-phi-1} and {\bf (P1)}, 
\[
 \Phi_\beta (\gamma ) \geq \Phi_\beta  (\gamma_1\cup\dots\cup\gamma_m ) = \sum_\ell \Phi_\beta 
 (\gamma_\ell ).
\]
Consequently,
\be 
\label{Wfbound}
\Wd
\lb\gamma \rb {\rm e}^{-\lambda\abs{\gamma}} \, \leq\, 
\prod_{l=1}^m \Wd
 (\gamma_\ell ) {\rm e}^{-\lambda\abs{\gamma_\ell} }
\cdot\prod_{k=1}^m {\rm e}^{-\lambda |\eta_k|} .
\ee
Fixing end points $\sfu_1, \sfv_2, \sfu_2, \dots$ and paths $\eta_\ell$ 
in \eqref{gleta}, 
and summing up with respect to all paths $\gamma_1, \dots, \gamma_m $ 
(with $\gamma_\ell = \lb\sfu_\ell , \dots , \sfv_{\ell+1}\rb$) satisfying 
properties {\bf (P1)} and {\bf (P2)} above we derive the following upper bound:
\be 
\label{eq:AN-FactBound}
\sum_{\gamma_1 , \dots , \gamma_m} 
\Wd
\lb\gamma \rb {\rm e}^{-\lambda\abs{\gamma}} \, \leq\, \prod_1^m 
H_\lambda (\sfv_{\ell +1} - \sfu_\ell ) {\rm e}^{-\lambda \sum\abs{\eta_\ell}} 
\leq {\rm e}^{-\sum\tl (\sfv_{\ell +1} - \sfu_\ell ) -\lambda \sum\abs{\eta_\ell}} .
\ee
Let us proceed with describing our algorithm to construct 
representation \eqref{gleta} with properties {\bf (P1)} and {\bf (B2)} for 
any path $\gamma\in \calP_\sfx$. 
\smallskip 

\noindent
{\bf Construction of skeletons.} Skeletons $\hat\gamma_K$ are constructed 
as a collection 
$\hat{\gamma}_K =  [ \frt_K , \frh_K]$, where $\frt_K$ is the trunk and 
$\frh_K$ is the set of hairs of $\hat{\gamma}_K $. 
Let $\gamma\in\calP_x$ and choose a scale $K$. In the sequel  we use 
$\UlK{K} = \lbr \sfu ~: ~\tl (\sfu )\leq K\rbr$ denote  the  ball of 
radius $K$ with respect to $\tl$ (note  that since, in general, 
$\tl (\sfy )\neq \tl (-\sfy )$, it does not have to be  a distance). 
Recall that $R$ denotes the range of the underlying random walk. Choose
$r= r_\lambda   = \min\lbr s ~:~ \bbB_R^d\subset \UlK{s}\rbr$, where as before 
$\bbB_R^d$ is the Euclidean
ball of radius $R$. Let us first explain decomposition \eqref{gleta} and construction 
of trunks (see Figure~\ref{fig:trunk}). 
\smallskip

\noindent
\step{0}. Set  $\sfu_0 =0$, $\tau_0 = 0$ and $\frt_0 =\lbr u_0\rbr$.  Go to
\step{1}. 
\smallskip

\noindent
\step{(l+1)}
If $\lb\gamma (\tau_l) ,\dots \gamma (n)\rb\subseteq \UlK{K} (u_l )$ then 
set $\sigma_{l+1} = n$ and stop.
Otherwise, 
 define 
\begin{equation*}
 \begin{split}
&\sigma_{l+1} = \min\lbr i >\tau_l~:~ 
\gamma (i)\not\in \UlK{K} (\sfu_l )\rbr\\
&\qquad\text{ and}\\
&\tau_{l+1} = 1+ \max 
\lbr i >\tau_l~:~ 
\gamma (i) \in \UlK{K+r } (\sfu_l )\rbr .
\end{split}
\end{equation*}
 Set $\sfv_{l+1} = \gamma (\sigma_{l+1})$ and $\sfu_{l+1} = \gamma (\tau_{l+1})$. 
Update $\frt_K = \frt_K\cup\lbr \sfu_{l+1}\rbr$ and go to \step{(l+2)} \qed

\begin{figure}[h]
\scalebox{.3}{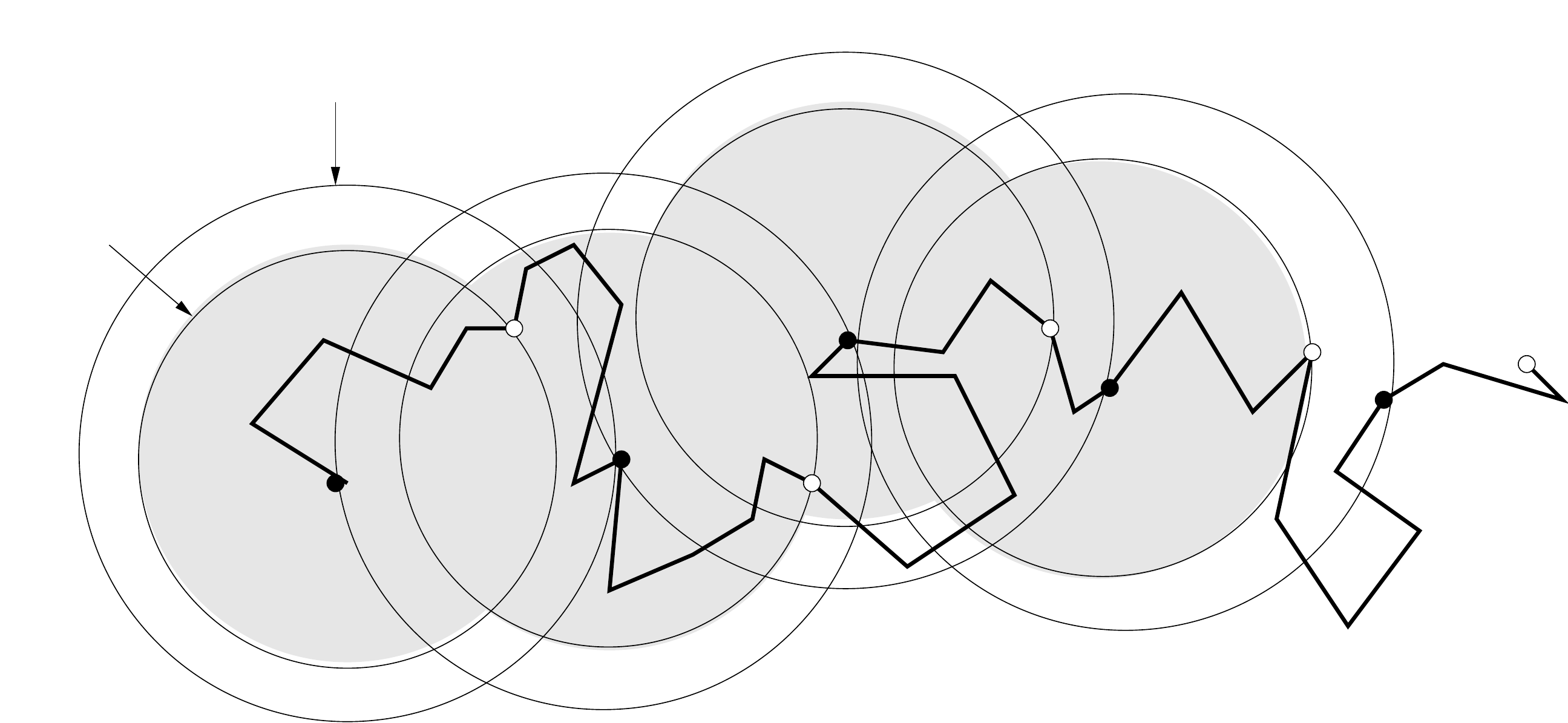}
%
%
\caption{Decomposition \eqref{gleta} and construction of the trunk $\frt_K =\lbr \sfu_1, \dots , \sfu_4\rbr$ 
of the skeleton.}
\label{fig:trunk}       
\end{figure}

\noindent
Clearly the above algorithm leads to a decomposition
of $\gamma$ as in \eqref{gleta} with 
\[
\gamma_l\, =\, \lb \gamma(\tau_l ),\dots ,\gamma (\sigma_{l+1} )\rb\quad
\text{and}\quad \eta_l = \lb \gamma (\sigma_l ),\dots ,\gamma (\tau_l )\rb , 
\]
and with $\gamma_1, \gamma_2 , \dots$ satisfying conditions {\bf (P1)} and {\bf (P2)}.

The set $\frt_K$ is called the trunk of the skeleton $\hat{\gamma}_K$ 
of $\gamma$ on $K$-th scale. 
The hairs $\frh_K$ of $\hat{\gamma}_K$ take into account those $\eta_l$-s
which are long on $K$-th scale. Recall that $\eta_l : \sfv_l\mapsto \sfu_{l}$.
It is equivalent, but, 
since eventually we want to keep track of vertices from the trunk $\frt_K$, 
more convenient to think about $\eta_l$ as of 
a reversed path from $\sfu_l$ to $\sfv_l$. 
Then the $l$-th hair $\frh_K^l = \frh_K [\eta_l ]$ of 
$\gamma$ is constructed as follows: 

If $\eta_l\subseteq \UlK{K} (\sfu_l )$ then $\frh_K^l =
\emptyset
$. 
Otherwise, set $\sfu = \sfu_l$, $\sfv = \sfv_l$, $\eta = \eta_l$, $m =\abs{\eta}$, 
and 
proceed with the following algorithm (see Figure~\ref{fig:hair}): 
\begin{figure}[h]
\sidecaption
\scalebox{.3}{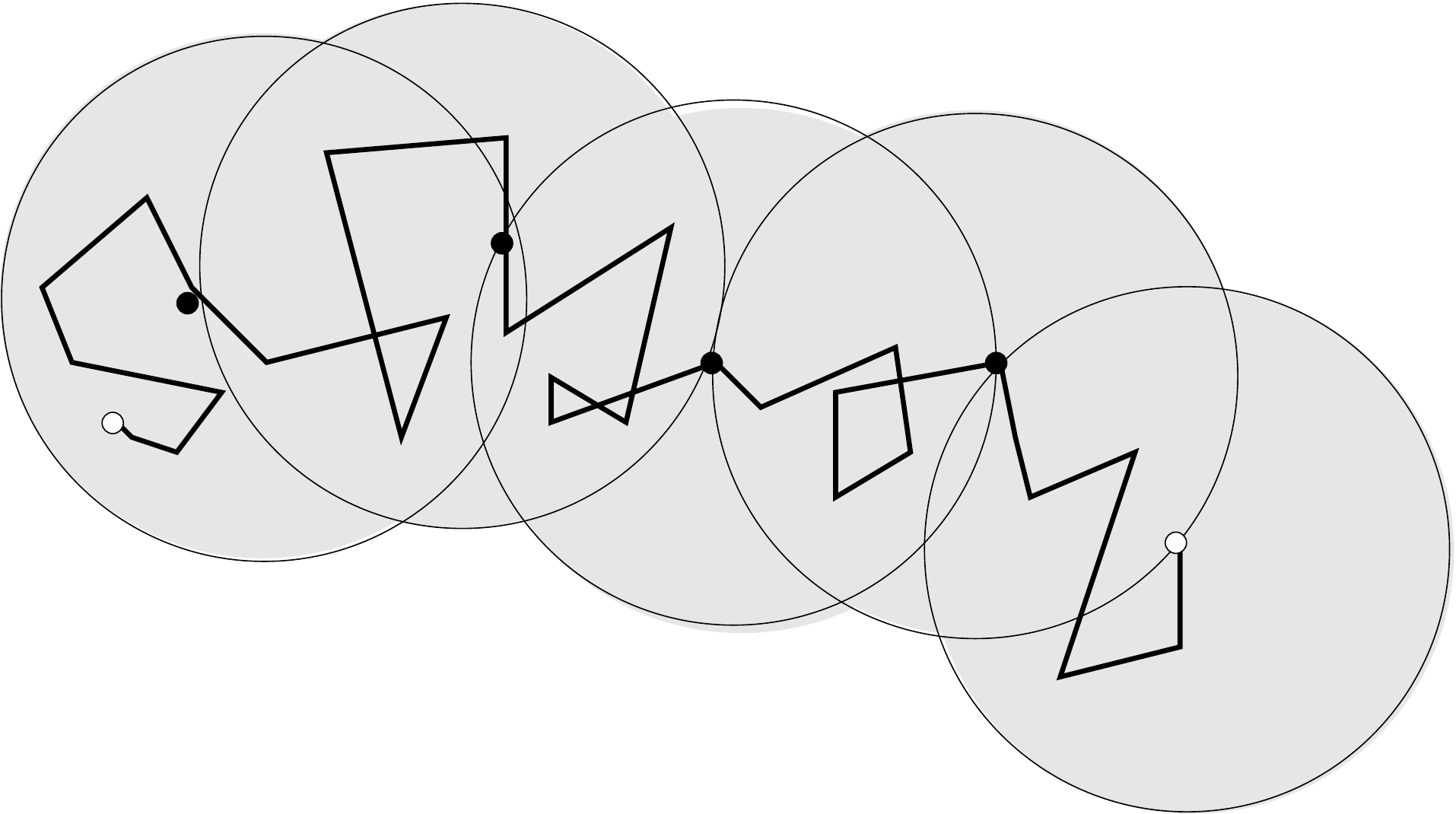}
%
%
\caption{Construction of a hair $\frh_K = \lbr \sfw_1, \dots , \sfw_4\rbr$.  }
\label{fig:hair}       
\end{figure}
\smallskip

\noindent
\step{0}. Set $\sfw_0 =\sfu$, $\tau_0 = 0$ and $\frh_K [\eta ] = \emptyset$. Go 
to \step{1}.
\smallskip

\noindent
\step{(l+1)}. If $\lb \eta (\tau_l ),\dots ,\eta (m)\rb \subseteq \UlK{K} (\sfu_l )$
then stop. Otherwise set
\[
\tau_{l+1}\, =\, \min\lbr j> \tau_l~:~ \eta (j)\not\in \UlK{K} (\sfw_l )\rbr .
\]
Define $\sfw_{l+1} = \eta  (\tau_{l+1} )$, update 
$\frh_K  = \frh_K \cup \lbr \sfw_{l+1}\rbr$ and go to 
\step{(l+2)}. \qed
\smallskip 

\noindent
{\bf Control of $\Wdn{h , \lambda} \lb \hat\gamma_K\rb$.} 
 In the super-critical case $\lambda >0$, and hairs could be  controlled in a crude fashion
via comparison with an underlying walk killed at rate $\lambda$. 
\begin{ex}
\label{ex:L4-E2}
 There exists 
$\epsilon =\epsilon \lb \Pd , \lambda \rb >0$, such that
\be 
\label{eq:L4-E2}
\sum_{\gamma\in\calP_\sfu } {\rm e}^{-\epsilon\abs{\gamma}}\Pd (\gamma ) 
\leqs {\rm e}^{-\epsilon\tl (\sfu )}
\ee
uniformly in $\sfu$. 
\end{ex}
Let $\hat{\gamma}_K = \lb \frt_K , \frh_K\rb$ be a skeleton. We shall carefully control  the 
geometry of the trunk 
$\frt_K = \lb\sfu_0 ,\dots, \sfu_{N}\rb $. As far as the hairs are considered, 
we shall rely on \eqref{eq:L4-E2}, and the only 
thing we shall
control is $\#\lb \frh_K\rb$-the total number of $K$-increments (in $\tl$-metrics). 

Let us fix $\sfu_0, \sfv_1 , \sfu_2, \dots $. By \eqref{eq:AN-FactBound} the 
contribution coming from $\gamma_\ell :\sfu_\ell \to \sfv_{\ell+1}$ paths is 
bounded above by the product of ${\rm e}^{-\tl (\sfv_{\ell+1} -\sfu_\ell )}$. 
Now, by construction $\sfv_{\ell+1} \in \UlK{K+r_\lambda }(\sfu_\ell )
\setminus\UlK{K} (\sfu_\ell )$, which 
means that $\tl (\sfv_{\ell+1} -\sfu_\ell ) \in [K, K+r_\lambda ]$. 
On the other hand if  $\sfu_{\ell +1}$ is defined, then $\sfu_{\ell +1} \in 
\UlK{K+2r_\lambda }(\sfu_\ell )\setminus \UlK{K+r_\lambda }(\sfu_\ell )
$
Hence, 
$\tl (\sfu_{\ell+1} -\sfu_\ell ) \leq \tl (\sfv_{\ell+1} -\sfu_\ell ) +2r_\lambda$.
There are $\leqs RK^{d-1}$ possible exit points from $\UlK{K} (\sfv_\ell )$-balls which 
are possible candidates for $\sfv_{\ell+1}$ vertices. 

Consequently, \eqref{eq:AN-FactBound}
and \eqref{eq:L4-E2} imply that  
there exists $c = c (\lambda, \beta )>0$ such 
that the following happens: Let $\hat\gamma_K = \lb \frt_K , \frh_K\rb $ be 
a skeleton with trunk $\frt_K = \lb \sfu_0, \dots, \sfu_N\rb$, and 
$\frh_K = \lbr\frh_K^\ell\rbr$ collection of hairs. Notation $\gamma\sim\hat{\gamma}_K$
means that $\hat{\gamma}_K$ is the $K$-skeleton of $\gamma$ in the sense of 
the two algorithms above.  Then, 
\be 
\label{eq:L4-skel-bound-1}
\sum_{\gamma\sim\hat{\gamma}_K} \Wdn{h , \lambda} (\gamma )
\leqs {\rm exp} \lbr
h\cdot\sfx -\sum_{\ell=0}^{N}\tl (\sfu_{\ell +1} - \sfu_\ell ) - 
\epsilon K \# (\frh_K ) + cN\log K 
\rbr , 
\ee
uniformly in $\sfx$, large enough scales $K$ and 
skeletons $\hat{\gamma}_K$. 
\smallskip 

\noindent 
{\bf Kesten's bound on the number of forests.} A forest $\calF_N$ is a collection of $N$ rooted
trees $\calF_N = \lb\calT_1, \dots,\calT_N\rb$ of forward branching ratio at most $\sfb$. The tree $\calF_\ell$ 
is rooted at $\sfu_\ell$. Given $M\in \bbN$ we wish to derive an upper bound on 
$\# (M , N)$-number of all
forests $\calF_N$ satisfying $\abs{\calF_N}=M$. Above  $\abs{\calF_N}=M$ is the number of vertices of 
$\calF_N$ different from the roots $\sfu_1, \dots , \sfu_N$. Let $\bbP_p^N$ be 
the product percolation 
measure on $\times\sfT_\ell^\sfb$ at the percolation value $p$, 
where $\sfT^\sfb$ is the set of (edge) percolation configurations on the rooted 
tree of branching ration $\sfb$. 
In this way $\calT_\ell$ is viewed as a connected 
component of $\sfu_\ell$. Clearly, 
\be 
\label{eq:L4-probBound-F}
\bbP_p^N\lb \abs{\calF_N}=M\rb \leq 1 .
\ee
Each realization of $\calF_N$ with $\abs{\calF_N}=M$ has probability which is bounded below by 
$p^M (1-p )^{b (N+M )}$. Therefore, \eqref{eq:L4-probBound-F} implies: 
\be 
\label{eq:L4-Kesten}
\# (M , N) \leq \lb \max_{p\in [0,1]} p^M (1-p )^{b (N+M )}\rb^{-1} .
\ee
For $x \in [0,1]$, 
\[
 \log (1-x ) = -\int_0^x \frac{\dd t}{1-t } \geq - \frac{x}{1-x }
\]
Choosing  $p = \frac{1}{b}$ we, therefore,  infer from \eqref{eq:L4-Kesten} 
\be 
\label{eq:L4-Bound-F}
\# (M , N) \leq {\rm e}^{M \log b   + (N+M)\frac{b}{b-1}} .
\ee
{\bf Surcharge cost of a skeleton.} 
A substitution of \eqref{eq:L4-Bound-F} with $b\eqvs  RK^{d-1}$ into \eqref{eq:L4-skel-bound-1} implies:
 There exist $\epsilon^\prime , c^\prime >0$, such that
\be 
\label{eq:L4-skel-bound}
\sumtwo{\gamma\sim{\frt}_K}{\# (\frh_K ) \geq M}  \Wdn{h , \lambda} (\gamma )
\leqs {\rm exp} \lbr
h\cdot\sfx -\sum_{\ell=0}^{N}\tl (\sfu_{\ell +1} - \sfu_\ell ) - \epsilon^\prime K  M + c^\prime N\log K 
\rbr .
\ee
uniformly in $\sfx$, scales $K$, trunks $\frt_K$ and $M\in \bbN$.

With \eqref{eq:L4-skel-bound} in mind let us define the surcharge cost of 
a skeleton $\hat{\gamma}_K=[\frt_K  , \frh_K]$ as follows: Recall the notation  
$\frs_h (\sfu ) = \tl (\sfu ) - h\cdot\sfu$. Then,    
\be 
\label{eq:L4-scost-skeleton}
\frs_h \lb \hat{\gamma}_K \rb \df  \sum_{\ell = 1}^{N} \frs_h (\sfu_\ell - \sfu_{\ell -1} ) + \epsilon^\prime K
\#\lb \frh_K\rb .
\ee
Above the trunk $\frt_K = \lb\sfu_0, \dots \sfu_N\rb$.  We 
conclude: 
\begin{lem}
\label{lem:L4-skel-s}
For any $\epsilon >0$ there exists a scale $K_0$, such that
\be 
\label{eq:L4-skel-s}
 \sum_{\frs_h (\hat\gamma_K ) >2\epsilon \abs{\sfx }} \Wdn{h, \lambda} \lb \hat\gamma_K\rb \leqs 
{\rm e}^{-\epsilon\abs{\sfx}} , 
\ee
for all $K\geq K_0$ fixed and uniformly in  $h\in\pKl$, $\sfx\in \bbZ^d$. By convention  
the summation above is with respect to
skeletons $\hat\gamma_K$ of paths $\gamma\in\calP_\sfx$.
\end{lem}
{\bf Cone points of skeletons.}
Recall the definition of the surcharge cone $\calY_1$ in \eqref{eq:L3-cone-Y}. 
Fix $\delta_2 \in (\delta_1 , 1)$ and $\delta_3 \in (\delta_2 , 1)$ , 
and define  enlargements $\calY_i$ of $\calY_1$ as follows: For $i=2,3$, 
\be 
\label{eq:L4-cone-Yi}
\calY_i = \lbr \sfx~:~ \frs (\sfx )\leq \delta_i \tl (\sfx )\rbr
\ee
Clearly, $\calY_i$-s are still  positive cones for $i=2,3$. 
Let $A^0 \df A\setminus 0$. Then, by construction, 
 $\calY^0_1\subset {\rm int}\lb \calY_2^0 \rb $ and 
 $\calY_2^0\subset{\rm int}\lb \calY_3^0 \rb $. 
\smallskip 

\noindent
Consider a skeleton $\hat\gamma_K = [\frt_K , \frh_K ]$. 
Let us say that 
a vertex of the trunk 
$\sfu_\ell\in \frt_K = \lb \sfu_0, \dots, \sfu_\ell, \dots , \sfu_{m+1}\rb$ is a
 $\calY_2$-cone 
point of  the skeleton $\hat\gamma_K$ if 
\be 
\label{eq:L4-cone-point-skeleton}
\hat\gamma_K \subset \lb \sfu_\ell - \calY_2\rb\cup \lb \sfu_\ell + \calY_2\rb .
\ee
Let $\#_{\rm bc}\lb \hat\gamma_K\rb$ be the total number of vertices of $\hat\gamma_K$ which are 
{\em not} cone points. 
\begin{prop}
\label{prop:L4-skel-cbad}
There exists $\nu_2 >0$ such that the following happens: For any $\epsilon >0$ there exists a 
scale $K_0$, such that
\be 
\label{eq:L4-skel-cbad}
 \sum_{\#_{\rm bc} (\hat\gamma_K ) >\epsilon \frac{\abs{\sfx }}{K}} \Wdn{h, \lambda} \lb \hat\gamma_K\rb \leqs 
{\rm e}^{-\nu_2 \abs{\sfx}} , 
\ee
for all $K\geq K_0$ fixed and uniformly in  $\sfx\in \bbZ^d$.
\end{prop}

A proof of Proposition~\ref{prop:L4-skel-cbad} contains
several steps  and we refer to \cite{IoffeVelenik-Annealed} for more details. 
 First of all we show that, 
up to exponentially small corrections, 
 most of the vertices of the trunk
$\frt_K$ are $\calY_1$-cone points of the latter. 
We shall
end up with $N^\prime \eqvs \frac{\abs{\sfx}}{K}$ $\calY_1$-cone points of $\frt_K$. 

Any $\calY_1$-cone point of the trunk $\frt_K$ is evidently 
also a $\calY_2$-cone point of the latter.
On the other hand,  in view of 
\eqref{eq:L4-skel-s},  we can restrict attention to 
$\# \lb \frh_K\rb \leq \frac{2\epsilon}{\epsilon^\prime}
 \frac{\abs{\sfx}}{K}$. For $\epsilon\ll \epsilon^\prime$ the total 
 number of leaves $\# \lb \frh_K\rb$ is only
a small fraction of $N^\prime$. It is clear that an addition of a leave 
is capable of blocking at 
most $c= c(\delta_1 , \delta_2 )$ $\calY_1$-cone points of $\frt_K$ 
from being a $\calY_2$-cone point of the 
whole skeleton $\hat\gamma_K$. It is important that the above 
geometric constant $c = c(\delta_1 , \delta_2 )$
 does not depend on 
the running scale $K$. Consequently, under the reduction we are working with on large enough 
scales $K$, there are just not 
 enough leaves to block all (and actually a small fraction of) $\calY_1$-cone 
 points of $\frt_K$ from 
being a $\calY_2$-cone point of the 
whole skeleton $\hat\gamma_K$.
\smallskip 

\noindent 
{\bf Cone points of paths $\gamma\in\calP_\sfx$.} 
  Let us say that $\sfu_\ell\in \gamma  = \lb \sfu_0, \dots, \sfu_\ell, \dots , 
\sfu_{n}\rb\in\calP_\sfx$ is a cone 
point of  $\gamma $ if $0 <\ell < n$ and 
\begin{figure}[t]
\sidecaption
\scalebox{.25}{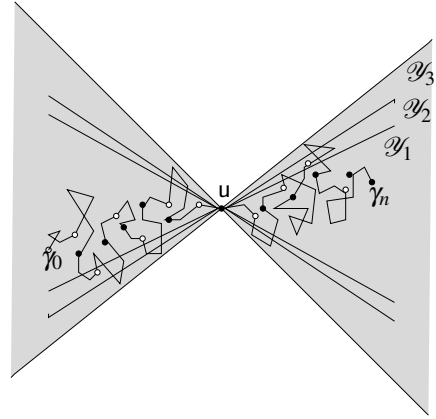}
%
%
\caption{$\sfu$ is a cone point of the  path $\gamma = \lb \gamma_0, \dots , \gamma_n\rb$. 
Black vertices belong to the trunk. Paths leading from white vertices
to black vertices give rise to hairs..}
\label{fig:cone-points}       
\end{figure}
\be 
\label{eq:L4-cone-point-path}
\gamma \subset \lb \sfu_\ell - \calY_3\rb\cup \lb \sfu_\ell + \calY_3\rb .
\ee
Let $\#_{\rm cone} (\gamma )$ be the total number of the cone points of $\gamma$. 
\begin{prop}
\label{prop:L4-cp-paths}
There exist $\epsilon >0$ and  $\nu >0$ such that: 
\be 
\label{eq:L4-cp-paths}
 \sum_{\#_{\rm cone} (\gamma  ) <  \epsilon\abs{\sfx }}  \Wdn{h, \lambda} \lb \gamma \rb \leqs 
{\rm e}^{-\nu  \abs{\sfx}} , 
\ee
uniformly in  $\sfx\in \bbZ^d$. 
\end{prop}
As before, we refer to \cite{IoffeVelenik-Annealed} for details of the proof. Construction of 
cone points is depicted on Figure~\ref{fig:cone-points}
\smallskip 

\noindent
{\bf Irreducible decomposition of paths in $\calP_\sfx$, $\calP_n$ and $\calP_{\sfx , n}$.}
In the sequel we set $\calY \df\calY_3$, where $\calY_3$ is the positive cone 
in Proposition~\ref{prop:L4-cp-paths}. 
A path $\gamma = \lb \sfu_0, \dots , \sfu_n\rb$ is said to be irreducible if it does not contain 
$\calY$-cone points. We shall work with three sub-families 
$\calF^{[l]} $, $\calF^{[r]} $ and 
$\calF = \calF^{[l]} \cap \calF^{[r]} $ of irreducible paths.
Those
are defined as follows:
\[
 \calF^{[l]}  = \lbr\gamma
 \ {\rm irreducible}:\gamma\subset \sfu_n - \calY \rbr, \  
\calF^{[r]} = \lbr\gamma
\ {\rm irreducible}:\gamma\subset \sfu_0 + \calY \rbr .
\]
Note that any $\gamma = \lb\sfu_0 , \dots , \sfu_n\rb \in \calF$ is 
automatically confined to the 
diamond shape
\be 
\label{eq:L3-diamond}
\gamma\subset D(\sfu_0 , \sfu_n )\df \lb \sfu_0 +\calY \rb\cap \lb \sfu_n -\calY \rb .
\ee
Proposition~\ref{eq:L4-cp-paths} implies that up to corrections of order 
${\rm e}^{-\nu\abs{\sfx}}$
one can restrict attention to paths $\gamma\in\calP_\sfx$ which have the following 
decomposition into irreducible pieces (see Figure~\ref{fig:decomp}):
\be 
\label{eq:L3-irreducible-decomp-d}
\gamma = \gamma^{[l]}\circ \gamma^1\circ \dots\circ \gamma^N\circ\gamma^{[r]} .
\ee
\begin{figure}[h]
\sidecaption
\scalebox{.25}{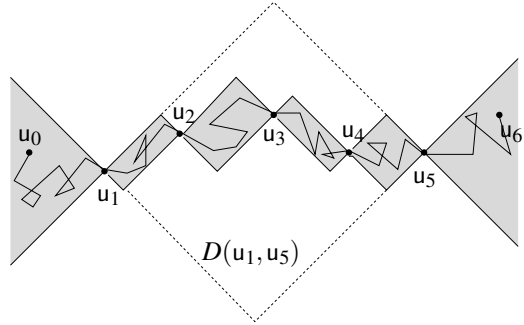}
%
%
\caption{Irreducible decomposition $\gamma =\gamma^{[l]}\circ \gamma^1\circ \dots\circ \gamma^4\circ\gamma^{[r]}$.
The path $\gamma^{[l]}: \sfu_0\mapsto\sfu_1$ belongs to $\calF^{[l]}$, and the path $\gamma^{[r]} :
\sfu_5\mapsto\sfu_6$ belongs to $\calF^{[r]}$. Paths $\gamma_\ell :\sfu_\ell\mapsto\sfu_{\ell +1}$ belong 
to $\calF$. Each irreducible $\gamma_\ell$ stays inside the diamond shape 
$D (\sfu_\ell , \sfu_{\ell +1})$, and 
the concatenation $\gamma_1\circ\dots\circ\gamma_4\subset D (\sfu_1 , \sfu_5 )$.}
\label{fig:decomp}       
\end{figure}

Define
\be 
\label{eq:L3-fweights-d}
\sff^{[l]} (\sfx , n ) = \sumtwo{\sfX (\gamma ) = \sfx , \abs{\gamma }=n}{\gamma\in \calF^{[l]}} 
\Wdn{h ,\lambda } (\gamma )\ {\rm and}\ 
\sff^{[r]} (\sfx , n ) = \sumtwo{\sfX (\gamma ) = \sfx , \abs{\gamma }=n}{\gamma\in \calF^{[r]}} 
\Wdn{h ,\lambda } (\gamma ) .
\ee
\begin{thm}
\label{thm:L3-RepAlphabets} 
The weights $\sff^{[l]}$ and $\sff^{[r]}$ have exponentially decaying tails: 
There exists $\nu >0$ and, for every $\lambda >0$,  $\chi_\lambda >0$  such that
the following mass gap estimate holds uniformly in $\sfx$ and $n$:
\be 
\label{eq:L3-massgap-d}
\sff^{[l]} (\sfx , n ),\sff^{[r]} (\sfx , n ) \leqs {\rm e}^{-\nu \abs{\sfx} 
-\chi_\lambda n  } .
\ee
Furthermore, for each $\lambda = \lambda (h )>0$, 
$\Wdn{h , \lambda}$ is a probability distribution on $\calF$. In particular, 
the family of weights $\lbr \sff (\sfx ,\sfn )\rbr$, 
\be 
\label{eq:AN-f-weights}
 \sff  (\sfx , n ) = \sumtwo{\sfX (\gamma ) = \sfx , \abs{\gamma }=n}{\gamma\in \calF } 
\Wdn{h ,\lambda } (\gamma ) \df \sum_{\gamma\in\calF_{\sfx , n}}
\Wdn{h ,\lambda } (\gamma )
\ee
is a probability distribution on $\bbZ^{d}\times\bbN$ with exponentially decaying tails. 
\end{thm}
\begin{remark}
 \label{rem:NotCritical}
 Note that exponential decay in $n$ is claimed only if $\lambda >0$. $\lambda =0$ 
 corresponds  to the case of critical drifts $h\in\partial \Knot$. For critical
 drifts, 
 the decay in $n$  is  sub-exponential and,  furthermore, the whole coarse-graining 
 (skeleton construction) procedure should be modified. 
 It happens, nevertheless, that the decay in $\sfx$ is still exponential. 
 We do not discuss critical case 
 in these lecture notes, and refer to \cite{IV-Critical}. 
\end{remark}
\begin{proof}
 For fixed $\lambda >0$ 
 bounds $\sff^{[l]} (\sfx , n ),\sff^{[r]} (\sfx , n ) \leqs {\rm e}^{-\nu\abs{\sfx}  }$
directly follow from Proposition~\ref{prop:L4-cp-paths}. 
The case to work out is when $\abs{\sfx}$ is much smaller than $n$, say $\abs{\sfx} <
\frac{\epsilon}{\lambda} n$. But then ${\rm e}^{h\cdot\sfx - \lambda n} \leq 
{\rm e}^{-(\lambda - \epsilon ) n}$. 
Since by 
\eqref{eq:TH-GlHl-bound} the two-point function 
$G_0 (\cdot )$ is bounded, the decay in $n$ indeed 
comes for free as long as $\lambda >0$. 

In order to see that $\lbr \sff  (\sfx , n )\rbr$ is a probability distribution
recall that $\Kl$ was characterized as the closure of the domain of convergence 
of $h \mapsto \sum_{\gamma} {\rm e}^{h\sfx - \lambda n}\Wd (\gamma )$. Thinking 
in terms of 
\eqref{eq:L3-irreducible-decomp-d}, and in view of \eqref{eq:L3-massgap-d} this 
necessarily implies that $\sum \sff (\sfx , n )=1$.   
\end{proof}
{\bf Local geometry of $\partial\Kl$ and analyticity of $\lambda$.} 
Inspecting construction of the cone $\calY$ for 
$h\in\partial\Kl$ we readily infer that the very same $\calY$ would do for all drifts
$g\in \lb h + \bbB_\epsilon^d\rb\cap\partial\Kl$, for some $\epsilon >0$ sufficiently 
small. Similarly, it would do for all $\abs{\mu -\lambda }$ sufficiently small.
We are in the general renewal framework of \eqref{eq:ARen-Key-eq}. The following 
theorem is a consequence of \eqref{eq:AR-lfunc}, Lemma~\ref{lem:AR-Xi}  
 of Subsection~\ref{ssec:renewal} 
and 
\eqref{eq:A-Sigma-map} of the Appendix: 
\begin{thm}
 \label{thm:AN-shape}
 There exists $\epsilon =\epsilon_\lambda >0$, such that the following happens: 
 Let $h$ be a super-critical drift; $\lambda (h ) = \lambda >0$. Construct $\calY$ 
 and, 
 accordingly, $\lbr f (\sfx , n)\rbr$ as in Theorem~\ref{thm:L3-RepAlphabets}. 
 Then for $(g, \mu )\in\bbB_\epsilon^{d+1} (h , \mu )$, 
 \be 
 \label{eq:AH-lambda-analytic}
 \mu = \lambda (g )\ \Leftrightarrow\ \sum_{\gamma\in\calF} \Wdn{g,\mu } \lb \gamma \rb 
 =  \sum_{\sfx , n} {\rm e}^{(g-h)\cdot\sfx + (\lambda -\mu )n }\sff (\sfx ,n ) =1 .
 \ee
 As a result, $\lambda (\cdot )$ is real analytic on $\bbB_\epsilon^d (h )$ 
 and $\Xi (h ) = {\rm Hess} (\lambda)  (h )$ is non-degenerate. 
 
 In particular, 
 \be 
 \label{eq:AN-shape}
 g\in \lb h + \bbB_\epsilon^d\rb\cap\partial\Kl \ \Leftrightarrow \ 
 \sum_{\sfx , n} {\rm e}^{(g-h )\sfx } \sff (\sfx , n ) =1 .
 \ee
 As a result, $\pKl$ is locally analytic and has a uniformly positive 
 Gaussian curvature. 
\end{thm}
\smallskip 

\noindent
{\bf Ornstein-Zernike Theory.} Theorem~\ref{thm:L3-RepAlphabets} paves the way 
for an application of the 
multidimensional renewal theory, as described in Subsection~\ref{ssec:renewal} 
to a study of various limit properties of annealed measures $\Pnf{h}$, 
whenever $h\not\in\Knot$ is a super-critical drift. For the rest of the section 
let us fix such $h$ and $\lambda = \lambda (h ) >0$. By the above, this generates
a cone $\calY$ and a probability distribution $\lbr \sff (\sfx , n ) \rbr$ with 
 exponentially  decaying tails. We declare that it is a probability distribution 
 of a random vector $\sfU = \lb \sfX , \sfT\rb\in \bbZ^d\times \bbN$. In view of
 our assumptions on the underlying random walk, $\sfU$ satisfies the non-degeneracy
 Assumption~\ref{as:AR-distribution}, which means that it is in the framework of
 the theory developed therein. 
 
 Let us construct the array $\lbr \sft (\sfx , n )\rbr$ via the renewal relation 
 \eqref{eq:ARenewal-d}. The number $\sft (\sfx , n )$ has the following meaning:
 Recall our definition $D (\sfx , \sfy ) = \lb \sfx +\calY\rb\cap \lb \sfy -\calY\rb$ 
 of diamond shapes, and 
 define the following three families  of diamond-confined paths: 
 \be 
 \label{eq:AN-Tpaths}
 \calT_\sfx = \lbr \gamma\in \calP_\sfx :\gamma\subset D(0, \sfx )\rbr, \ 
 \calT_n = \lbr \gamma\in \calP_n :\gamma\subset D(0, \gamma_n )\rbr\ 
 {\rm and}\ 
 \calT_{\sfx , n} = \calT_\sfx\cap \calT_n .
 \ee
 As before, $\sft (\sfx ) = \sum_n \sft (\sfx , n )$ and 
 $\sft (n )= \sum_\sfx \sft (\sfx , n )$. Then, 
 \be 
 \label{eq:AN-sft-quantities}
 \sft (\sfx , n ) = \sum_{\gamma\in\calT_{\sfx , n}}\Wdn{h, \lambda } (\gamma ),\ 
 \sft (\sfx ) = \sum_{\gamma\in\calT_{\sfx}}\Wdn{h, \lambda } (\gamma )\ 
 {\rm and}\ 
 \sft (n ) = \sum_{\gamma\in\calT_{  n}}\Wdn{h, \lambda } (\gamma ). 
 \ee
 {\bf Asymptotics of partition functions.} By Theorem~\ref{thm:L3-RepAlphabets}, 
 the partition function $Z_n (h )$ in \eqref{IN-pf} satisfies:
 \be 
 \label{eq:AN-pf-as}
 {\rm e}^{- n \lambda}Z_n (h ) =  \Wdn{h, \lambda}\lb \calP_n \rb = 
 \bigo{{\rm e}^{-\chi_\lambda n}} +\sum_{k+m+j =n}
 \sff^{[l]}(k )\sft (m )\sff^{[r]} (j ) .
 \ee
 Define $\kappa (h ) = \lb \sum_k \sff^{[l]}(k )\rb \lb \sum_j \sff^{[l]}(j )\rb$ and 
 $\mu (h ) = \sum n\sff (n ) = \bbE\sfT $. By Lemma~\ref{lem:AR-lim-exp}
 \be 
 \label{AN-conv-pf}
 \lim_{n\to\infty}{\rm e}^{- n \lambda (h )}Z_n (h ) = \frac{\kappa (h )}{\mu (h )} , 
 \ee
 exponentially fast. 
 \smallskip

 \noindent
 {\bf Limiting spatial extension and other  limit theorems.} 
 Since $\lambda$ is differentiable at any $h\not\in\Knot$, \eqref{eq:L2-ballistic} and LLN 
 \eqref{eq:TH-LLN} follow with $\sfv = \nabla (h )$. However, since for any $h\not\in \Knot$ 
  the probability distribution $\lbr \sff (\sfx , \sfn )\rbr$ in 
  \eqref{eq:AN-f-weights} has exponential tails much sharper local limit results follow 
 along the lines of Subsection~\ref{ssec:renewal}:  
 Let $g\not\in\Knot$ and $\sfu = \nabla\lambda (g )$. 
 Fix $\delta$ sufficiently small and 
 consider $\bbB_\delta^d (\sfu )$. 
 For $\sfw\in \bbB_\delta^d (\sfu )$ define 
 $\sfw_n = \lfloor n\sfw \rfloor/n$ and let 
$g_n$ being defined via $\sfw_n = \nabla\lambda ( g_n )$.
 \begin{thm}
  \label{eq:thm:AN-Local-lim-P}
There exists a positive real analytic function $\psi$ on $\bbB_\delta^d (\sfu )$ such that
\be 
\label{eq:AR-LocLim-P}
\Pnf{h} \lb \lfloor n\sfw \rfloor \rb = 
\frac{\psi (\sfw )}{ \sqrt{(2\pi )^d{\rm det}\, \Xi (g_n )}} {\rm e}^{-n I_h  (\sfw_n )}
\lb 1 +\smo{1} \rb , 
\ee
uniformly in  $\sfw \in \bbB_\delta^d (\sfu )$. 

In particular (considering $\sfu =\sfv$),  under $\Pnf{h}$ the distribution of the 
rescaled end-point $\frac{ \sfX - n\sfv}{\sqrt{n}}$ 
converges to the $d$-dimensional mean-zero normal distribution  with covariance 
matrix $\Xi (h ) = {\rm Hess} (\lambda )(h)$. 
\end{thm}
Let us turn to the (Ornstein-Zernike) 
asymptotics of the two point function $G_\lambda$. Let $\sfx\neq 0$ and 
$h = \nabla\tl (\sfx )$, that is $h\in \partial\Kl$ and $\tl (\sfx) = h\cdot \sfx$. Since, 
as we already know, $\partial \Kl$ is strictly convex, such $h$ is unambiguously defined. 
Under the weights $\Wdn{h, \lambda}$ 
the irreducible decomposition \eqref{eq:L3-irreducible-decomp-d} 
folds in the sense that the $\Wdn{h, \lambda}$-weight of all 
paths $\gamma\in\calP_\sfx$ which do not comply with it, is 
exponentially negligible as compared to $G_\lambda (\sfx )$. Hence:
 \be
  \label{eq:AH-OZ}
  {\rm e}^{\tl (\sfx )} G_\lambda (\sfx ) \lb 1 +\smo{1} \rb = \sum_n \sft (\sfx , n )
  = \frac{c (h )}{\sqrt{\abs{\sfx}^{d-1}}} \lb 1 +\smo{1} \rb, 
 \ee
asymptotically in $\sfx$ large. This follows from \eqref{eq:AR-LocLim-P} and Gaussian summation 
formula. 
\smallskip 

\noindent
{\bf Invariance principles.}
There are two possible setups for formulating invariance principles for annealed 
polymers. The first is when we consider $\Pnf{h}$, and accordingly polymers $\gamma$ with 
fixed number $n$ of steps. In this case one defines
\[
 x_n (t ) = \frac{1}{\sqrt {n }}\lb \gamma_{\lfloor nt\rfloor} - nt\sfv\rb, 
\]
and concludes from Theorem~\ref{eq:thm:AN-Local-lim-P} that $x_n (\cdot )$ converges to 
a $d$-dimensional Brownian motion with covariance matrix $\Xi (h )$. 

A somewhat different $(d-1)$-dimensional invariance principle holds in the conjugate ensemble
of {crossing} polymers. To define the latter fix $\sfx \neq 0$, $\lambda >0$ and 
consider the following probability distribution $\Plf{\sfx}$ on the family  $\calP_\sfx$ of 
all polymers $\gamma$ which have displacement $\sfX (\gamma )= \sfx$:
\be 
\label{eq:AN-Plf-x}
\Plf{\sfx} (\gamma )= \frac{1}{G_\lambda (\sfx )}{\rm e}^{-\lambda\abs{\gamma }} \Wd (\gamma ) .
\ee
Consider again 
the irreducible decomposition \eqref{eq:L3-irreducible-decomp-d} of paths $\gamma\in\calP_\sfx$. 
Let $0, \sfu_1, \dots, \sfu_{N+1}, \sfx$ be the end-points of the corresponding 
irreducible paths. We can approximate $\gamma $ by a linear interpolation through these vertices. 
We employ the language of Subsection~\ref{ssec:A-Surfaces} of the Appendix. Let 
$\frn (h) = \frac{\sfx}{\abs{\sfx }}$ and $\frv_1 ,\dots ,\frv_{d-1}$ are unit vectors in the 
direction of principal curvatures of $\pKl$ at $h$.  
In the orthogonal frame 
$\lb \frv_1 ,\dots ,\frv_{d-1} , \frn (h)\rb$, the linear interpolation through the vertices 
of the irreducible decomposition of $\gamma$ can be represented as a function $Y :[0, \abs{\sfx}]\mapsto 
T_h \pKl$. Consider the rescaling 
\[
 y_\sfx (t ) = \frac{1}{\sqrt{\sfx}} Y (t\abs{\sfx} ). 
\]
Then, \eqref{eq:AH-OZ} and the quadratic expansion formula \eqref{eq:QExpansion-tau} of the Appendix leads to 
the following conclusion: Let $\sfx_m $ be a sequence of points with $\abs{\sfx_m}\to\infty$ and 
$\lim_{m\to\infty}\frac{\sfx_m}{\abs{\sfx_m }} = \frn (h )$. Then the distribution of $y_{\sfx_m }(\cdot )$
under $\Plf{\sfx_m }$ converges to the distribution of the $(d-1)$-dimensional Brownian bridge with
the diagonal covariance matrix ${\rm diag}\lb \chi_1 (h ), \dots , \chi_{d-1} (h )\rb$, where 
$\chi_i (h )$ are the principal curvatures of $\pKl$ at $h$. 

\section{Very weak disorder in $d\geq 4$.}
\label{sec:Weak} 
The notion of very weak disorder depends on the dimension $d\geq 4$ and 
on the pulling force $h\neq 0$. It is quantified in terms of  
continuous positive non-decreasing functions 
$\zeta_d$ on $(0,\infty)$; $\lim_{h\downarrow 0}\zeta_d (h )= 0$.  Function 
$\zeta_d$ does not have an 
independent physical meaning: 
It is  needed to ensure a certain percolation property \eqref{eq:WD-perc-cone}, and 
 to ensure validity of a certain $\bbL_2$-type estimate
formulated in Lemma~\ref{lem:WDBound} below.
\begin{definition}
 \label{def:VW-disorder}
 Let us say that the polymer model \eqref{IN-pd} is in the regime of very weak disorder
 if $h\neq 0$ and 
 \be
 \label{eq:WD-VW}
 \phi_\beta (1) \leq \zeta_d \lb \abs{h}\rb .
 \ee
\end{definition}
\begin{remark}
\label{rem:WD-condition}
\eqref{eq:WD-VW} 
 is a technical condition, and it has three main implications as 
 it is explained below after formulation of Theorem~\ref{thm:B}. 
 \end{remark}
\smallskip 

\noindent
{\bf Quenched Polymers at very weak disorder.} 
In the regime of very weak disorder quenched polymers behave like their annealed
counter-parts. Precisely: For $h\neq 0$, continue to use $\sfv_h = \nabla \lambda ( h)$
and $\Xi_h = {\rm Hess}(\lambda )(h )$ for the limiting spatial extension and the 
diffusivity of the annealed model. Let ${\rm Cl}_\infty$ be the unique infinite 
connected cluster of $\lbr \sfx~:~\sfVo_\sfx <\infty\rbr$. 
\begin{thm}
 \label{thm:A}
Fix $h\neq 0$. Then, in the regime of very weak disorder,
the following holds $\calQ$-a.s.\ on the event $\lbr 0\in {\rm Cl}_\infty\rbr$:
\begin{itemize}
\item  The limit
\be
\label{eq:Aclaima}
 \lim_{n\to\infty} \frac{\Zno  (h)}{\Zn (h)}
\ee
exists and is a strictly positive, square-integrable random variable.
\item 
For every $\epsilon >0$, 
\be
\label{eq:Aclaimb}
\sum_n \Pnfo{h}\lb  \left|\frac{ \sfX(\gamma )}{n} - \sfv_h \right| >\epsilon  \rb < \infty .
\ee
\item 
For every $\alpha\in \bbR^{d}$,
\be
\label{eq:Aclaimc}
 \lim_{n\to\infty} \Pnfo{h} \lb  {\rm exp} 
\lbr \frac{i\alpha}{\sqrt{n}}\lb X (\gamma ) - n\sfv_h\rb\rbr\rb =
\exp\lbr -\frac12 \Xi_h \alpha\cdot\alpha \rbr .
\ee
\end{itemize}
\end{thm}
\begin{remark}
 \label{rem:W-23cond}
Since in the regime of very weak disorder $h\not\in \Knot$, the series in \eqref{eq:TH-LLN} converge 
exponentially fast. Using $\Zno (A |h )$ for the restriction of $\Zno$ to paths from $A$ we conclude
 (from exponential Markov inequality) 
that there exists $c= c (\epsilon ) >0$ such that
\[
 \sum_n\calQ\lb \Zno\lb \left|\frac{ \sfX(\gamma )}{n} - \sfv_h \right| >\epsilon \big| h  \rb > 
 {\rm e}^{-c n} \Zn (h ) \rb
 <\infty .
\]
In other words, \eqref{eq:Aclaimb} routinely follows from \eqref{eq:Aclaima} and exponential 
bounds on annealed polymers. 
\end{remark}

\noindent
{\bf Reformulation in terms of basic partition functions.} 
For each $h\neq 0$ and $\beta >0$ 
basic annealed partition functions were defined in \eqref{eq:AN-sft-quantities}. 
Here is the corresponding definition of basic quenched partition functions:
\be 
 \label{eq:Q-sft-quantities}
 \sfto (\sfx , n ) = \sum_{\gamma\in\calT_{\sfx , n}}{\rm e}^{h\cdot 
 \sfX (\gamma )-\lambda\abs{\gamma}}
 \Wdo (\gamma ), 
 \ee
 and, accordingly $\sfto (\sfx )= \sum_n\sfto (\sfx , n )$ and $\sfto (n )= 
 \sum_\sfx\sfto (\sfx , n )$. 
 
Let $\calY = \calY_h$ be the cone used to define irreducible paths. Then  
${\rm Cl}_\infty^{h}$ is the infinite connected component (unique if exists )
of $\lbr \sfx\in \calY_h~:~ \sfVo_\sfx <\infty\rbr$. 
\begin{thm}
\label{thm:B}
Fix $h\neq 0$. Then, in the regime of very weak disorder, 
infinite connected cluster ${\rm Cl}_\infty^{h}$ exists $\calQ$-a.s.. 
Furthermore, 
the following holds $\calQ$-a.s.\ on the event $\lbr 0\in {\rm Cl}_\infty^{h}\rbr$:
\begin{itemize} 
\item The limit
\be
\label{eq:claima}
\sfs^\omega \df \lim_{n\to\infty} \frac{\tq{n}}{\ta{n}}
\ee
exists and is a strictly positive, square-integrable random variable.
\item 
For every
$\alpha\in\bbR^{d+1}$,
\be
\label{eq:claimc}
\lim_{n\to\infty}
\frac1{\tq{n}  }\sum_x\exp\lbr \frac{i\alpha}{\sqrt{n}}\cdot \lb x - nv\rb\rbr
\tq{x ,n}
=
\exp\lbr -\frac12 \Sigma\alpha\cdot\alpha \rbr .
\ee
\end{itemize}
\end{thm}
Below we shall explain the proof of \eqref{eq:claima}. The $\calQ$-a.s. CLT
follows in a rather similar fashion, albeit with some additional technicalities, 
and we refer to \cite{IV-CLT} for the complete proof. 
\smallskip 

\noindent
{\bf Three properties of very weak disorder.} 
The role of (technical) condition $\phi_\beta (1)\leq \zeta_d  (\abs{h} )$ is threefold: 

First of all,
setting $p_d = \calQ\lb \sfVo =\infty\rb$ and noting that
\be 
\label{eq:WD-perc}
 \zeta_d \lb \abs{h}\rb \geq 
 \phi_\beta (1) = -\log\calE\lb {\rm e}^{-\beta\sfVo }\rb \geq -\log (1-p_d ) \geq p_d, 
\ee
we conclude that \eqref{eq:WD-VW} implies that: 
\be 
\label{eq:WD-pd-small}
\calQ\lb \sfVo =\infty\rb\leq \zeta_d \lb \abs{h}\rb .
\ee
Thus, 
in view of \eqref{eq:WD-perc}, condition 
\eqref{eq:WD-VW} 
implies that 
the infinite 
cluster ${\rm Cl}_\infty^{h}$ exists: Namely if 
we choose $\zeta_d$ such that 
$\sup_h \zeta_d (\abs{h} )$ is sufficiently small, 
then 
\be 
\label{eq:WD-perc-cone}
\calQ\lb \text{there is an infinite cluster ${\rm Cl}_\infty^{h}$}\rb =1, 
\ee
for all situations in question. 

The second implication is that for any $h\neq 0$ fixed, $h\not\in\Knot (\beta )$ for all 
$\beta$ sufficiently small. In other words, in the regime of  very weak disorder 
the annealed model 
is always in the ballistic phase. Indeed, since $\phi_\beta (\ell )\leq 
\ell\phi_\beta (1 )$, 
\[
 Z_n (h ) \geq {\rm e}^{-n \phi_\beta (1)}\lb \sfE_d {\rm e}^{h\cdot\sfX}\rb^n .
\]
Consequently $\lambda (h ) >0$ whenever  $\log \lb \sfE_d {\rm e}^{h\cdot\sfX}\rb
 > \phi_\beta (1)$. 

The third  implication is an $\bbL_2$-estimate \eqref{eq:WDBound} below. Recall that 
for $h$ fixed, annealed measures $\Pnf{h}$ have limiting spatial extension 
$\sfv_h = \nabla\lambda (h )$, and they satisfy sharp classical local limit 
asymptotics around this value. 

For a subset $A\subseteq\bbZ^{d}$, let $\calA$ be the $\sigma$-algebra
generated by
$\lbr \sfVo_\sfx \rbr_{\sfx\in A}$. We shall call such $\sigma$-algebras cylindrical.
\begin{lem}
\label{lem:WDBound}
{For any dimension $d\geq 4$ there exists a positive non-decreasing function $\zeta_d$ on 
$(0,\infty )$ and a number $\rho < 1/12$ such that the following holds: If $\phi_\beta  (1) 
<\zeta_d (\abs{h} )$, 
then  }
there exist  constants $c_1 ,c_2 <\infty$ such that 
\be
\label{eq:WDBound}
\begin{split}
&\abs{
\calE 
 \left[ \tq{\sfx , \ell}\tq{\sfy ,\ell } \calE\lb \sffof{\sfx }(\sfz, m) -\sff (\sfz, m )\big|
\calA \rb \calE \lb \sffof{\sfy} (\sfw, k )  -\sff (\sfw, k )\big|\calA\rb
 \right]
} \\
&\quad
\leq \frac{c_1{\rm e}^{- c_2 (m+ k )}}{ \ell^{d -\rho}} \exp\lbr
-c_2\lb \abs{\sfx -\sfy}
+ \frac{\abs{\sfx -\ell \sfv_h}^2}{\ell} +\frac{\abs{\sfy -\ell \sfv_h}^2}{\ell}\rb\rbr , 
\end{split}
\ee
{for all} $\sfx , \sfy, \sfz, \sfw, m, k$ 
 and  all cylindrical  $\sigma$-algebras $\calA$ such that both
$\tq{\sfx,\ell}$ and $\tq{\sfy,\ell}$ are
$\calA$-measurable.
\end{lem}
\begin{remark}
Since the underlying random walk is of bounded range, 
\[
\#\lbr \sfz : \sffo (\sfz , m )
\neq 0\rbr \leqs m^{d} .
\]
Hence \eqref{eq:WDBound} also holds with 
$\sffof{\sfx }(m)$ and  $\sffof{\sfy} (k )$ instead of 
$\sffof{\sfx }(\sfz, m)$ and  $\sffof{\sfy} (\sfw, k )$. \newline 
There is nothing sacred about the condition $\rho <1/12$. We just need $\rho$ to
be sufficiently small. In fact, \eqref{eq:WDBound} holds with  $\rho=0$,
although a proof of such statement would be a bit  more involved. 
\end{remark}
We refer to \cite{IV-CLT} (Lemma~2.1 there) for a proof of Lemma~\ref{lem:WDBound}. 
The claim \eqref{eq:WDBound} has a transparent meaning: For $\rho=0$, the expression
\[
\frac{c_1}{\ell^d} 
{\rm exp}\lbr -c_2\lb 
\frac{\abs{\sfx -\ell \sfv_h}^2}{\ell} +\frac{\abs{\sfy -\ell \sfv_h}^2}{\ell}\rb
\rbr
\]
is just the local limit bound on the annealed quantity $\ta{\sfx , \ell}\ta{\sfy , \ell}$. 
 The term ${\rm e}^{-c_2 ( k+m )}$ reflects  exponential
decay of irreducible terms  $\lb \sffof{\sfx} (\sfz, m )  -\sff (\sfz , m )\rb$ and, 
accordingly, 
$\lb \sffof{\sfy} (\sfw, k )  -\sff (\sfw, k )\rb$. The term 
${\rm e}^{-c_2\abs{\sfx - \sfy }}$ appears for the following reason (see Figure~\ref{fig:kites}):  
By irreducibly, 
$\sffof{\sfx} (\sfz ,m )  -\sff (\sfz ,m )$ depends only on $\sfVo_\sfu$ with 
$\sfu$ belonging to the diamond shape $D(\sfx , \sfx+\sfz )$. Similarly, 
$\sffof{\sfy} (\sfw ,m )  -\sff (\sfw ,m )$ depends only on variables 
inside $D(\sfy , \sfy +\sfw )$. All these terms have zero mean. Consequently, 
\[
\calE 
 \left[ \tq{\sfx , \ell}\tq{\sfy ,\ell } \calE\lb \sffof{\sfx }(\sfz, m) -\sff (\sfz ,  m )\big|
\calA \rb \calE \lb \sffof{\sfy} (\sfw , k )  -\sff (\sfw, k )\big|\calA\rb \right] = 0, 
\]
whenever $D(\sfx , \sfx+\sfz )\cap D (\sfy , \sfy+\sfw )= \emptyset$. The remaining terms
 satisfy $\max\lbr\sfz, \sfw\rbr\geqs \abs{\sfx - \sfy }$. In other words, 
 the term ${\rm e}^{-c_2\abs{\sfx - \sfy }}$ also reflects exponential decay of irreducible
 connections. 
 \begin{figure}[h]
\scalebox{.25}{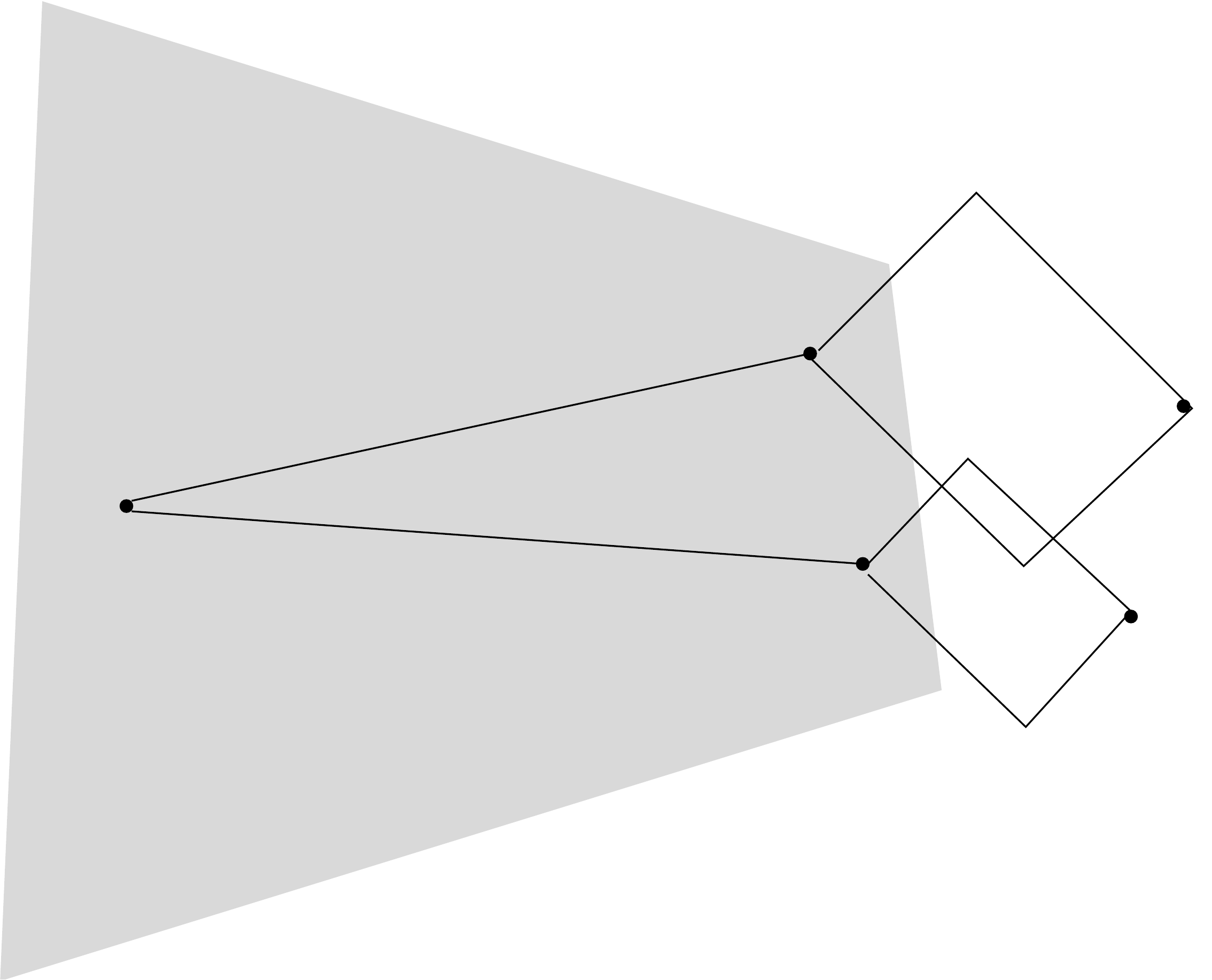}
%
%
\caption{Kites $\tq{\sfx , \ell} \calE\lb \sffof{\sfx }(\sfz, m) -\sff (\sfz ,  m )\big| \calA \rb$ 
and $\tq{\sfy , \ell} \calE\lb \sffof{\sfy }(\sfw, k) -\sff (\sfz ,  k )\big| \calA \rb$ may 
have a non-zero 
covariance only if their diamond shapes 
$D (\sfx , \sfx +\sfz )$ and  $D (\sfy , \sfy +\sfw )  
$ intersect. 
}
\label{fig:kites}       
\end{figure}

The disorder imposes an
attractive interaction between two replicas. The impact of 
 small $\xi_d (\abs{h})$ condition in the very weak disorder regime,
 as formulated in \eqref{eq:WDBound}, 
 is that this interaction
is not strong enough to  destroy individual annealed  asymptotics. 
\smallskip 

\noindent
{\bf Sinai's decomposition of $\tq{\sfx, n }$.} 
We rely on an expansion  similar
to the one employed by Sinai \cite{Sinai} in the context of directed polymers. 
By construction quantities
$\tq{\sfx, n }$ satisfy the following (random) renewal relation:
\be 
\label{eq:QRenewal-d}
\tq{\sfx , 0 } = \1_{\lbr \sfx =0\rbr}\quad{\rm and}\quad 
\tq{\sfx , n} = \sum_{m=1}^{n}\sum_{\sfy} \fq{\sfy , m }\sftof{\sfy} (\sfx -\sfy , n-m ) .
\ee
Iterating in \eqref{eq:QRenewal-d} we obtain (for $n>0$): 
\[
 \tq{\sfx , n} = \fq{\sfx , n} + 
 \sum_{r=1}^\infty\sum_{\sfx_1,\dots \sfx_r}\sum_{n_1+\dots +n_r =n}
 \prod\sffof{\sfx_{i-1}}(\sfx_i - \sfx_{i-1} , n_i) , 
\]
where $\sfx_0\df 0$. Writing, 
\[
\sffo (\sfy  , m) = 
\fa{\sfy, m} + \lb \sffo (\sfy , m) - 
\fa{\sfy , m}\rb,
\]
we, after expansion and re-summation,  arrive 
to the following decomposition: 
\be
\label{eq:Sinai0}
\tq{\sfx ,n} = \ta{\sfx, n} +
\sum_{\ell =0}^{n-1}\sum_{m=1}^{n-\ell}\sum_{r=0}^{n-\ell-m}
\sum_{\sfy, \sfz}
\tq{\sfy ,\ell}\lb  \sffof{\sfy}( \sfz -\sfy , m) - \fa{ \sfz -\sfy,m}\rb \ta{\sfx - \sfz ,r} .
\ee
In particular, the decomposition of $\tq{n}$ is given by
\be
\label{eq:Sinai-tn}
\tq{n} = \ta{n} +
\sum_{\ell =0}^{n-1}\sum_{m=1}^{n-\ell}\sum_{r=0}^{n-\ell-m}
\sum_{\sfy}
\tq{\sfy ,\ell}\lb  \sffof{\sfy}(m) - \fa{m}\rb \ta{r} .
\ee
In order to prove \eqref{eq:claimc} one needs to consider the full decomposition 
\eqref{eq:Sinai0}. As it was already mentioned, we shall not do it here and, instead, refer 
to \cite{IV-CLT}. From now on, we shall concentrate on proving \eqref{eq:claima} and, accordingly
shall consider the reduced decomposition \eqref{eq:Sinai-tn}. Nevertheless, modulo additional 
technicalities, the proof of 
\eqref{eq:claima} captures all essential features of the argument. 

Recall that for the annealed quantities,  $\lim_{n\to\infty}\ta{n} 
= \frac{1}{\mu (h )} =\frac{1}{\mu}$ 
exponentially fast. Writing $\ta{m} = \frac{1}{\mu} + ( \ta{m}- \frac{1}{\mu} )$
 in all the corresponding terms in \eqref{eq:Sinai-tn}, we infer that 
 $\tq{n}$ 
can be represented
as
\be
\label{eq:Sinai}
\tq{n} = \frac{1}{\mu } \sfs^\omega (n)
+  \epsilon_n^\omega + \lb \ta{n} - \frac{1}{\mu}\rb
\\
\ee
where
\be
\label{eq:snterm}
\sfs^\omega (n ) = 1 + \sum_{\ell \leq n} \sum_\sfx
\tq{\sfx ,\ell}
 \lb\sffof{\sfx}  -1\rb,
\ee
and
the correction term $\epsilon_n^\omega$ 
 is given by
\be
\label{eq:EpsTerm}
\begin{split}
\epsilon_n^\omega = \epsilon_{n ,1}^\omega -  \epsilon_{n ,2}^\omega &\df
\sum_{\ell +m +r =n} \sum_{\sfx }
\tq{\sfx ,\ell}\lb \sffof{\sfx} (m ) - \fa{m}\rb 
\lb \ta{r} -\frac{1}{\mu}\rb\\ 
&+
- \frac{1}{\mu} \sumtwo{\ell\leq n}{m >  n -\ell}\sum_\sfx
\tq{x ,\ell}\lb \sffof{\sfx} (m ) - \fa{m}\rb .
\end{split}
\ee
The term $\ta{n} - \frac{1}{\mu}$ is negligible. Our target claim \eqref{eq:claima} is 
a direct consequence of the following proposition:
\begin{prop}
\label{prop:WD-limits}
In the very weak disorder regime the following happens $\calP$-a.s.:
\be
\label{eq:EpsConv}
\lim_{n\to\infty} \sfs^\omega (n ) = \sfs^\omega = 
 1 + \sum_{\sfx}\tq{\sfx}\lb \sffof{\sfx} - 1\rb  
\ \ {\rm and}\ \
\sum_n \calE ( \epsilon^ \omega_n )^2 <\infty .
\ee
Furthermore, $\sfso >0$ on the set $\lbr 0\in {\rm Cl}_\infty^{h}\rbr$. 
\end{prop}
\begin{remark}
 \label{rem:WD-sero-s}
 Note that the formula for $\sfso$ is compatible with the common sense 
 if the random walk is trapped (case $\calQ (\sfVo = \infty )>0$.)
 Indeed, in such situation  $\lim_{n\to\infty}\tq{n}$
 should be clearly zero. On the other hand, if the random walk is trapped, 
 then the sum $1+\sum_{\sfx}\tq{\sfx}\lb \sffof{\sfx} - 1\rb$ contains 
 only finitely many 
 non-zero terms. 
 Using $\tq{0}=1$, let us rewrite it as 
 \[
  1+ \sum_{\sfx}\tq{\sfx} \sffof{\sfx} - \sum_{\sfx\neq 0}\tq{\sfx} - 1 = 
  \sum_{\sfx}\tq{\sfx} \sffof{\sfx} - \sum_{\sfx\neq 0}\tq{\sfx} .
 \]
 However, for $\sfx\neq 0$, 
 \[
  \tq{\sfx} = \sum_{\sfy}\tq{\sfy}\sffof{\sfy} (\sfx - \sfy )\ \Rightarrow\ 
  \sum_{\sfx\neq 0}\tq{\sfx} = \sum_{\sfy}\tq{\sfy} \sffof{\sfy } .
 \]
\end{remark}
{\bf Mixingale form of $\sfso (n )$ and $\epsilon^ \omega_n$.} 
Let us rewrite $\sfso (n )$ as
\be
\label{eq:Xl}
 \sfso (n ) = 1 + \sum_{\ell\leq n} \sfY_\ell\quad {\rm where}\quad 
 \sfY_\ell = 
 \sum_\sfx \tq{\sfx ,\ell}\lb\sffof{\sfx }  -1\rb.
\ee
The variables $\sfY_\ell$ are mean zero, and it is easy to deduce from 
the basic $\bbL_2$-estimate \eqref{eq:WDBound} that in the regime of weak disorder, 
$\sum\calE\lb \sfY_\ell^2 \rb <\infty$. Should $\lbr \sfY_\ell\rbr$ be a martingale
difference sequence (as in the case of directed polymers), we would be done. However, 
since, in principle, same vertices $\sfx$ may appear in different $\sfY_\ell$-s, there 
seems to be no natural martingale structure at our disposal. Instead one should make
a proper use of mixing properties of $\lbr \sfY_\ell\rbr$. Hence the name mixingale, 
which was introduced in \cite{McLeish}. In order to prove convergence 
of $\sfso (n )$ we shall rely on the mixingale approach developed in \cite{McLeish}.

Turning to the correction terms in \eqref{eq:EpsTerm}, note that both $\epsilon^\omega_{n, 1}$
and $\epsilon^\omega_{n, 2}$ could be written in the form 
\be 
\label{eq:WD-eps-form}
\sum_{\ell\leq n} \sum_{\sfx }\tq{\sfx ,\ell} \sum_m a^{(n)} (\ell , m )
\lb\sffof{\sfx } (m ) - \fa{m}\rb \df \sum_{\ell\leq n} \sfZ^{(n)}_\ell , 
\ee
where 
\be 
\label{eq:WD-coed-an}
a^{(n)} (\ell , m ) = \lb \ta{n - \ell - m}-\frac{1}{\mu}\rb\1_{\ell +m \leq n}\ \text{and}\ 
a^{(n)} (\ell , m ) = -\1_{\ell +m >n}
\ee
respectively in the cases of $\epsilon^\omega_{n, 1}$
and $\epsilon^\omega_{n, 2}$. Again, $\lbr \sfZ^{(n)}_\ell\rbr$ is not a martingale difference
sequence, and we shall rely on the mixingale approach of \cite{McLeish} for deducing their 
second convergence statement in  \eqref{eq:EpsConv}. 

Below we shall formulate a particular case of the maximal inequality for 
mixingales \cite{McLeish}. To keep relation with quenched polymers, and specifically with 
\eqref{eq:Xl} and  \eqref{eq:WD-eps-form},  in mind, let us introduce the following 
filtration 
$\lbr\calA_m\rbr$. Recall that the end-point of the $m$-step annealed polymer stays
close to $m\sfv_h = m\nabla\lambda (h )$. 
Define
half-spaces $\cHm{m}$ and the corresponding $\sigma$-algebras $\calA_m$ as
\be
\label{eq:sigma-algebras}
\cHm{m} = \setof{\sfx\in\bbZ^d }{\sfx\cdot\sfv_h \leq m |\sfv_h |^2}\quad\text{and}\quad
\calA_m =\sigma\setof{\sfVo_\sfx }{x\in\cHm{m}} .
\ee
\smallskip 

\noindent
{\bf Mixingale maximal inequality and convergence theorem of McLeish.}  
Let $\sfY_1 , \sfY_2, \ldots $ be a sequence of zero-mean, square-integrable random
variables. Let also
$\lbr\calA_k\rbr_{-\infty}^{\infty}$ be a filtration of $\sigma$-algebras.
Suppose that there exist $\epsilon >0$ and
numbers $d_1, d_2 , \ldots $ in such a way that
\be
\label{eq:McLeish}
\calE\lb \bbE\lb \sfY_\ell ~\big| \calA_{\ell -k}\rb^2\rb\leq
\frac{d_\ell^2}{(1+k)^{1+\epsilon}}
\quad\text{and}\quad
\calE\lb \sfY_\ell - \bbE\lb \sfY_\ell ~\big| \calA_{\ell + k}\rb\rb^2 \leq
\frac{d_\ell^2}{(1+k)^{1+\epsilon}}
\ee
for all $\ell = 1,2, \ldots $ and $k\geq 0$. Then~\cite{McLeish} there exists $K
= K (\epsilon) < \infty$ such that, for
all $n_1 \leq  n_2$,
\be
\label{eq:Maximal}
\calE\lbr \max_{n_1\leq r\leq n_2}\lb \sum_{n_1}^r \sfY_\ell\rb^2 
\rbr \leq K \sum_{n_1}^{n_2}
d_\ell^2 .
\ee
In particular, if $\sum_\ell d_\ell^2 <\infty$, then $\sum_\ell \sfY_\ell$
converges $\calQ$-a.s.\ and in $\bbL_2$.
\smallskip 

\noindent 
{\bf Convergence of 
 $\sfso (n )$. } Consider decomposition 
 \eqref{eq:Xl}. Clearly, 
 \be 
 \label{eq:W-condE1}
 \calE \lb \sfY_\ell ~\big| \calA_{\ell -k}\rb = \sum_{\sfx \in \cHm{\ell - k}}
 \tq{\sfx ,\ell}\lb\sffof{\sfx }  -1\rb
 \ee
 Applying \eqref{eq:WDBound} we conclude that for any $\sfx, \sfy\in \cHm{\ell - k}$, 
 \be
\label{eq:WDBound-s}
\begin{split}
&\abs{
\calE
 \left[ \tq{\sfx , \ell}\tq{\sfy ,\ell } \calE\lb \sffof{\sfx } - 1\big|
\calA_{\ell -k} \rb \calE\lb \sffof{\sfy}   -1\big|\calA_{\ell - k}\rb
 \right]
} \\
&\quad
\leq \frac{c_3}{ \ell^{d -\rho}} \exp\lbr
-c_2\lb \abs{\sfx -\sfy}
+ \frac{\abs{\sfx -\ell \sfv_h}^2}{\ell} +\frac{\abs{\sfy -\ell \sfv_h}^2}{\ell}\rb\rbr , 
\end{split}
\ee
Consequently, summing up with respect to $\sfx, \sfy\in \cHm{\ell - k}$ we infer 
that for any $\epsilon \geq 2\rho$:
\be 
\label{eq:WD-leqsbound1}
 \calE\lb \calE \lb \sfY_\ell\Big| \calA_{\ell-k}\rb^2\rb \leq 
 \frac{c_5{\rm e}^{-c_4\frac{k^2}{\ell}}}{\ell^{{d}/{2} -\rho}} \leq
 \frac{c_6}{\ell^{(d-1)/{2}- \epsilon } (1+ k)^{1+\epsilon}} \df 
 \frac{d_{\ell , -}^2}{(1+ k)^{1+\epsilon}}. 
\ee
On the last step we have relied on a trivial asymptotic inequality
\[
 \frac{{\rm e}^{-c_4\frac{k^2}{\ell}}}{\ell^{(1+\epsilon )/2}}
 \leqs \frac{1}{(1+k )^{1+\epsilon}}. 
\]
Note that if $(d-1)/{2}- \epsilon >1$, which is compatible with $d\geq 4$ and 
$\rho <1/12$, then $\sum d_{\ell , -}^2 <\infty$. 

Turning to the second condition in \eqref{eq:McLeish} 
note first of all 
that $\calE\lb \tq{\sfx ,\ell}\lb\sffof{\sfx }  -1 \rb\big|\calA_{\ell +k} \rb =0$ 
 whenever $\sfx\in\cHp{\ell + k}$, and 
 \[
 \calE\lb \tq{\sfx ,\ell}\lb\sffof{\sfx } (\sfy  )  -\sff (\sfy ) 
 \rb\big|\calA_{\ell +k} \rb = \tq{\sfx ,\ell}\lb\sffof{\sfx } (\sfy  )  -
 \sff (\sfy  )\rb
 \]
  whenever $\sfx +\sfy \in \cHm{\ell + k}$. Therefore,  
 \be 
 \label{eq:WD-twoterms-plus}
 \begin{split}
  \sfY_\ell - \calE\lb \sfY_\ell\Big|\calA_{\ell +k} \rb 
  &= \sum_{\sfx\in \cHp{\ell + k}} 
  \tq{\sfx ,\ell}\lb\sffof{\sfx }  -1 \rb \\
  &+ 
  \sumtwo{\sfx \in\cHm{\ell + k}}{ \sfy \in \cHp{\ell + k}} 
  \tq{\sfx ,\ell}\lb \sffof{\sfx } (\sfy -\sfx ) -\calE \lb 
  \sffof{\sfx } (\sfy -\sfx )\big| \calA_{\ell +k} \rb\rb .
  \end{split}
  \ee
The first term in \eqref{eq:WD-twoterms-plus} has exactly the same structure as
\eqref{eq:W-condE1}. The second term in \eqref{eq:WD-twoterms-plus} happens to be
even more localized (see discussion of (2.14) in \cite{IV-CLT}). The conclusion is:
\be 
\label{eq:WD-leqsbound2}
\calE\lb \sfY_\ell - \calE\lb \sfY_\ell\Big|\calA_{\ell +k} \rb^2\rb 
\leq 
 \frac{d_{\ell , +}^2}{(1+ k)^{1+\epsilon}} , 
\ee
where $d_{\ell , +}^2 \leqs \ell^{- (d-1 )/2 +\epsilon}$. 

Set $d_\ell^2 = \max\lbr d_{\ell , -}^2 ,d_{\ell , +}^2\rbr \leqs
 \ell^{- (d-1 )/2 +\epsilon}$, we, in view of the feasible choice 
 $(d-1 )/2 -\epsilon >1$, conclude from \eqref{eq:Maximal} that $\sfso (n )$
 is indeed a $\calQ$-a.s. converging sequence. 
 \smallskip 
 
 \noindent
 {\bf Correction terms.} Treatment of correction terms in their mixingale
 representation \eqref{eq:WD-eps-form} follows a similar pattern. We refer to 
 Section~2.2 in \cite{IV-CLT} for the proof of the second claim in 
 \eqref{eq:EpsConv}. 
 \smallskip 
 
 \noindent
 {\bf Positivity of $\sfso$.} As we have already checked the sum $\sfso = 1 +\sum_\sfx 
 \tq{\sfx }\lb\sffof{\sfx }  -1\rb$ converges $\calQ$-a.s. and in $\bbL_2$. In 
 particular, $\calE (\sfso ) = 1$. We claim that $\sfso >0$, $\calQ$-a.s. on the event
  $\lbr 0\in {\rm Cl}_\infty^{h}\rbr$. In order to prove  this it would be enough 
  to check that 
  \be 
  \label{eq:WD-cone-x}
  \calQ\lb \exists\sfx \in \calY : \sfsof{\sfx} >0\rb =1 .
  \ee
  Let us sketch the argument: 
  If $ \sfsof{\sfx} >0$, then $\sfx\in {\rm Cl}_\infty^{h}$. But 
  there is exactly  one infinite cluster in $\calY$. Hence $0$ is connected to
  $\sfx$ by a finite path $\gamma \subset {\rm Cl}_\infty^{h}\subset \calY$. 
  Now, by assumption on $\sfx$, $\lim_{n\to\infty} \sftof{\sfx }(n ) = 
  \lim_{n\to\infty} \sum_\sfz \sftof{\sfx}(\sfz , n ) >0$. By comparison with 
  annealed quantities (large deviations, for instance \eqref{eq:AR-LocLim})
  we, at least for large $n$,  may ignore terms 
  $\sftof{\sfx }(\sfz , n )$ with $\gamma\not\subset (\sfx +\sfz) - \calY$. 
  Which means that 
  $\liminf_{n\to\infty}\sfto (n) \geq \Wdon{h , \lambda}(\gamma )\frac{\sfsof{\sfx}}{\mu}$, 
  where 
  $\Wdon{h , \lambda}(\gamma ) = {\rm e}^{h\cdot\sfX (\gamma ) - \lambda\abs{\gamma}} 
  \Wdo (\gamma ) >0$. 
  
  It remains to check \eqref{eq:WD-cone-x}. Consider sets 
  \[
   B_n = \partial\cHp{n}\cap\calY , \quad \abs{B_n}\eqvs n^{d-1} .
 \]
We refer to the last Subsection of \cite{IV-Crossing} for the proof of the following 
statement: 
\be 
\label{eq:WD-intersection}
\abs{\Cov (\sfsof{\sfx } , \sfsof{\sfy })}\leqs \frac{1}{\abs{\sfy - \sfx}^{d/2-1}} , 
\ee
uniformly in $n$ and in $\sfx, \sfy \in B_n$. 
The quantity $\frac{1}{\abs{\sfy - \sfx}^{d/2 -1}}$  
in \eqref{eq:WD-intersection} 
represents an 
intersection probability for 
trajectories  of two  ballistic $d$-dimensional random walks 
(such as the effective random walks with step distribution $\sff (\sfw, m)$)
which start at $\sfx$ and $\sfy$. The statement  
\eqref{eq:WD-intersection} is very similar in spirit to that of Lemma~\ref{lem:WDBound}:
 it says that possible weak attraction due to disorder does not destroy such
 asymptotics. 
 
 With \eqref{eq:WD-intersection} at our disposal it is very easy to finish the 
 proof of \eqref{eq:WD-cone-x}. Indeed, it implies that 
 \[ 
  \Var\lb \frac{1}{n^{d-1}}\sum_{\sfx \in B_n }\sfsof{\sfx}\rb 
  \leqs \frac{1}{n^{d/2 - 1 }} ,  
 \]
and since $\calE\lb \sum_{\sfx \in B_n } \sfsof{\sfx} \rb = \abs{ B_n}\eqvs n^{d-1}$, 
the conclusion follows by 
Chebychev inequality  and Borel-Cantelli argument.

\section{Strong disorder} 
\label{sec:Strong}
In this section, we work only 
under Assumption~{\bf A1} of the Introduction, and 
we do not impose any further  assumptions on the 
environment $\lbr \sfV^\omega_\sfx  \rbr$.  The case of traps;  
 $\calQ \lb \sfVo =\infty \rb \in (0,1)$,   is not excluded 
and  we even do not need {\bf A2} or any other restriction   on the size
of the latter probability. 

The environment is always strong in 
two dimensions in the following sense (level {\bf L1} in the language of 
the Introduction):
\begin{thm}
\label{thm:Strong}
Let $d=2$ and $\beta , \lambda >0$. There exists $c = c(\beta ,\lambda ) >0$ such 
that the following holds: 
 Let $\lambda (h ) = \lambda$ (in particular $h\not\in \Knot$). Then, $\calQ$-a.s. 
\begin{equation}  
\label{eq:Strong}
\limsup_{n\to\infty} \frac{1}{n}\log\frac{\Zno (h  )}{\Zn (h )} < -c.
\end{equation}
In particular, $\lambda^\omega  (h ) <\lambda (h ) = \lambda$ 
whenever $\lambda^\omega$ is well defined.
\end{thm}
\begin{remark}
As in~\cite{Lacoin} and, subsequently, \cite{Zygouras-StrongDisorder} 
proving strong disorder in dimension 
$d=3$ is a substantially more delicate task.
\end{remark}
Let us explain Theorem~\ref{thm:Strong}:  By the exponential Markov inequality 
(and Borel-Cantelli) it is sufficient to prove that there 
exist $c^\prime >0$ and  $\alpha >0$ such that 
\begin{equation} 
\label{eq:FracMom}
\calE \lbr 
\lb \frac{\Zno (h )}{\Zn (h ) }\rb^\alpha \rbr  \leq e^{-c^\prime  n} .
\end{equation}
We shall try to establish \eqref{eq:FracMom} with $\alpha\in (0,1 )$. This is 
the fractional moment method of \cite{Lacoin}. It has a transparent logic: 
Since $\calE \lb \Zno (h )\rb = \Zn (h )$, expecting  \eqref{eq:Strong} means that 
$\Zno (h )$ takes excessive exponentially  high values with exponentially 
small probabilities. Taking 
fractional moments  in \eqref{eq:FracMom} amounts to truncating these high values.  
\smallskip

\noindent
\textbf{Reduction to Basic Partition Functions.}
Recall the definition of diamond-confined (basic) partition 
functions $\sff(\sfx, n) = \calE\lb  \sffo (\sfx, n)\rb$, 
$\sft (\sfx, n) = \calE \lb\sfto (\sfx, n)\rb$ and, accordingly, 
$\sff (\sfx )$, $\sft (\sfx ), \dots$ in \eqref{eq:AN-sft-quantities}. 
Since $\lim_{n\to\infty} \sft (n ) = \mu(h )^{-1}$, theorem
target statement \eqref{eq:Strong} would follow from 
\be 
\label{eq:Strong-tf} 
\limsup_{n\to\infty}\frac{1}{n}\log \frac{\sfto (n )}{\sft (n )} = 
\limsup_{n\to\infty}\frac{1}{n}\log \sfto (n ) < 0. 
\ee
In its turn, in view of Teorem~\ref{thm:L3-RepAlphabets}, 
 \eqref{eq:Strong-tf} is routinely implied by the following  statement 
(\eqref{eq:rterm-lim} below): Let $\rrq{N}$ be the partition function of $N$ irreducible steps:
\begin{equation}  
\label{eq:pfSteps}
\rrq{N} \df \sum_{\sfu_1,\cdots, \sfu_N} \sffo (\sfu_1) \sffof{u_1} (\sfu_2 -\sfu_1)
\cdots \sffof{\sfu_{N-1}} (\sfu_N -\sfu_{N-1}) = \sum_{\sfx} \rrq{\sfx, N}.
\end{equation}
Then, $\calQ$-a.s. 
\be
\label{eq:rterm-lim}
\limsup_{N\to\infty}\frac{1}{N}\log \rrq{N} < 0.
\ee
Again by Borel-Cantelli and 
the exponential Markov inequality, \eqref{eq:rterm-lim} would follow as soon 
as we check that for some $\alpha >0$, 
\begin{equation} 
\label{eq:rterm}
\limsup_{N\to\infty} 
\frac{1}{N}
\log\bbE\lb \rrq{N}\rb^\alpha < 0 .
\end{equation}
\smallskip

\noindent
\textbf{Fractional Moments.}
The proof of the fractional moment bound \eqref{eq:rterm} comprises 
several steps.
\smallskip 

\noindent 
\step{1} 
Following \cite{Lacoin}: \eqref{eq:rterm} is verified  once we show that 
there exist $N\in\bbN$ and $\alpha\in (0,1)$ such that
\begin{equation} 
\label{eq:target}
\calE\lbr \sum_x \lb \rrq{x ,N}\rb^\alpha \rbr <1 .
\end{equation}
Indeed, first of all if $a_i\geq 0$ and $\alpha \in (0,1)$, then 
\be
\label{eq:ST-trivial} 
 \lb \sum a_i \rb^\alpha \leq \sum a_i^\alpha .
\ee
Equivalently (setting $p =1/\alpha >1$ and $b_i = a_i^\alpha$), 
$
\sum b_i^p \leq \lb \sum b_i\rb^p .
$
Since 
\[
 \frac{\dd}{\dd b_n} \lb \sum b_i\rb^p \geq p b_n^{p-1} = \frac{\dd}{\dd b_n} \sum b_i^p , 
\]
the latter form of \eqref{eq:ST-trivial} follows by induction.  

We proceed with proving  that \eqref{eq:target} implies \eqref{eq:rterm}. 
Evidently, by \eqref{eq:ST-trivial}, 
\[
  \rrq{N+M} = \sum_\sfx \rrq{N, \sfx }\rrqof{M}{\sfx} \ 
  \Rightarrow \ 
  \lb \rrq{N+M}\rb^\alpha \leq \sum_\sfx \lb \rrq{N, \sfx }\rb^{\alpha} 
  \lb\rrqof{M}{\sfx}\rb^\alpha , 
\]
for any $\alpha \in (0,1)$. Since  $\rrq{N, \sfx }$ and $\rrqof{M}{\sfx}$ are 
independent, and $\rrq{M}$ is translation invariant, it follows that
\[
 \calE \lbr (\rrq{N+ M})^\alpha\rbr \leq 
 \calE\lbr \sum_\sfx \lb \rrq{N, \sfx }\rb^\alpha \rbr
 \calE\lbr  (\rrq{M})^\alpha\rbr  .
\]
Hence \eqref{eq:target}, implies exponential decay of 
$M\mapsto \calE\lbr  (\rrq{M})^\alpha\rbr $. 
\smallskip

\noindent
\step{2} Let $\sfv_h = \sum_\sfx \sfx\sff (\sfx )$; mean displacement under probability
measure $\lbr\sff (\sfx )\rbr$. By Theorem~\ref{thm:L3-RepAlphabets}, the latter 
distribution has exponential tails, and classical moderate deviation results apply. 
For $\sfy\in\bbZ^d$ define the distance from $\sfy$ to the line in the direction of $\sfv_h$; 
 $\dd_h (\sfy ) = \min_a \abs{\sfy - a\sfv_h}$. 
Pick $K$ sufficiently large and $\epsilon$ small, and consider
\[
A_N = \lbr \sfy\in\bbZ^d~:~ 0\leq \sfy\cdot \sfv_h\leq KN\ {\rm and}\ \dd_h (\sfy )
\leq N^{\frac{1}{2} +\epsilon}\rbr.
\]
Recall that $\rra{\sfx , N}$ is the distribution of the end point of the $N$-step 
random walk with $\lbr\sff (\sfx )\rbr$ being the one step distribution. 
With a slight abuse of notation, we can consider $\rra{N}$ as a distribution on
the set of all $N$-step trajectories of this random walk:
\be
\label{eq:ST:trajectories}
\rra{N}(\sfx_1, \dots , \sfx_N ) = \sff (\sfx_1 )\sff (\sfx_2 - \sfx_1 )\dots 
\sff (\sfx_N -\sfx_{N-1}) .
\ee
By classical (Gaussian)  moderate deviation estimates, there exists $c>0$ such that
\be 
\label{eq:ST-MD}
\sum_{\sfx\not\in A_N}\rra{\sfx, N}^\alpha \leq {\rm e}^{- c\alpha N^{2\epsilon}}
\ {\rm and}\ 
\rra{N}\lb \lbr\sfx_1, \dots , \sfx_N\rbr\not\subset A_N\rb \leq {\rm e}^{- cN^{2\epsilon}} .
\ee
Furthermore, with another slight abuse
of notation we can consider $\sfr_N (\cdot )$ as the distribution on the family 
of all $N$-concatenations $\gamma = \gamma_1\circ\gamma_2\circ\dots\circ\gamma_N$ 
of irreducible paths $\gamma_i\in\calF$. In this way, 
\[
 \rra{N} (\gamma) = \prod_1^N\Wdn{h , \lambda}(\gamma_i ) .
\]
Recall from \eqref{eq:L3-diamond} that irreducible paths $\gamma_i$ 
satisfy the following diamond confinement condition:
 If $\sfx_{i-1}, \sfx_i$ are the end points of $\gamma_i$, then 
 $\gamma_i\subset D (\sfx_{i-1} , \sfx_i )$. 
\begin{ex}
\label{ex:ST-diamonds}
Prove the following generalization of the second of 
\eqref{eq:ST-MD} (see Figure~\ref{fig:Aset}): There exists $c>0$ such that
\be 
\label{eq:ST-MD-diamond}
\rra{N}\lb \cup_i D(\sfx_{i-1} , \sfx_i )\not\subset A_N\rb 
\leq {\rm e}^{- cN^{2\epsilon}}.
\ee
\end{ex}
\begin{figure}[h]
\sidecaption
\scalebox{.25}{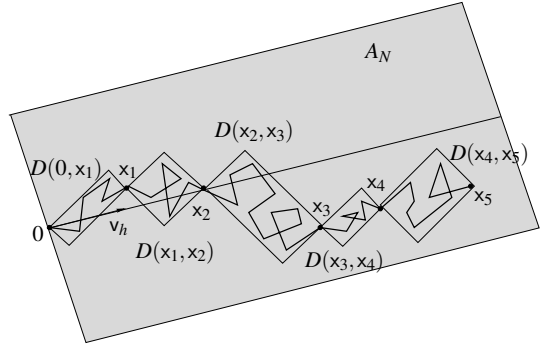}
%
%
\caption{Example: $N=5$. The path $\lb 0, \sfx_1, \dots , \sfx_5\rb$ of the effective random
walk, and the union of diamond shapes $\cup_i D(\sfx_{i-1} , \sfx_i ) \subset A_N$.}
\label{fig:Aset}       
\end{figure}

Since any concatenation $\gamma = \gamma_1\circ\gamma_2\circ\dots\circ\gamma_N$ 
of irreducible paths $\gamma_i\in\calF$ satisfies 
$\gamma\subseteq \cup_i D(\sfx_{i-1} , \sfx_i )$, we readily infer that under $\rra{N}$
typical annealed paths stay inside $A_N$, 
\be 
\label{eq:ST-MD-paths}
\rra{N}\lb \gamma\not\subset  A_N\rb \leq {\rm e}^{- cN^{2\epsilon}} \ 
\Rightarrow \ 
\calE\lbr \lb \rrq{N}\lb \gamma\not\subset  A_N\rb\rb^\alpha\rbr  
\leq {\rm e}^{- \alpha cN^{2\epsilon}}, 
\ee
 for any $\alpha\in (0,1)$ (by Jensen's inequality). 

Since $\calE \lbr \bigl(\rrq{x ,N}\bigr)^\alpha\rbr \leq r_{N,x}^\alpha$, we, 
in view of the first of \eqref{eq:ST-MD}, may restrict summation in \eqref{eq:target}
to $\sfx\in A_N$. In view of \eqref{eq:ST-MD-paths}, it would be enough to check 
that 
\be 
\label{eq:ST-target1}
\sum_{x\in A_N} \calE\lbr \lb \rrq{x, N}(\gamma\subset A_N )\rb^\alpha\rbr 
\leq 
\abs{A_N}\calE \lbr \lb \rrq{N}\lb \gamma\subset A_N\rb\rb^\alpha\rbr  < 1. 
\ee
The first inequality above  is a crude over-counting, but for $d=2$ it will do. 
\smallskip 

\noindent
\step{3} We, therefore, concentrate on proving the second inequality in \eqref{eq:ST-target1}.  
At this stage, we shall modify the distribution of the environment 
inside  $A_N$ in the 
following way: 
The modified law of the 
environment, which we shall denote  $\calQ_\delta$ is still product 
and, for every $\sfx\in A_N$,
\[
 \frac{\dd \calQ_\delta}{\dd \calQ}
\lb \sfVo_\sfx  \rb 
\df e^{\delta \psi \lb \sfVo_\sfx \rb - g (\delta )},  \quad 
\text{where\quad  $e^{g (\delta )} =\log  \calE\lb  e^{\delta \psi \lb \sfVo \rb }\rb $, }
\]
and $\psi$ is a bounded non-decreasing function on $\bbR_+$, 
for instance $\psi (v ) =v\wedge 1$. 

The  annealed potential in the modified environment is  $\psi_{ \beta}( \ell , \delta ) = 
-\log\calE_\delta\lb {\rm e}^{-\beta\ell \sfVo}\rb$. Note that for any $\ell\geq 1$, 
\be 
\label{eq:ST-phi-delta}
 \frac{\dd \phi_{ \beta} (\ell , \delta )}{\dd\delta}\big|_{\delta=0} 
 =\calE\lb \psi (\sfVo )\rb \calE ({\rm e}^{-\ell\beta\sfVo} ) 
 - \calE\lb \psi (\sfVo ){\rm e}^{-\ell\beta\sfVo}\rb > 0 .
\ee
Indeed, since $\psi$ is non-decreasing and ${\rm e}^{-\ell v}$ 
is decreasing, the last inequality follows from positive association of 
one-dimensional probability measures, as described in the beginning of 
Subsection~\ref{sub:TH-annealed}. 

By \eqref{eq:ST-MD-paths} we can ignore paths which do not stay 
inside $A_N$. Thus, \eqref{eq:ST-phi-delta} implies: 
There exists $c_\psi >0$ such that for all $\delta$
sufficiently small, 
\be 
\label{eq:ST-rN-delta}
\calE_\delta\lb \sfr^\omega \rb = \sum_{\sfx ,n}\calE_\delta\lb \sffo (\sfx, n )
\rb \leq 1-c_\psi \delta \ \Rightarrow\ 
\calE_\delta( \rrq{N }) \leq {\rm e}^{-Nc_\psi\delta }.
\ee
From H\"{o}lder's inequality,
\be 
\label{eq:ST-Holder}
\begin{split}
 \calE \lbr \lb \rrq{N}\lb \gamma\subset A_N\rb \rb^\alpha \rbr 
 &\leq \Bigl(  \calE_\delta\lbr \bigl(  
\frac{\dd \calQ}{\dd \calQ_\delta}\bigr)^{1/(1-\alpha ) }\rbr \Bigr)^{1-\alpha}
\bigl( \calE_\delta \lbr  \rrq{N}(\gamma\subset A_N ) \rbr \bigr)^\alpha \\
&\stackrel{\eqref{eq:ST-rN-delta}}{\leq} 
\Bigl(  \calE_\delta\lbr \bigl(  
\frac{\dd \calQ}{\dd \calQ_\delta}\bigr)^{1/(1-\alpha ) }\rbr \Bigr)^{1-\alpha} 
{\rm e}^{-Nc_\psi\alpha \delta }.
\end{split}
\ee
Now, the first term on the right hand side of \eqref{eq:ST-Holder} is 
\be 
\label{eq:ST-change-back}
\begin{split}
\calE_\delta\lbr \bigl(  
\frac{\dd \calQ}{\dd \calQ_\delta}\bigr)^{1/(1-\alpha ) }\rbr 
&= 
\calE\lbr 
\frac{\dd \calQ_\delta}{\dd \calQ}
\bigl(  
\frac{\dd \calQ}{\dd \calQ_\delta}\bigr)^{1/(1-\alpha ) }\rbr 
\\
&= 
\lb 
\calE \lbr {\rm e}^{\frac{\alpha}{1-\alpha} (g (\delta ) - \delta\psi (\sfVo ))}\rbr
\rb^{\abs{A_N}} .
\end{split}
\ee
However, the first order terms in $\delta$ cancel:
\begin{equation} 
\label{eq:quadratic}
\frac{\dd}{\dd \delta}\Big|_{\delta =0}
\log \calE \lbr {\rm e}^{\frac{\alpha}{1-\alpha} (g (\delta ) - \delta\psi (\sfVo ))}
\rbr
 = \frac{\alpha}{1-\alpha}\lb g^\prime (0) - \calE\lbr \psi (\sfVo )\rbr \rb = 0.
\end{equation}
Consequently, by the second order expansion, there exists $\nu_\psi <\infty$, 
such that 
\be 
\label{eq:ST-q-bound}
\Bigl(  \calE_\delta\lbr \bigl(  
\frac{\dd \calQ}{\dd \calQ_\delta}\bigr)^{1/(1-\alpha ) }\rbr \Bigr)^{1-\alpha} 
\leq {\rm e}^{\frac{\nu_\psi}{1-\alpha}\delta^2 \abs{A_N}} .
\ee
A substitution to \eqref{eq:ST-Holder} yields: 
\be 
\label{eq:ST-final}
\calE \lbr \lb \rrq{N}\lb \gamma\subset A_N\rb \rb^\alpha \rbr \leq 
{\rm e}^{-Nc_\psi\alpha \delta + \frac{\nu_\psi}{1-\alpha}\delta^2 \abs{A_N}} . 
\ee
We are now ready to specify the choice of $\delta = \delta_N$: 
In two dimensions; $d=2$, the cardinality $\abs{A_N}\leqs N^{\frac{3}{2}+\epsilon}$. 
Hence, \eqref{eq:ST-target1} follows whenever we choose 
\[
\frac{C\log N }{N} \ll  \delta_N\ll N^{-\frac{1}{2}-\epsilon} 
\]
with $C = C (\alpha )$ being sufficiently large.  

\section{Appendix: Geometry of convex bodies and Large deviations.}
\label{Asec:geometry}
In  these notes we shall restrict attention to finite dimensional spaces $\bbR^d$. The principal
references  are \cite{BoF,Rock,Schneider} for convex geometry and \cite{Varadhan,DS,DZ,FrankLD} 
for large deviations.
\subsection{Convexity and duality.}
{\bf Convex functions.} 
A function $\phi :\bbR^d\to \bbR\cup\infty $ is said to be convex if 
\[
 \phi \lb t \sfx +(1-t )\sfy\rb \leq t\phi (\sfx ) + (1-t)\phi (\sfy ) , 
\]
for all $\sfx , \sfy\in \bbR^d$ and all $t\in [0,1]$. 
\begin{remark}
 \label{rem:A-rem-conv}
Note that by definition we permit $\infty$ values, but not $-\infty$ values. 
\end{remark}
Alternatively, $\phi :\bbR^d\to \bbR\cup\infty $ is convex if the set
\be 
\label{eq:A-epi}
{\rm epi}(\phi ) = \lbr (\sfx , \alpha )~:~ \phi (\sfx )\leq \alpha\rbr \subset \bbR^{d+1}
\ee
is convex. We shall work with convex lower-semicontinuous functions: $\phi$ is 
lower-semicontinuous if for any $\sfx$ and any sequence $\sfx_n$ 
converging to $\sfx$, 
\[
 \phi (\sfx )\leq \lim_{n\to\infty}\phi (\sfx_n ) .
\]
Alternatively, $\phi$ is lower-semicontinuous if the set ${\rm epi}(\phi )$ 
is closed. 

A basic example of  a convex and lower-semicontinuous (actually continuous ) function is
an affine function
\[
 \frl_{a,h} (\sfx )  =a + h\cdot \sfx .
\]
 \begin{thm}
\label{thm:A-conv-lsc} 
The following are equivalent:

\noindent
(a) $\phi:\bbR^d\to \bbR\cup\infty $ is convex and lower-semicontinuous.

\noindent
(b) ${\rm epi} (\phi )$ is convex and closed. 

\noindent
(c) $\phi$ can be recovered from its affine minorants: 
\be
\label{eq:A-phi-l}
\phi (x) = \sup_{a,h}\lbr \frl_{a, h} (\sfx )~:~ \frl_{a,h }\leq \phi\rbr .
\ee
(d) ${\rm epi} (\phi )$ is the intersection of closed half-spaces
\be
\label{eq:A-epiphi-l}
{\rm epi} (\phi ) = \bigcap_{\frl_{a,h}\leq \phi } {\rm epi}\lb \frl_{a,h}\rb .
\ee
\end{thm}
\begin{definition}
 \label{def:A-subdif} $\phi$ is sub-differentiable at $\sfx$ if there exists $\frl_{a,h} \leq \phi$
such that $\phi (\sfx ) = \frl_{a,h} (\sfx )$. In the latter case we write $h\in\partial\phi (\sfx )$.
\end{definition}
\begin{ex}
\label{ex:A-conv1}
Check that $h\in\partial \phi (\sfx )$  iff $\phi (\sfx ) <\infty$ and 
\be 
\label{eq:A-subdif}
\phi (\sfy ) - \phi (\sfx ) \geq h\cdot (\sfy - \sfx )
\ee
for any $\sfy\in\bbR^d$. 
 \end{ex}
Convex functions on $\bbR^d$  are always sub-differentiable at interiour points of their effective
domains.  
In general, sets  $\partial\phi (\sfx )$ may be empty, may be singletons 
 or they may contain continuum of different slopes $h$.  
\begin{ex}
\label{ex:A-conv2}
Find an example with $\partial \phi (\sfx ) = \emptyset$. 
Prove that in general $\partial\phi (\sfx )$ is closed and convex. 
Check that a convex $\phi$ is differentiable at $\sfx$ with $\nabla\phi (\sfx ) = h$ iff
$\partial\phi (\sfx ) = \lbr h\rbr$.
\end{ex}
\begin{definition}
\label{def:A-LF-transform}
Let $\phi :\bbR^d\mapsto \bbR\cup\infty$. The Legendre-Fenchel transform, or the convex 
conjugate, of $\phi$ is 
\be 
\label{eq:A-LF-transform}
\phi^* (\sfx ) = \sup_{h}\lbr h\cdot \sfx -\phi (h )\rbr .
\ee 
\end{definition}
By construction, $\phi^*$ is always convex and lower-semicontinuous: Indeed, 
\[
 {\rm epi} (\phi^* ) = \bigcap_{h}{\rm epi}\lb \frl_{-\phi (h ) , h}\rb ,
\]
which is obviously closed and convex.
\smallskip 

\noindent
{\bf Duality.}
Let $\phi $ be a convex and lower-semicontinuous function. let us say
that $h$ and $\sfx$ are a pair of conjugate points if $\sfx\in\partial\phi (h )$. 
\begin{thm}
 \label{thm:A-duality}
If $\phi$ is convex and lower-semicontinuous, then 
 $\lb\lb\phi\rb^*\rb^* = \phi$. In the latter situation, the notion of conjugate points is symmetric, 
namely the following are equivalent:
\be
\label{eq:A-duality}
\sfx \in\partial\phi (h )\Leftrightarrow h\in\partial\phi^* (\sfx )
\Leftrightarrow  \phi (h ) + \phi^* (\sfx ) = h\cdot\sfx .
\ee
Let $h, \sfx $ be a pair of conjugate points. Then strict convexity of $\phi$ at $h$ is 
equivalent to differentiability of $\phi^*$ at $\sfx$. Namely, 
\be 
\label{eq:A-duality-diff}
\forall\, g\neq h,\ \phi (g )-\phi (h ) > \sfx\cdot (g-h )\ 
\Leftrightarrow\ \nabla\phi^* (\sfx ) = h .
\ee
\end{thm}
{\bf Support and  Minkowski functions.}
Let $\bfK\subset \bbR^d$ be a compact convex set with non-empty interiour around the 
origin  $0\in {\rm int}\lb \bfK\rb$. 
\begin{definition}
 \label{def:A-convex-functionals} The function
\be 
\label{eq:A-K-char}
\chi_{\bfK} ( h ) = 
\begin{cases}
 0, &\text{if $h\in\bfK$}\\
\infty , &\text{otherwise}
\end{cases}
\ee
is called the characteristic function of $\bfK$. 

The function
\be 
\label{eq:A-K-sup}
\tau_{\bfK} (\sfx ) = \sup_{h\in \bfK} h\cdot \sfx = \max_{h\in\pbfK} h\cdot \sfx 
\ee
is called the support function of $\bfK$. 

The function 
\be 
\label{eq:A-K-Mink}
\alpha_{\bfK} ( h ) = \inf\lbr r >0 ~:~ h \in r \bfK\rbr
\ee
is called the Minkowski function of $\bfK$. 
\end{definition}
As it will become apparent below  functions $\chi_{\bfK}, \tau_{\bfK}$ and $\phi_{\bfK}$ are 
convex and lower-semicontinuous. 
\smallskip 

\noindent
{\bf Duality relation between $\chi_{\bfK}$ and $\tau_{\bfK}$.}  
The characteristic function $\chi_{\bfK}$ is convex and 
lower-semicontinuous since  ${\rm epi}\lb \chi_{\bfK}\rb = \bfK\times [0, \infty)$. 
The support function 
 $\tau_{\bfK}$ is the supremum 
of linear functions. As such it is homogeneous of order one. 
Also, \eqref{eq:A-K-sup} could be recorded in the form which makes $\tau_{\bfK}$ 
to be  the convex conjugate:
\be 
\label{eq:A-tau-chi}
 \tau_{\bfK} = \chi_{\bfK}^* \ \text{and, by Theorem~\ref{thm:A-duality},}\ \chi_{\bfK} = \tau_{\bfK}^*.
\ee
Since $\tau_{\bfK}$ is homogeneous, the latter reads as
\be
\label{eq:A-K-tau}
\bfK = \bigcap_{\frn\in\bbS^{d-1}}\lbr h\, :\, h\cdot\frn \leq \tau_{\bfK} (\frn )\rbr .
\ee
By \eqref{eq:A-K-char} if  $h\in\pbfK$ then  $\sfx\in\partial\chi_{\bfK} (h)$ 
if and only if $\sfx$ is in the 
direction of the outward normal to a hyperplane which touches $\pbfK$ at $h$. Thus, 
$\partial \chi_{\bfK} (h)$ is always a closed convex cone (which can be just a semi-line).

 Other way around, 
$\partial\tau_{\bfK} (\sfx )$ contains all boundary points 
$h\in\pbfK$, such that $\sfx$ is the direction of the outward normal to a supporting hyperplane 
at $h$. In particular, 
\be 
\label{eq:A-homog-sub}
 \partial\tau_{\bfK} (r \sfx ) = \partial\tau_{\bfK} (\sfx ) , 
\ee
and $\partial\tau_{\bfK} ( \sfx ) \subset\pbfK$ is a closed convex facet (which can be just one
point). 

As a consequence: $\tau_{\bfK}$ is differentiable at $\sfx\neq 0$ iff the supporting hyperplane
with the outward normal direction of $\sfx$ touches $\pbfK$ at exactly one point $h$. In 
particular $\tau_{\bfK}$ is differentiable at any $\sfy\neq 0$ iff $\pbfK$ is strictly convex. 
\smallskip 

\noindent 
{\bf Polarity relation between $\tau_{\bfK}$ and $\alpha_{\bfK}$.} 
As for $\alpha_{\bfK} ( h )$ 
the assumptions $\bfK$ is bounded and  $0\in {\rm int}\lb \bfK\rb$ imply 
 that for any $h\neq 0$, the value $\alpha_{\bfK} (h)$ is positive and finite. 
 Consequently,
\[
 \frac{h}{\alpha_{\bfK} ( h )}\in\pbfK \ {\rm and}\  \exists\, \sfx\neq 0\ \text{such that}\ 
\tau_{\bfK} (\sfx ) = \frac{\sfx\cdot h}{ \alpha_{\bfK} ( h )} .
\]
On the other hand, again since $\frac{h}{\alpha_{\bfK} ( h )}\in\pbfK$, 
\[
 \tau_{\bfK} \lb \sfy \rb \geq \frac{\sfy\cdot h}{ \alpha_{\bfK} ( h )} , 
\]
for any $\sfy\in\bbR^d$.  Since $\tau_{\bfK}$ is homogeneous of order one, we, therefore,  conclude:
\be 
\label{eq:A-dual-support}
\alpha_{\bfK} ( h ) = \max\lbr h\cdot \sfy ~:~ \tau_{\bfK} (\sfy )\leq 1\rbr .
\ee
In other words, $\alpha_{\bfK}$ is the support function of the closed convex set:
\be
\label{eq:A-polar} 
\bfK^* = \lbr \sfy ~:~ \tau_{\bfK} (\sfy )\leq 1\rbr .
\ee
In particular, $\alpha_{\bfK}$ is convex and lower-semicontinuous. 
\begin{ex}
 \label{ex:AK-polar}
For any $\sfx, h\neq 0$, $\sfx\cdot h \leq  \tau_{\bfK} (\sfx ) \alpha_{\bfK} ( h )$. Furthermore, 
\be 
\label{eq:AK-polar}
\frac{\sfx\cdot h}{\tau_{\bfK} (\sfx ) \alpha_{\bfK} ( h )} =1\, \Leftrightarrow\, 
\frac{\sfx}{\tau_{\bfK} (\sfx )}\in\partial \alpha_{\bfK} ( h )\, \Leftrightarrow\, 
\frac{h}{\alpha_{\bfK} ( h )} \in \partial\tau_{\bfK} (\sfx ) .
\ee
\end{ex}
Actually, by homogeneity it would be enough to establish \eqref{eq:AK-polar} for 
$\sfx\in\partial \bfK^*$ (equivalently  $\tau_{\bfK} (\sfx )=1$) and 
$h\in\pbfK$ (equivalently $\alpha_K (h )=1$). In the latter case, let us say that $\sfx\in \partial \bfK^*$
and $h\in\pbfK$ are in polar relation  if $\sfx\cdot h =1$. 
\begin{ex}
 \label{ex:AK-polar-points}
Let $\sfx , h$ be in polar relation. Then 
\be 
\label{eq:AK-polar-points}
\text{$\pbfK$ is strictly convex (smooth) at $h$ iff  $\partial \bfK^*$ is smooth (strictly convex) at $\sfx$} .
\ee

\end{ex}
\begin{remark}
\label{rem:A-generalities} Most of the above notions can be defined and 
effectively studied in much 
more generality than we do. In particular, one can go beyond assumptions of finite dimensions 
and non-empty interiour. 
\end{remark}

\subsection{Curves and surfaces.} 
\label{ssec:A-Surfaces}
Let $M$ be a smooth $(d-1)$-dimensional surface (without boundary) 
embedded in $\Rd$. For $u\in M$ let
$\frn (u)\in\bbS^{d-1}$ be the normal direction at $u$. $T_uM$ is the tangent space to 
$M$ at $u$. Thus $\frn $ is a map $\frn :\, M\mapsto \bbS^{d-1}$.
It is called the Gauss map, and its differential $\dd \frn_u$ is called 
 the Weingarten map. 
 Since $T_u M =  T_{\sfn (u ) }\bbS^{d-1}$, we may consider 
$\dd \frn_u $ as  a linear map on the tangent space 
$T_uM$. 
\begin{ex}
\label{ex:L1.10} Check that $\dd \frn_u $ is self-adjoint (with respect to the usual Euclidean
scalar product on $\Rd$).  Hence, the eigenvalues $\chi_1, \dots , \chi_{d-1}$ of $\dd\frn_u$
are real, and the corresponding normalized eigenvectors $\frv_1 ,\dots , \frv_{d-1}$ form an 
orthonormal basis of $T_u M$. 
\end{ex}
\begin{definition}
\label{def:L1-curvatures} 
Eigenvalues $\chi_1 , \dots , \chi_{d-1} \geq 0$
 of $\dd\frn$ are called principal curvatures of $M$ at $u$. The normalized 
eigenvectors $\frv_1 ,\dots , \frv_{d-1}$ are called directions of principal curvature. 
The product $\prod_\ell  \chi_\ell$ is called the Gaussian curvature. 
\end{definition}
Assume that $M$ is locally given by a level set of a smooth function $\lambda (\cdot )$, such that 
$\nabla\lambda (u) \neq 0$. That is, in a neighbourhood of $u$; 
$v\in M\, \leftrightarrow\, \lambda (v )= 0$.
Then, 
\[
 \frn_v = \frac{\nabla\lambda (v )}{\abs{\nabla \lambda (v )}} . 
\]
Define $\Sigma_u = {\rm Hess}[\lambda ](u )$. 
\begin{ex}
 \label{ex:A-Sigma-map}
Check that for $g\in T_uM$, 
\be 
\label{eq:A-Sigma-map} 
\dd\frn_u g\cdot g = \frac{1}{\abs{\nabla \lambda (v )}}\Sigma_u g\cdot g .
\ee
\end{ex}
{\bf Convex surfaces.} Let now $M = \pbfK$, and $\bfK$ is a 
bounded 
convex body with non-empty interiour. 
In the sequel we shall assume that the boundary $\partial {\mathbf K}$ is smooth 
(at least $\sfC_2$). Let $\tau = \tau_{\bfK}$ be the support function of ${\mathbf K}$. 
Whenever defined the Hessian $\Xi_\sfx = {\rm Hess}[\tau ] (\sfx )$ has a natural interpretation 
in terms of the curvatures of ${\mathbf K}$ at $h = \nabla\tau (\sfx )$. 
We are following Chapter~2.5 in \cite{Schneider}. 
\smallskip 

\noindent 
{\bf General case.}
 We assume that 
$M$ is smooth and that 
the Gaussian curvature of $M$ is uniformly non-zero. 
In particular, $M=\pbfK$ is strictly convex, and, by duality relations, 
its support function $\tau$ is differentiable. 
In the sequel  $\frn (h)$ is understood as the exteriour
normal to $M$ at $h$. 

Recall 
that 
for every $\sfx\neq 0$, the gradient
$\nabla\tau (\sfx)  = h_\sfx \in \pbfK$, and could be characterized by $h_\sfx\cdot \sfx = \tau (\sfx)$. 
Consequently, $\nabla \tau (r\sfx ) = \nabla \tau (\sfx)$. This is a homogeneity relation. It readily
implies the following: Let  $\Xi_\sfx$ be the Hessian of $\tau$ at $\sfx$. Then, 
\be  
\label{eq:homog-tau}
 \Xi_\sfx \sfx =0 .
\ee
Let $h\in \pbfK$ and let $\frn = \frn ( h)$ be the normal direction to $\pbfK$ at $h$. 
Then 
\[
 \nabla \tau (\frn (h )) = h.
\]
 In other words, the  restriction of $\nabla\tau$ to $\bbS^{d-1}$ is precisely the 
inverse of the Gauss map $\frn$. Hence the restriction $\hat\Xi_{\frn (h)}$ of $\Xi_\frn$ to 
$T_h \pbfK$ is 
the inverse of the Weingarten map $\dd\frn_h $. 
\begin{definition} 
\label{def:L1-radii-curvature}
Let $\frn\in\bbS^{d-1}$ and $h =\nabla\tau (\frn )$. 
Eigenvalues $r_\ell = 1/\chi_\ell$  
 of $\hat \Xi_\frn$ are called principal radii of curvature of $M=\pbfK$ at $h$.
\end{definition}
{\bf Example: Smooth convex curves.} 
Let $\frn_\theta   = (\cos\theta ,\sin\theta )$. 
Radius of curvature $r (\theta ) = 1/\chi (\theta )$ 
 of the boundary $\partial{\mathbf K}$ at a point 
$h_\theta  = \nabla  \tau (\frn_\theta )$ is given by 
\[
 r (\theta )  = \frac{\dd^2}{\dd\theta^2} \tau (\theta ) + \tau (\theta ) , 
\]
where we put $\tau (\theta ) = \tau (\frn_\theta )$. 
Indeed, $\sfv_\theta  = \frn^\prime_\theta  = (-\sin\theta , \cos\theta )$ is the unit 
spanning vector of $T_{h_\theta}\pbfK$. Note that $\frac{\dd}{\dd \theta}\sfv_\theta = - \frn_\theta$. 
 Hence, 
\[
 \frac{\dd^2}{\dd\theta^2}\tau (\theta ) = 
\frac{\dd}{\dd \theta}\lb \nabla\tau (\frn_\theta )\cdot \sfv_\theta\rb = 
\Xi_{\frn_\theta}\sfv_\theta\cdot \sfv_\theta - \nabla\tau (\frn_\theta )\cdot\frn_\theta  = 
\Xi_{\frn_\theta}\sfv_\theta\cdot \sfv_\theta - \tau (\frn_\theta )
\]
{\bf Second order expansion.} Let $\lb \frv_1 ,\dots ,\frv_{d-1} , \frn (h)\rb$ be orthonormal 
coordinate frame, where $\lb \frv_1 ,\dots ,\frv_{d-1}\rb$ is a basis of 
$T_h M$. Consider matrix elements  $\Xi_\frn(i,j )$ in this coordinates. Then the homogeneity 
relation \eqref{eq:homog-tau} applied at $\sfx= \frn (h)$ yields:
\be 
\label{eq:Xi_x}
 \Xi_\frn (\ell , d ) = 0\ \text{for all $\ell= 1, \dots, d$} .
\ee
Which means that as a quadratic form $\Xi_\frn$ satisfies:
\be
\label{eq:Xi_x-projection}
\Xi_\frn \sfu\cdot \sfw = \Xi_\frn \pi_h \sfu\cdot \pi_h \sfv ,
\ee
 where $\pi_h$ is the orthogonal projection on $T_h\pbfK$. Furthermore, 
\begin{ex}
 \label{ex:A-L1.8}
Check that the Hessian $\Xi_\sfx\df {\rm Hess}_\sfx \tau  = \frac{1}{\abs{\sfx}}\Xi_\frn$, where 
$\frn = \frn_\sfx \in \bbS^{d-1}$ is the unit vector in the  direction
of $\sfx$. 
\end{ex}
Consequently, second order expansion takes the form: For any  $\sfx\neq 0$ and 
$t\in (0,1)$
\be 
\label{eq:A-2dOrder-tau}
 \tl\lb t\sfx +\sfv\rb +\tl \lb (1-t )\sfx -\sfv \rb - \tl (\sfx ) = 
\frac{\Xi_\frn \sfv  \cdot \sfv}{2t (1-t) \abs{\sfx }} + 
\smo{\frac{\abs{\sfv}^2}{\abs{\sfx }}} .
\ee
Recording 
this in the 
(orthonormal) basis of principal curvatures, we deduce 
the following Corollary: 
\begin{corollary}
\label{cor:QExpansion-tau}
 Let $\sfx\in\Rd$; $\frn = \frn_\sfx = \frac{\sfx}{\abs{\sfx}}\in\bbS^{d-1}$, 
and let $h = \nabla\tau (\sfx )$. Consider the 
orthogonal frame $(\frv_1, \dots ,\frv_{d-1} , \frn )$, where $\frv_\ell$-s are the directions
of principal curvature of $\pbfK$ at $h$. Then, for any $t\in (0,1)$ and for 
any $y_1, \dots y_{d-1}$, 
\be 
\label{eq:QExpansion-tau}
\begin{split}
&\tau\lb t\sfx + \sum_{\ell = 1}^{d-1} y_\ell\frv_\ell \rb 
+ \tau\lb (1-t)\sfx - \sum_{\ell =1}^{d-1} y_\ell\frv_\ell \rb- \tau (\sfx) \\
&\qquad = \sum_{\ell = 1}^{d-1} \frac{y_\ell^2}{2t (1-t )\abs{\sfx} \chi_\ell  }
\quad + 
\smo{\frac{\sum y_\ell^2}{\abs{\sfx}} } .
\end{split}
\ee
\end{corollary}
{\bf Strict triangle inequality.} 
If principal curvatures of $\pbfK$ are  uniformly bounded or, equivalently, 
if quadratic forms $\Xi_\frn $ are uniformly (in $\frn\in \bbS^{d-1}$) positive definite,  
then there exists a constant
$c>0$ such that 
\be 
\label{eq:A-strict-t-tau}
\tau (\sfx ) +\tau (\sfy ) - \tau (\sfx +\sfy ) 
\geq c\lb \abs{\sfx } +\abs{\sfy } - \abs{\sfx +\sfy }\rb .
\ee
In order to prove \eqref{eq:A-strict-t-tau} note, first of all, 
that since for any $\sfz\neq 0$, $\nabla\tau (\sfz ) = \nabla\tau (\frn_\sfz )$, 
one can rewrite
the left han side of \eqref{eq:A-strict-t-tau} as 
\[
 \tau (\sfx ) +\tau (\sfy ) - \tau (\sfx +\sfy ) = \sfx \cdot\lb \nabla\tau (\frn_{\sfx} ) 
 - \nabla\tau (\frn_{\sfx +\sfy} )\rb
  + \sfy\cdot \lb \nabla\tau (\frn_{\sfy} ) - \nabla\tau (\frn_{\sfx +\sfy} )\rb .
\]
Simiraly, 
\[
 \abs{\sfx } +\abs{\sfy } - \abs{\sfx +\sfy } = \sfx\cdot \lb \frn_{\sfx} - \frn_{\sfx +\sfy}\rb 
 + \sfy\cdot \lb \frn_{\sfy} - \frn_{\sfx +\sfy}\rb . 
\]
Therefore, \eqref{eq:A-strict-t-tau} will follow if we show that for any two unit vectors
$\frn, \frm\in\bbS^{d-1}$, 
\be 
\label{eq:A-strict-t-tau-1} 
\frn\cdot \lb \nabla\tau (\frn ) - \nabla\tau (\frm )\rb \geq 
c\, \frn\cdot\lb \frn -\frm\rb .
\ee
Set $\Delta = \frn -\frm$. 
Since $\frn\cdot\lb \frn -\frm\rb\leqs \abs{\Delta}^2$, and since we are not pushing for the 
optimal value of $c$ in \eqref{eq:A-strict-t-tau}, it would be enough to consider second order
expansion in $\abs{\Delta}$.

To this end define $\gamma_t = \frm +t\Delta$  
and $h_t = \nabla\tau (\gamma_t )$.  
Then, 
\be 
\label{eq:A-2d-order} 
\begin{split}
 \frn\cdot \lb \nabla\tau (\frn ) - \nabla\tau (\frm )\rb
 &= 
 \frn\cdot\int_0^1 \frac{\dd}{\dd t}\nabla\tau (\gamma_t )\dd t = 
 \int_0^1 \Xi_{\gamma_t } \frn\cdot \Delta\dd t \\
 & = 
 \int_0^1 \Xi_{\gamma_t } \pi_{h_t}\frn\cdot \pi_{h_t} \Delta\dd t .
 \end{split}
\ee
The last equality above is \eqref{eq:Xi_x-projection}. By construction 
$\gamma_t$ is orthogonal to $T_{h_t}\partial \bfK$.  Hence the projection 
\[
 \pi_{h_t }\frn = \frn -\frac{\frn\cdot\gamma_t}{\abs{\gamma_t}^2}\gamma_t 
  = (1-t)\Delta + \frac{(1-t)\Delta \cdot\gamma_t}{\abs{\gamma_t}^2}\gamma_t
=  (1-t)\Delta + 
 \smo{\abs{\Delta}} .
\]
On the other hand $\pi_{h_t }\Delta = \Delta + \smo{\abs{\Delta}}$. Hence, up to higher 
order terms in $\abs{\Delta}$, 
\be 
\label{A-qf-bound}
\int_0^1 \Xi_{\gamma_t } 
\pi_{h_t}\frn\cdot \pi_{h_t} \Delta\dd t \geq \frac{1}{2} 
\min_{ h\in\partial \bfK}\min_{\ell}
r_\ell (h )
\abs{\frn -\frm}^2 , 
\ee 
and \eqref{eq:A-strict-t-tau} follows.

\subsection{Large deviations.}
\label{Asec:LD}
{\bf The setup.}
Although the framework of the theory is much more general we shall restrict attention
to probabilities on finite-dimensional spaces. 
Let $\lbr \bbP_n\rbr$ be a family of probability measures on $\bbR^d$.
\begin{definition}
 \label{def:LD-RateFunction} A function $J : \bbR^d \mapsto [0,\infty]$ is said to be a 
rate function if it is proper (${\rm Dom}(J)\df\lbr\sfx: J (\sfx )<\infty\rbr\neq \emptyset$) 
and if  it has compact level sets. In particular  rate functions are
always lower-semicontinuous. 
\end{definition}
\begin{definition}
 \label{def:LD-LDP}
A family $\lbr \bbP_n\rbr$ satisfies large deviation principle with rate function $J$ 
(and speed $n$) if:
\begin{description}
 \item[Upper Bound] For every closed $F\subseteq\bbR^d$
\be 
\label{eq:LD-UB}
\limsup_{n\to\infty}\frac{1}{n}\log \bbP_n \lb F \rb \leq - \inf_{\sfx\in F} J(\sfx ) .
\ee
 \item[Lower  Bound] For every open  $O\subseteq\bbR^d$
\be 
\label{eq:LD-LB}
\limsup_{n\to\infty}\frac{1}{n}\log \bbP_n \lb O \rb \geq  - \inf_{\sfx\in O} J(\sfx ) .
\ee
\end{description}
\end{definition}
There is an alternative formulation of the lower bound:
\begin{ex}
 \label{ex:LD-LB-x}
Check that \eqref{eq:LD-LB} is equivalent to: For every $\sfx\in\bbR^d$ the family 
$\lbr \bbP_n\rbr$ satisfies the LD lower bound at $\sfx$, that is for any open 
neighbourhood $O$ of $\sfx$, 
\be 
\label{eq:LD-LB-x}
\limsup_{n\to\infty}\frac{1}{n}\log \bbP_n \lb O \rb \geq  - J(\sfx ) .
\ee
\end{ex}
All the measures we shall work with are exponentially tight:
\begin{definition}
 \label{def:LD-ET}
A family $\lbr \bbP_n\rbr$ is exponentially tight if for any $R$ one can find a compact 
subset $K_R$ of $\bbR^d$ such that
\be 
\label{eq:LD-ET}
\limsup_{n\to\infty}\frac{1}{n}\log \bbP_n \lb K_R^{\sfc} \rb \leq -R .
\ee
\end{definition}
If exponential tightness is checked then one needs derive upper bounds only for all compact sets:
\begin{ex}
 \label{ex:LD.3}
Check that if $\lbr \bbP_n\rbr$ is exponentially tight and it satisfies 
\eqref{eq:LD-LB} for all open sets and  \eqref{eq:LD-UB} for all 
{\em compact} sets, then it satisfies LDP. 
\end{ex}
In particular, $\lbr \bbP_n\rbr$ satisfies an upper large deviation bound with $J$ if 

\noindent
(a) It is exponentially tight. 

\noindent
(b) For every $\sfx\in \bbR^d$, the family $\lbr \bbP_n\rbr$ satisfies the following upper 
large deviation bound at $\sfx$: 
\be 
\label{eq:LD-UB-x-general}
\lim_{\delta\downarrow 0} \limsup_{n\to\infty}\frac{1}{n}\log \bbP_n \lb \left| \frac{\sfX}{n}
-\sfx\right| \leq \delta 
\rb 
\leq -  J(\sfx ) .
\ee
We shall mostly work with measures on $\frac{1}{n}\bbZ^d$ 
which are generated by scaled random variables $\frac{1}{n}\sfX $, for instance when 
 $\sfX = \sfX (\gamma )$ is the spatial extension of a polymer or the 
end point of a self-interacting random walk. In the latter case we shall modify the 
notion \eqref{eq:LD-UB-x-general} of point-wise LD upper bound as follows: 
\begin{definition}
 \label{def:LD-x}
A family $\lbr \bbP_n\rbr$ of probability measures on $\bbZ^d$ satisfies an 
upper LD bound at $\sfx\in\bbR^d$ if for any $R >0$
\be 
\label{eq:LD-UB-x}
\limsup_{n\to\infty}\frac{1}{n}\log \bbP_n \lb \sfX = \lfloor n\sfx\rfloor \rb 
\leq -  J(\sfx )\wedge R .
\ee
\end{definition}
At a first glance constant $R$ in \eqref{eq:LD-UB-x} does not seem to contribute
to the statement. However, checking and formulating things this way may be convenient. 
\begin{ex}
 \label{ex:LD.4}
Let  $J$ be a rate function. Check that if $\lbr \bbP_n\rbr$ is exponentially tight, 
if the lower bound \eqref{eq:LD-LB-x} is satisfied, 
 and if \eqref{eq:LD-UB-x} is satisfied, for any $R \in [0,\infty )$,   
 {\em uniformly} on 
compact subsets of $\bbR^d$, then $\lbr \bbP_n\rbr$ satisfies the LD principle  in the sense of 
Definition~\ref{def:LD-LDP}.  
\end{ex}
{\bf Log-moment generating functions and convex conjugates.}
Frequently quests after LD rate functions stick to the following pattern: Assume 
that the (limiting) log-moment generating function
\be
\label{eq:LD-lmgf}
\lambda (h ) = \lim_{n\to\infty}\frac{1}{n}\log\bbE_n {\rm e}^{h\cdot \sfX }
\ee 
is well defined (and not identically $\infty$) for all $h\in\bbR^d$. 
\begin{ex}
 \label{ex:LD.5}
Check that if $\lambda (\cdot )$ in \eqref{eq:LD-lmgf} is indeed defined, then
it is convex and lower-semicontinuous. 
\end{ex}
Consider the Legendre-Fenchel transform $I$ of $\lambda$
\be
\label{eq:LD-I}
I (\sfx ) = \sup_h\lbr h\cdot \sfx - \lambda (h)\rbr .
\ee
Here is one of the basic general LD results:
\begin{thm}
 \label{thm:LD-convex}
Assume that $\lambda$ in \eqref{eq:LD-lmgf} is well defined and proper. 
\begin{description}
 \item[Upper Bound.] For any $\sfx\in\bbR^d$ the family $\lbr \bbP_n\rbr$
satisfies upper LD bound \eqref{eq:LD-UB-x} with $I$  at $\sfx$. 
\item[Lower Bound.] If, in addition, $I$ is sub-differential and 
strictly convex at $\sfx$, then 
$\lbr \bbP_n\rbr$
satisfies a lower  LD bound at $\sfx$ with $I$   in \eqref{eq:LD-LB-x}.
\end{description}
\end{thm}
Sub-differentiability and strict convexity  over finite-dimensional spaces
are studied in great generality (e.g. low-dimensional effective domains, behaviour 
at the boundary of relative interiours etc )  and detail 
\cite{Rock}. 

Lower LD bounds with 
$I$ generically {\em do not} hold. In particular true LD rate functions $J$ are generically 
{\em non-convex}. However, $I = J$ in many important examples such as sums of i.i.d.-s 
and   Markov chains. Moreover, $I=J$ for most of polymer models with purely 
attractive or repulsive interactions. A notable 
exception is provided by one-dimensional polymers with repulsion \cite{GdH95,Ko94}(which we 
do not discuss here). 
Under minor additional integrability conditions 
the relation between $I$ and $J$ could be described as follows:

Let $\phi$ be a function on $\bbR_+$ with a super-linear growth at $\infty$: 
\[
 \lim_{t\to\infty}\frac{\phi (t )}{t} = \infty .
\]
\begin{lem}
\label{lem:LD-IJ}
Assume that $\lbr \bbP_n\rbr$ satisfies LDP with rate function $J$, 
and assume that 
\be 
\label{eq:J-growth}
\limsup\frac{1}{n}\bbE_n {\rm e}^{n \phi \lb\frac{\abs{\sfX}}{n}\rb  } <\infty .
\ee
Then~\cite{DS}  $\lambda (\cdot )$ in \eqref{eq:LD-lmgf} is defined and equals to
\[
 \lambda (h) = \sup_{\sfx}\lbr \sfx\cdot h - J (\sfx ) \rbr .
\]
Consequently, $I$ is the convex lower-semicontinuous  envelop of $J$, that is 
\[
 I (\sfx ) = \sup\lbr \frl_{a, h}(\sfx)~:~ \frl_{a, h}\leq J\rbr.
\]
In particular, if $J$ convex, then $I=J$. 
\end{lem}

\end{document}